\documentclass[11pt,reqno]{amsart}    
\def\MR#1{}
\usepackage[margin=37mm]{geometry}
\usepackage{amsfonts,amsmath,amssymb,enumerate}    
\usepackage{amsthm, bm}    
\usepackage{mathrsfs}
\usepackage{setspace}
\usepackage{comment}

\usepackage{pgfplots}
\usepackage{tikz-3dplot}
\usetikzlibrary{arrows,matrix}
\usetikzlibrary{decorations.markings,shapes.geometric,shapes.misc,cd} 
\usetikzlibrary{3d}
\usetikzlibrary {fadings}
\tikzfading[name=fade out,
inner color=transparent!100,
outer color=transparent!0]

\pgfplotsset{compat=1.17}
\usepgfplotslibrary{fillbetween}

\usepgfplotslibrary{external}
\def\noexternalize{yes}
\ifdefined\noexternalize\else
\tikzset{external/up to date check = md5}
\fi

\usepackage{mathtools}    
\usepackage{hyperref}

\usepackage[capitalise]{cleveref}
\crefname{subsection}{\S\!}{\S\!}
\Crefname{subsection}{Section}{Sections}
\crefname{section}{Section}{Sections}
\crefname{subappendix}{\S\!}{\S\!}
\Crefname{subappendix}{Section}{Sections}
\usepackage{upgreek}
\usepackage{simpler-wick}

\usepackage[normalem]{ulem}

\usepackage[hyperref,doi,url=false,
backref,
giveninits=true,
backend=biber,
style=alphabetic,
maxbibnames=99]{biblatex}
\bibliography{bibliography.bib}

\AtEveryBibitem{
 \clearlist{address}
 \clearfield{date}
 \clearfield{isbn}
 \clearfield{issn}
 \clearlist{location}
 \clearfield{month}
 \clearfield{series}
 
 \ifentrytype{book}{}{
  \clearlist{publisher}
  \clearname{editor}
 }
 \ifentrytype{misc}{}{
  \clearfield{eprint}
 }
}

\crefname{equation}{}{}
\Crefname{equation}{}{}

\crefname{thm}{Theorem}{Theorems}

\DeclareFontEncoding{LS1}{}{}
\DeclareFontSubstitution{LS1}{stix}{m}{n}
\DeclareSymbolFont{symbols2}{LS1}{stixfrak}{m}{n}
\DeclareMathSymbol{\typecolon}{\mathbin}{symbols2}{"25}

\newcommand{\bilin}[2]{\hspace{.5mm}\kappa\hspace{-.5mm}\bigl( #1|#2\bigr)}

\newcommand{\pairing}[2]{\left< #1 \mid #2 \right>}

\theoremstyle{plain}
\newtheorem{thm}{Theorem}
\newtheorem{lem}[thm]{Lemma}
\newtheorem{prop}[thm]{Proposition}
\newtheorem{cor}[thm]{Corollary}

\newtheorem*{prop*}{Proposition}
\theoremstyle{definition}
\newtheorem{defn}[thm]{Definition}

\theoremstyle{remark}
\newtheorem{remarkx}[thm]{Remark}
\newenvironment{rem}
  {\pushQED{\qed}\remarkx}
  {\popQED\endremarkx}

\newtheorem*{rem*}{Remark}
\newtheorem*{ack}{Acknowledgements}

\newcommand{\be}{\begin{equation}}    
\newcommand{\ee}{\end{equation}}    
\newcommand{\beu}{\begin{equation*}}    
\newcommand{\eeu}{\end{equation*}}    
\newcommand{\bea}{\begin{eqnarray}}    
\newcommand{\eea}{\end{eqnarray}}    
\newcommand{\beaa}{\begin{eqnarray*}}    
\newcommand{\eeaa}{\end{eqnarray*}}    
\newcommand{\bmx}{\begin{pmatrix}}    
\newcommand{\emx}{\end{pmatrix}}

\newcommand{\ol}{\overline}    

\newcommand{\del}{\partial}    
\newcommand{\g}{{\mathfrak g}}
\renewcommand{\a}{{\mathfrak a}}
\renewcommand{\b}{{\mathfrak b}}

\newcommand{\gh}{{\widehat \g}}

\newcommand{\mc}{\mathcal}

\newcommand{\half}{\frac{1}{2}}

\newcommand{\nn}{\nonumber}

\newcommand{\8}{{\infty}}

\newcommand{\ZZ}{{\mathbb Z}}

\newcommand{\CC}{{\mathbb C}}
\newcommand{\AAA}{{\mathbf A}}
\renewcommand{\AA}{{\mathbb A}}

\newcommand{\id}{{\textup{id}}}    
    
\newcommand{\wh}{\widehat}

\newcommand{\Qc}[2]{\pi_i \circ Q \circ \pi_j}

\newcommand{\upby}[1]{\hspace{0mm}\smash{\mathsf s}^{#1}\hspace{-.25mm}}

\newcommand{\down}{\hspace{.0mm}\smash{\mathsf s}^{-1}\hspace{-.1mm}}
\newcommand{\downby}[1]{\hspace{0mm}\smash{\mathsf s}^{-#1}\hspace{-.25mm}}

\newcommand{\A}{\mathcal A}

\newcommand{\goi}[2]{=}    
\newcommand{\Hom}{\mathrm{Hom}}

\newcommand{\Homcont}{\mathrm{Hom}^{\mathrm{cts}}}

\newcommand{\qisom}{\simeq}

\newcommand{\btp}{\begin{tikzpicture}[baseline=0pt,scale=0.9,line width=0.25pt]}    
\newcommand{\etp}{\end{tikzpicture}}

\newcommand{\pt}{\mathrm{pt.}}

\newcommand{\wt}{\widetilde}

\DeclareMathOperator{\res}{res}

\DeclareMathOperator{\Spec}{Spec}

\DeclareMathOperator{\gr}{gr}

\renewcommand{\O}{\mc O}

\newcommand{\ox}{\mathbin\otimes}

\newcommand{\bul}{\bullet}
\newcommand{\invlim}{\varprojlim}
\newcommand{\directlim}{\varinjlim}

\newcommand{\into}{\hookrightarrow}

\newcommand{\op}{\mathrm{op}}

\newcommand{\F}{\mc F}

\DeclareMathOperator{\Sym}{Sym}

\newcommand{\C}{\mathcal C}

\makeatletter
\newcommand{\extp}{\@ifnextchar^\@extp{\@extp^{\,}}}
\def\@extp^#1{\mathop{\bigwedge\nolimits^{\!#1}}}
\makeatother
\makeatletter
\newcommand{\hextp}{\@ifnextchar^\@hextp{\@hextp^{\,}}}

\def\@hextp^#1{\mathop{\wh{\bigwedge\nolimits^{\!#1}}}}
\makeatother

\newcommand{\catname}[1]{\mathbf{#1}}

\newcommand{\Set}{\catname{Set}}

\newcommand{\CAlg}{\catname{CAlg}}
\newcommand{\LieAlg}{\catname{LieAlg}}
\newcommand{\dgCAlg}{\catname{dgCAlg}}
\newcommand{\dgVect}{\catname{dgVect}}
\newcommand{\grVect}{\catname{grVect}}

\newcommand{\dgLie}{\catname{dgLieAlg}}

\newcommand{\Top}{\catname{Top}}

\newcommand{\Open}{\catname{Open}}

\newcommand{\CCh}{\catname{CCh}}

\newcommand{\dd}{\mathrm{d}}

\newcommand{\Disc}{\mathrm{Disc}}
\newcommand{\Discp}{\mathrm{Disc}^\times}

\DeclareMathOperator{\Tot}{Tot}

\DeclareMathOperator{\Cone}{Cone}
\DeclareMathOperator{\Cocone}{Cocone}
\DeclareMathOperator{\hoker}{hoker}

\newcommand{\ww}{\mathbf w}
\newcommand{\zz}{\mathbf z}
\DeclareMathOperator{\Th}{Th}

\newcommand{\Deltas}{\Delta}
\newcommand{\emb}{\mathrm{emb}}
\newcommand{\sca}{{\Pi}}

\newcommand{\dth}{\dd_{\Th}}
\newcommand{\Pre}[1]{{\Deltas\!\downarrow\! #1}}

\newcommand{\Rect}{\mathrm{Rect}_2}
\newcommand{\Rectn}{\mathrm{Rect}_2(N)}

\newcommand{\DD}{\mathrm{PDisc}_2}
\newcommand{\DDx}{\DD^\times}

\newcommand{\DDp}[1]{\DD(#1)}
\newcommand{\DDpx}[1]{\DDx(#1)}

\newcommand{\ADpx}[1]{A_{\DDpx #1}}

\newcommand{\ARN}{A_{\Rectn}}

\newcommand{\gDp}[1]{\g_{\DDp #1}}

\newcommand{\gDpx}[1]{\g_{\DDpx #1}}

\newcommand{\gglob}{\g_{\mathrm{Global}}}
\newcommand{\gN}{\gglob}

\newcommand{\gloc}{\g_{\mathrm{PDiscs}}}
\newcommand{\glocx}{\g_{\mathrm{PDiscs}^\times}}

\newcommand{\galoc}{\wt\gloc}
\newcommand{\galocx}{\wt\glocx}

\newcommand{\AU}{A_U}
\newcommand{\AD}{A_{\DD}}
\newcommand{\ADx}{A_{\DDx}}

\newcommand{\gD}{\g_{\DD}}

\newcommand{\gDx}{\g_{\DDx}}

\newcommand{\CAlgemb}{\CAlg^\emb}

\newcommand{\wop}{a}
\newcommand{\zop}{b}

\newcommand{\iglob}{I_{\mathrm{Global}}}
\newcommand{\iloc}{I_{\mathrm{Discs}}}
\newcommand{\pglob}{P_{\mathrm{Global}}}
\newcommand{\ploc}{P_{\mathrm{Discs}}}

\newcommand{\hglob}{h_{\mathrm{Global}}}
\newcommand{\hlocx}{h_{\mathrm{Discs^\times}}}
\newcommand{\hoffdiag}{h_{\mathrm{off diag.}}}

\newcommand{\itot}{I}
\newcommand{\ptot}{P}
\newcommand{\RRR}{B}
\newcommand{\gp}{\g_+}
\newcommand{\gm}{\g_-}

\newcommand{\xx}{\mathbf x}
\newcommand{\ttop}{\times^\Top}
\newcommand{\Eta}{{\mathrm E}}

\newcommand{\dfn}[1]{\emph{\uline{#1}}}

\newcommand{\Linfinity}{{\texorpdfstring{$L_\8\,$}{L-infinity}}}
\newcommand{\Ip}{i_+}
\renewcommand{\Im}{i_-}
\newcommand{\oxC}{\ox}

\def\coords{
\coordinate (z=0)  at (0, 1,0)  ;  
\coordinate (w=0)  at (0,-1,0)  ;  
\coordinate (z=8)  at ( 1,0,0)  ;  
\coordinate (w=8)  at (-1,0,0)  ;  
\coordinate (z=u)  at (0,0, 1)  ;  
\coordinate (w=U)  at (0,0,-1)  ;  
\coordinate (p1p1) at (0,0, 0)  ;
\path (z=0) + (w=U) coordinate (0U);
\path (z=0) + (w=8) coordinate (08);
\path (z=u) + (w=0) coordinate (u0);
\path (z=u) + (w=8) coordinate (u8);
\path (z=8) + (w=0) coordinate (80);
\path (z=8) + (w=U) coordinate (8U);

\path (barycentric cs:{0U}=1,{z=0}=1) coordinate ({0U}{z=0});
\path (barycentric cs:{08}=1,{z=0}=1) coordinate ({08}{z=0});
\path (barycentric cs:{u0}=1,{z=u}=1) coordinate ({u0}{z=u});
\path (barycentric cs:{u8}=1,{z=u}=1) coordinate ({u8}{z=u});
\path (barycentric cs:{80}=1,{z=8}=1) coordinate ({80}{z=8});
\path (barycentric cs:{8U}=1,{z=8}=1) coordinate ({8U}{z=8});

\path (barycentric cs:{u0}=1,{w=0}=1) coordinate ({u0}{w=0});
\path (barycentric cs:{80}=1,{w=0}=1) coordinate ({80}{w=0});
\path (barycentric cs:{0U}=1,{w=U}=1) coordinate ({0U}{w=U});
\path (barycentric cs:{8U}=1,{w=U}=1) coordinate ({8U}{w=U});
\path (barycentric cs:{08}=1,{w=8}=1) coordinate ({08}{w=8});
\path (barycentric cs:{u8}=1,{w=8}=1) coordinate ({u8}{w=8});

\path (barycentric cs:{0U}=1,{p1p1}=1) coordinate ({0U}{p1p1});
\path (barycentric cs:{08}=1,{p1p1}=1) coordinate ({08}{p1p1});
\path (barycentric cs:{u0}=1,{p1p1}=1) coordinate ({u0}{p1p1});
\path (barycentric cs:{u8}=1,{p1p1}=1) coordinate ({u8}{p1p1});
\path (barycentric cs:{80}=1,{p1p1}=1) coordinate ({80}{p1p1});
\path (barycentric cs:{8U}=1,{p1p1}=1) coordinate ({8U}{p1p1});

\path (barycentric cs:{z=0}=1,{p1p1}=1.4) coordinate ({z=0}{p1p1});
\path (barycentric cs:{z=u}=1,{p1p1}=1.4) coordinate ({z=u}{p1p1});
\path (barycentric cs:{z=8}=1,{p1p1}=1.4) coordinate ({z=8}{p1p1});
\path (barycentric cs:{w=0}=1,{p1p1}=1.4) coordinate ({w=0}{p1p1});
\path (barycentric cs:{w=U}=1,{p1p1}=1.4) coordinate ({w=U}{p1p1});
\path (barycentric cs:{w=8}=1,{p1p1}=1.4) coordinate ({w=8}{p1p1});

\path (barycentric cs:{0U}=1,{z=0}=1.3,p1p1=1) coordinate ({0U}{z=0}{p1p1});
\path (barycentric cs:{08}=1,{z=0}=1.3,p1p1=1) coordinate ({08}{z=0}{p1p1});
\path (barycentric cs:{u0}=1,{z=u}=1.3,p1p1=1) coordinate ({u0}{z=u}{p1p1});
\path (barycentric cs:{u8}=1,{z=u}=1.3,p1p1=1) coordinate ({u8}{z=u}{p1p1});
\path (barycentric cs:{80}=1,{z=8}=1.3,p1p1=1) coordinate ({80}{z=8}{p1p1});
\path (barycentric cs:{8U}=1,{z=8}=1.3,p1p1=1) coordinate ({8U}{z=8}{p1p1});
\path (barycentric cs:{u0}=1,{w=0}=1.3,p1p1=1) coordinate ({u0}{w=0}{p1p1});
\path (barycentric cs:{80}=1,{w=0}=1.3,p1p1=1) coordinate ({80}{w=0}{p1p1});
\path (barycentric cs:{0U}=1,{w=U}=1.3,p1p1=1) coordinate ({0U}{w=U}{p1p1});
\path (barycentric cs:{8U}=1,{w=U}=1.3,p1p1=1) coordinate ({8U}{w=U}{p1p1});
\path (barycentric cs:{08}=1,{w=8}=1.3,p1p1=1) coordinate ({08}{w=8}{p1p1});
\path (barycentric cs:{u8}=1,{w=8}=1.3,p1p1=1) coordinate ({u8}{w=8}{p1p1});
}

\def\globalpicthree
{
\coords
\node (Nz=0)   at (z=0)  {$z=z_1$} ;
\node (Nw=0)   at (w=0)  {$w=w_1$} ;
\node (Nz=8)   at (z=8)  {$z=z_3$} ;
\node (Nw=8)   at (w=8)  {$w=w_3$} ;
\node (Nz=u)   at (z=u)  {$z=z_2$} ;
\node (Nw=U)   at (w=U)  {$w=w_2$} ;
\node (Np1p1)  at (p1p1) {$\Eta$}  ;
\node (N0U) at (0U) {$(w_2,z_1)$}  ;
\node (N08) at (08) {$(w_3,z_1)$}  ;
\node (Nu0) at (u0) {$(w_1,z_2)$}  ;
\node (Nu8) at (u8) {$(w_2,z_2)$}  ;
\node (N80) at (80) {$(w_1,z_3)$}  ;
\node (N8U) at (8U) {$(w_2,z_3)$}  ;
\draw[fill=blue,opacity=.1] (p1p1) -- (z=0) -- (0U) -- cycle;
\draw[fill=blue,opacity=.1] (p1p1) -- (z=0) -- (08) -- cycle;
\draw[fill=blue,opacity=.1] (p1p1) -- (z=u) -- (u0) -- cycle;
\draw[fill=blue,opacity=.1] (p1p1) -- (z=u) -- (u8) -- cycle;
\draw[fill=blue,opacity=.1] (p1p1) -- (z=8) -- (80) -- cycle;
\draw[fill=blue,opacity=.1] (p1p1) -- (z=8) -- (8U) -- cycle;
\draw[fill=blue,opacity=.1] (p1p1) -- (w=0) -- (u0) -- cycle;
\draw[fill=blue,opacity=.1] (p1p1) -- (w=0) -- (80) -- cycle;
\draw[fill=blue,opacity=.1] (p1p1) -- (w=U) -- (0U) -- cycle;
\draw[fill=blue,opacity=.1] (p1p1) -- (w=U) -- (8U) -- cycle;
\draw[fill=blue,opacity=.1] (p1p1) -- (w=8) -- (08) -- cycle;
\draw[fill=blue,opacity=.1] (p1p1) -- (w=8) -- (u8) -- cycle;
}

\def\Rgpicsthree{
$ $\vspace{-0cm}\be
\begin{tikzpicture}[baseline=-50]
\begin{scope}[scale=.3,local bounding box = Global]
\fill[gray!15] (1,-1) rectangle (16,11); 
\draw (0,0) -- (17,0) node[right] (z0) {$z=z_1$};
\draw (0,10) -- (17,10) node[right] (z8) {$z=z_3$};
\draw (0,3) -- (17,3) node[right] (zu) {$z=z_2$};
\draw (2,12) -- (2,-2) node[below=-.1cm] (wU) {$w=w_1$};
\draw[] (9,12) -- (9,-2) node[below] (w1) {$w=w_2$};
\draw (15,12) -- (15,-2) node[below] (w1) {$w=w_3$};
\draw[fill=black]  (2,3) circle (.2);\draw[fill=black]  (2,10) circle (.2);\draw[fill=white]  (2,0) circle (.2);
\draw[fill=white]  (9,3) circle (.2);\draw[fill=black]  (9,10) circle (.2);\draw[fill=black]  (9,0) circle (.2);
\draw[fill=black]  (15,3) circle (.2);
\draw[fill=white]  (15,10) circle (.2);
\draw[fill=black]  (15,0) circle (.2);
\end{scope}
\end{tikzpicture}\qquad
\begin{tikzpicture}
\begin{scope}[xshift=0,scale=2,
x = {(0+90:1cm)}, 
y = {(120+90:1cm)}, 
z = {(240+90:1cm)},
local bounding box=G]
\globalpicthree
\end{scope}
\end{tikzpicture}
\nn\ee
}

\def\squarecoords
{
\coordinate (0U)  at (1,1,0)  ;  
\coordinate (z=0) at (0,1,0)  ;  
\coordinate (0wi) at (-1,1,0) ;  
\coordinate (w=U)  at (1,0,0)  ;  
\coordinate (w=wi)at (-1,0,0) ;  
\coordinate (z=u) at (0,-1,0)  ;  
\coordinate (uwi) at (-1,-1,0) ;   
\coordinate (p1p1) at (0,0,0) ;   
\coordinate (uU) at (1,-1,0);
\coordinate ({uU  }{z=u }{p1p1}) at (barycentric cs:p1p1=1,{uU}=1,{z=u}=1);
\coordinate ({uU  }{w=U }{p1p1}) at (barycentric cs:p1p1=1,{w=U}=1,{uU}=1);
\coordinate ({0wi }{w=wi}{p1p1}) at (barycentric cs:p1p1=1,{0wi}=1,{w=wi}=1);
\coordinate ({0wi }{z=0 }{p1p1}) at (barycentric cs:p1p1=1,{z=0}=1,{0wi}=1);
\coordinate ({0U  }{z=0 }{p1p1}) at (barycentric cs:p1p1=1,{0U}=1,{z=0}=1);
\coordinate ({0U  }{w=U }{p1p1}) at (barycentric cs:p1p1=1,{w=U}=1,{0U}=1);
\coordinate ({uwi }{z=u }{p1p1}) at (barycentric cs:p1p1=1,{z=u}=1,{uwi}=1);
\coordinate ({uwi }{w=wi}{p1p1}) at (barycentric cs:p1p1=1,{uwi}=1,{w=wi}=1);
}

\def\drawthree
{
\draw[fill=blue,fill opacity=.1] (p1p1) -- (0wi) -- (w=wi) --cycle; 
\draw[fill=blue,fill opacity=.1] (p1p1) -- (z=0) -- (0wi)  --cycle; 
\draw[fill=blue,fill opacity=.1] (p1p1) -- (0U)  -- (z=0)  --cycle; 
\draw[fill=blue,fill opacity=.1] (p1p1) -- (w=U) -- (0U)   --cycle; 
\draw[fill=blue,fill opacity=.1] (p1p1) -- (z=u) -- (uwi)  --cycle; 
\draw[fill=blue,fill opacity=.1] (p1p1) -- (uwi) -- (w=wi) --cycle;
}

\def\threefourpiconetwothree
{\squarecoords
\drawthree 
\begin{scope}[every node/.style={font=\scriptsize}]
\node[inner sep=1.5,label= left:{$(w_1,z_3)$}]   (N0U)   at (0U)  {};
\node[inner sep=1.5,label= below:{$(z=z_3)$}]   (Nz=0)  at (z=0) {};
\node[inner sep=1.5,label= right:{$(w_2,z_3)$}] (N0wi)  at (0wi) {};
\node[inner sep=1.5,label= left:{$(w=w_1)$}]   (Nw=U)  at (w=U) {};
\node[inner sep=1.5,label= right:{$(w=w_2)$}] (Nw=wi) at (w=wi){};
\node[inner sep=1.5,label= above:{$(z=z_1)$}]   (Nz=u)  at (z=u) {};
\node[inner sep=1.5,label= right:{$(w_2,z_1)$}] (Nuwi)  at (uwi) {};
\node[inner sep=1.5,label= above left:{$\Eta$}] (Np1p1) at (p1p1) {};
\end{scope}  
}

\def\threefourpicikl
{\squarecoords
\drawthree 
\begin{scope}[every node/.style={font=\scriptsize}]
\node[inner sep=1.5,label= left:{$(w_i,z_\ell)$}]   (N0U)   at (0U)  {};
\node[inner sep=1.5,label= below:{$(z=z_\ell)$}]   (Nz=0)  at (z=0) {};
\node[inner sep=1.5,label= right:{$(w_k,z_\ell)$}] (N0wi)  at (0wi) {};
\node[inner sep=1.5,label= left:{$(w=w_i)$}]   (Nw=U)  at (w=U) {};
\node[inner sep=1.5,label= right:{$(w=w_k)$}] (Nw=wi) at (w=wi){};
\node[inner sep=1.5,label= above:{$(z=z_i)$}]   (Nz=u)  at (z=u) {};
\node[inner sep=1.5,label= right:{$(w_k,z_i)$}] (Nuwi)  at (uwi) {};
\node[inner sep=1.5,label= above left:{$\Eta$}] (Np1p1) at (p1p1) {};
\end{scope}  
}

\def\fourfourpicijkl
{
\squarecoords
\draw[fill=blue,fill opacity=.1] (p1p1) -- (w=U) -- (uU)   --cycle; 
\draw[fill=blue,fill opacity=.1] (p1p1) -- (z=u) -- (uU)  --cycle; 
\drawthree 
\begin{scope}[every node/.style={font=\scriptsize}]
\node[inner sep=1.5,label= left:{$(w_i,z_j)$}]   (NuU)   at (uU)  {};
\node[inner sep=1.5,label= left:{$(w_i,z_\ell)$}]   (N0U)   at (0U)  {};
\node[inner sep=1.5,label= below:{$(z=z_\ell)$}]   (Nz=0)  at (z=0) {};
\node[inner sep=1.5,label= right:{$(w_k,z_\ell)$}] (N0wi)  at (0wi) {};
\node[inner sep=1.5,label= left:{$(w=w_i)$}]   (Nw=U)  at (w=U) {};
\node[inner sep=1.5,label= right:{$(w=w_k)$}] (Nw=wi) at (w=wi){};
\node[inner sep=1.5,label= above:{$(z=z_j)$}]   (Nz=u)  at (z=u) {};
\node[inner sep=1.5,label= right:{$(w_k,z_j)$}] (Nuwi)  at (uwi) {};
\node[inner sep=1.5,label= above left:{$\Eta$}] (Np1p1) at (p1p1) {};
\end{scope}  
}

\def\fourfourpiccdef
{
\squarecoords
\draw[fill=blue,fill opacity=.1] (p1p1) -- (w=U) -- (uU)   --cycle; 
\draw[fill=blue,fill opacity=.1] (p1p1) -- (z=u) -- (uU)  --cycle; 
\drawthree 
\begin{scope}[every node/.style={font=\scriptsize}]
\node[inner sep=1.5,label= left:{$(c,d)$}]   (NuU)   at (uU)  {};
\node[inner sep=1.5,label= left:{$(c,f)$}]   (N0U)   at (0U)  {};
\node[inner sep=1.5,label= below:{$(z=f)$}]   (Nz=0)  at (z=0) {};
\node[inner sep=1.5,label= right:{$(e,f)$}] (N0wi)  at (0wi) {};
\node[inner sep=1.5,label= left:{$(w=c)$}]   (Nw=U)  at (w=U) {};
\node[inner sep=1.5,label= right:{$(w=e)$}] (Nw=wi) at (w=wi){};
\node[inner sep=1.5,label= above:{$(z=d)$}]   (Nz=u)  at (z=u) {};
\node[inner sep=1.5,label= right:{$(e,d)$}] (Nuwi)  at (uwi) {};
\node[inner sep=1.5,label= above left:{$\Eta$}] (Np1p1) at (p1p1) {};
\end{scope}  
}

\DeclareMathOperator{\Flag}{Flag}
\author{Luigi Alfonsi and Charles Young}
\address{
Department of Physics, Astronomy and Mathematics, University of Hertfordshire, College Lane, Hatfield AL10 9AB, UK.}  \email{l.alfonsi@herts.ac.uk, c.young8@herts.ac.uk}
\date{\today}

\begin{document}

\ifdefined\noexternalize\else\tikzexternaldisable\fi
\title[Higher current algebras and homotopy Manin triples]{Higher current algebras, homotopy Manin triples, and a rectilinear adelic complex}

\begin{abstract}
The notion of a Manin triple of Lie algebras admits a generalization, to dg Lie algebras, in which various properties are required to hold only up to homotopy. 

This paper introduces two classes of examples of such \emph{homotopy Manin triples}. These examples are associated to analogs in complex dimension two of, respectively, the punctured formal 1-disc, and the complex plane with multiple punctures. The dg Lie algebras which appear include certain \emph{higher current algebras} in the sense of Faonte, Hennion and Kapranov \cite{FHK}. 

We work in a ringed space we call \emph{rectilinear space}, and one of the tools we introduce is a model of the derived sections of its structure sheaf, whose construction is in the spirit of the adelic complexes for schemes due to Parshin and Beilinson. 

\end{abstract}


\maketitle
\setcounter{tocdepth}{1}
\tableofcontents


\section{Introduction}
\subsection{Higher current algebras and homotopy Manin triples}
Manin triples are fundamental objects in integrable systems and quantum groups. 
Recall that a Manin triple $(\a,\pairing --,\a_+,\a_-)$ is a Lie algebra $\a$ with a symmetric nondegenerate invariant bilinear form, and two isotropic Lie subalgebras $\a_\pm \subset \a$ such that $\a=_\CC \a_+\oplus \a_-$ as vector spaces \cite{DrinfeldQuantumGroups}.
One important class of examples arises when when $\a$ is a current algebra: namely, one may take
\be \a = \g \ox \CC((z)), \quad \a_+ = \g \ox \CC[[z]], \quad \a_- = \g \ox z^{-1} \CC[z^{-1}] \label{yangman}\ee
and $\pairing{x \ox f(z)}{y\ox g(z)} = \bilin x y \res_z f(z)g(z)$, with $\g$ any finite-dimensional simple Lie algebra over $\CC$ and $\bilin --$ its standard bilinear form. (See, e.g., \cite{ChariPressley,BBTbook}.) 

Here the commutative algebras $\CC[[z]]$, of formal Taylor series, and $\CC((z))$, of formal Laurent series, can be seen as algebras of functions 
\be \CC[[z]] \cong \Gamma(\Disc_1,\hat\O), \qquad \CC((z)) \cong \Gamma(\Discp_1,\hat\O),\nn\ee
on, respectively, the formal disc $\Disc_1$, 
and the punctured formal disc $\Discp_1:=\Disc_1 \setminus\{\pt\}$, in complex dimension one. 

It is natural to try to extend this to higher dimensions
, but an apparent obstacle arises: in dimension $n\geq 2$, the structure sheaf $\hat\O$ admits no more sections over the punctured disc $\Discp_n = \Disc_n\setminus\{\pt\}$ than it does over the disc $\Disc_n$ itself: $\Gamma(\Discp_n,\hat\O) = \Gamma(\Disc_n,\hat\O) \cong \CC[[w,z]]$. So the would-be higher current algebra $\g\ox \Gamma(\Discp_n,\hat\O)$ seems to be ``missing all the  negative modes''. 
However, Faonte, Hennion and Kapranov \cite{FHK} -- and see also \cite{KapranovConformal,Kapranov,GWHigherKM} -- make the following observation: one may replace these algebras of functions by their derived analogs, $R\Gamma(\Disc_2,\hat\O)$ and $R\Gamma(\Discp_2,\hat\O)$. By definition $R\Gamma(\Discp_2,\hat\O)$ is a cochain complex whose cohomology computes the sheaf cohomology of $\hat\O$, as we recall in \cref{sec: sheaf cohomology} below. And it is well known that the structure sheaf $\hat\O$ has higher cohomology when $n\geq 2$. In this paper we focus exclusively on the case of dimension $n=2$, where one has
\be H^\bul(\Discp_2,\hat\O) \cong \begin{cases} \CC[[w,z]] & \bul = 0 \\ w^{-1} z^{-1} \CC[w^{-1},z^{-1}] & \bul = 1 \\ 0 & \text{otherwise.}
\end{cases}
\nn\ee
One sees that natural candidates for the ``missing'' negative modes reemerge in the first cohomology.
Moreover, as stressed in \cite{FHK}, the derived sections form more than just a cochain complex: $R\Gamma(\Discp_2,\hat\O)$ is naturally a  differential graded (dg) commutative algebra (unique up to quasi-isomorphism). Thus, one obtains the dg Lie algebra 
\be \g \ox R\Gamma(\Discp_2,\hat\O), \nn\ee
which it is natural to call a \dfn{higher current algebra}. 

\bigskip

The first goal of this paper is to give an analog of the Manin triple \cref{yangman} above, in complex dimension two, following this philosophy. To do so, we first need to clarify what it should mean to give a Manin triple in the differential graded setting. We give our definition in \cref{sec: Manin triples}: roughly speaking, it says that a \dfn{homotopy Manin triple (in dg Lie algebras)} is a triple of dg Lie algebras and maps of dg Lie algebras between them
\be \a_+ \xrightarrow{\iota_+} \a \xleftarrow{\iota_-} \a_-,\nn\ee
such that there is a homotopy equivalence $\a \simeq \a_+\oplus \a_-$ of dg vector spaces; together with an invariant pairing $\pairing --$ on $\a$ which is non-degenerate up to homotopy, and for which $\a_\pm$ are isotropic, again up to homotopy. 

Manin \Linfinity-triples -- or strongly homotopy Manin triples -- have been defined previously in \cite{Kravchenko}. The definition we give here is compatible with that definition in a sense we discuss in \cref{rem: l8}. 

Then our first result, \cref{thm: local mt}, gives an example of such a homotopy Manin triple. We call it a ``local'' homotopy Manin triple because it is associated to the formal punctured polydisc at a point in complex dimension two. Namely it has 
\be \a = \g \ox R\Gamma(\DDx,\hat\O), \qquad \a_+ = \g \ox R\Gamma(\DD,\hat\O) \cong \g \ox \CC[[w]] \oxC \CC[[z]] \label{loci}\ee
and $\a_-$ a certain differential graded Lie algebra whose cohomology is a copy of 
\be H^1(\a) \cong \g\ox w^{-1}z^{-1} \CC[w^{-1},z^{-1}]\nn\ee
sitting in cohomological degree one. 

(We shall explain the meaning of the \emph{poly}disc, $\DDx$, in a moment.)

\bigskip

Next, our main result is \cref{thm: global mt}, which exhibits what we call a ``global'' homotopy Manin triple. It is an analog, in complex dimension two, of another important and familiar class of Manin triples, namely those of the form
\be \a = \g \ox \bigoplus_{i=1}^N \CC((x-a_i)), \quad \a_+ = \g \ox \bigoplus_{i=1}^N \CC[[x-a_i]] \nn\ee
and 
\be \a_- = \g \ox \CC(x)_{a_1,\dots,a_N}^\8 \label{amintro}\ee
where $\CC(x)_{a_1,\dots,a_N}^\8$ denotes  the commutative algebra of rational expressions in $x$ vanishing at $\8$ and singular at most at the finitely many distinct points $a_1,\dots,a_N\in \CC$.
One may think of these points
\be a_1,\dots,a_N \nn\ee
as \emph{marked points} or \emph{punctures} in the complex plane. This family of Manin triples is -- after introducing a central extension -- closely related to rational Gaudin models \cite{FFR} and rational conformal blocks on the Riemann sphere; see \cite{FrenkelBenZvi} and references therein. 

The homotopy Manin triple $(\a,\a_+,\a_-)$ we introduce in \cref{thm: global mt} is also defined by a finite collection of distinct marked points,
\be (w_1,z_1),\dots,(w_N,z_N),\nn\ee
which now live in $\CC^2$. 
For $\a$ and $\a_+$ it has direct sums of copies of the dg Lie algebras mentioned in \cref{loci} above, one for each marked point:
\be \a \cong \bigoplus_{i=1}^N \g \ox R\Gamma(\DDpx{{w_i,z_i}},\hat \O),\qquad \a_+  \cong \bigoplus_{i=1}^N \g \ox R\Gamma(\DDp{{w_i,z_i}},\hat \O) ,\nn\ee
and it has 
\be \a_- = \gglob, \nn\ee 
a certain dg Lie algebra which we shall describe in detail in \cref{sec: global manin triple} below.

\bigskip

We should stress that the existence of this homotopy Manin triple of \cref{thm: global mt} should not be seen as particularly surprising, conceptually. It is essentially implicit already in \cite{FHK}: see Proposition 1.1.4 there and \cref{prop: hoker} below. 
Our main goal in the present paper is rather to give explicit \emph{models}, in dg Lie algebras, for the various derived algebras above and the maps between them, and to give explicit descriptions of the various homotopies. Our hope is that these models will prove to be useful for doing concrete calculations. 

Now we indicate how we construct these models and introduce the other main theme of the present paper, which is what we call the \emph{rectilinear} setting.

\subsection{The rectilinear setting}
At this point we are about to do something which, to an algebraic geometer, will probably appear mildly barbarous.

Rather than starting with the usual affine space $\AA^2 = \Spec\CC[w,z]$, in this paper we are going to start instead with the product \emph{in topological spaces} of two copies of the affine line $\AA^1$. We shall call the resulting space,
\be \Rect := \AA^1 \ttop \AA^1,\nn\ee
rectilinear space. It has the same set of closed points as $\AA^2$, namely $\{(a,b)\in \CC^2\}$, but very many fewer generalized points: it has points corresponding to the rectilinear lines $w=a$ and $z=a$, $a\in \CC$, but it lacks points corresponding to all the other algebraic curves:
\be\includegraphics{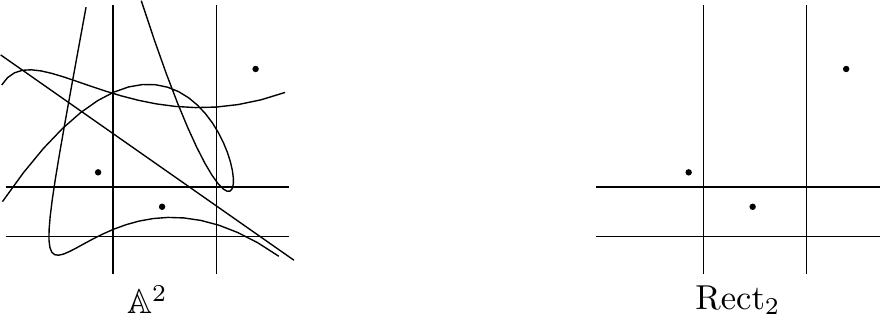}
\nn\ee
We get a ringed space $(\Rect,\O)$ where $\O$ is the sheaf of commutative algebras whose spaces of local sections are spanned by products $f(w) g(z)$ of rational functions.

Similarly, locally, in place of the formal 2-disc $\Disc_2 = \Spec\CC[[w,z]]$, we shall work with the product in topological spaces of two copies of the formal 1-disc,
\be \DD := \Disc_1 \ttop \Disc_1, \nn\ee
which we shall call the formal polydisc. Like $\Disc_2$, it has a single closed point, $(0,0)$, and generalized points corresponding to the lines $w=0$ and $z=0$, but it lacks points corresponding to all the other germs of algebraic curves:
\be\includegraphics{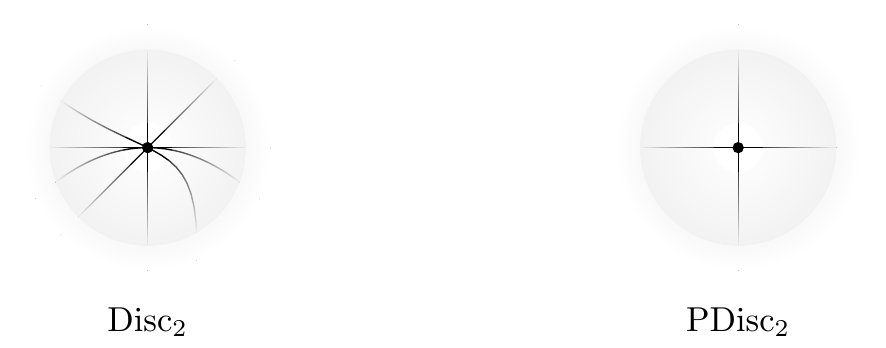}
\nn\ee

This choice breaks symmetry, and results in a space with fewer good properties: $\AA^2$ is an affine scheme; $\Rect$ is merely a ringed space. 
To motivate it, let us give some background on the problem in mathematical physics we are seeking to develop tools to address.

\subsection{Motivation}
We are ultimately interested in solving the spectral problem for integrable quantum field theories in $1+1$ spacetime dimensions. The specific approach we have in mind originates in the paper \cite{FFsolitons}, in which Feigin and Frenkel make the important observation that certain integrable quantum field theories can be seen as quantum Gaudin models of affine type. See also \cite{FH}. And indeed, Vicedo \cite{V17} showed that large classes of classical integrable field theories can be realized as classical Gaudin models of affine type. See also \cite{DLMV2,LacroixThesis,Lacroix2019,ABL}.  

By ``solving the spectral problem for an integrable quantum theory'', we mean roughly speaking defining a hierarchy of mutually commuting conserved quantities (``higher Hamiltonians'') and then characterizing their joint spectra on suitable classes of representations.  

Quantum Gaudin models of \emph{finite} type are among the best-understood quantum integrable systems. In particular, there is a complete description of their higher Hamiltonians, which generate what is known variously as a Bethe or Gaudin algebra \cite{Fopers,MTV1,MTVschubert,RybnikovProof}, and of the spectra of these higher Hamiltonians, which is given in terms of certain local systems called opers \cite{BDopers,Fopersontheprojectiveline,MVopers}. This description can be interpreted as the geometric incarnation of the Langlands correspondence \cite{Fre07}.

At least from one perspective, tools from chiral conformal field theory -- namely, vertex algebras and rational conformal blocks -- are key to establishing these results. These rational conformal blocks are defined \emph{on the spectral plane}, i.e. the copy of the complex plane $\CC$ with marked points which defines any rational Gaudin model;  a Gaudin model of finite type is not in any sense a field theory itself. 

Quantum Gaudin models of affine type are less systematically understood (though see \cite{FFsolitons,FH,LVY,LVY2,Young,KTL,FY}). One problem is that it is difficult to extend tools which work in finite type to affine types; see e.g. \cite{Young2} for an attempt to do so for the Wakimoto construction/Feigin-Frenkel homomorphism. In any case, any attempt to generalise from a finite type algebra $\g$ to an affine type algebra $\gh$ is almost bound to be missing half of the story if both are merely regarded as Kac-Moody algebras, because it ignores the geometrical interpretation of $\gh$ as a centrally extended current algebra, associated to a punctured disc; which is to say, morally speaking, that it ignores the fact that the Gaudin model is describing a field theory. Indeed, vertex algebras, and the additional control they give over the representation theory of $\gh$, should again enter the picture, but these vertex algebras should be associated to a copy of $\CC$ which is quite distinct from the spectral plane. In the case of chiral conformal field theories at least, this new copy of $\CC$ is morally the \emph{worldsheet} of the field theory.

Thus, one expects to be in a situation in which there are two copies of the complex plane in play, 
\be \CC_{\text{spectral plane}} \times \CC_{\text{worldsheet}}, \nn\ee
whose interpretations are conceptually quite distinct.
Let $w,z$ be the Cartesian coordinates. It is natural to think that one will ultimately want to attach representation-theoretic data to:
\begin{itemize}
\item rectilinear lines, e.g. fibres $(w=a)$ over points in the spectral plane, 
\item points $(a,b)$, and probably also
\item rectilinear flags, e.g.
\be (a,b) \subset (w=a) \subset \CC^2 .\label{rflag}\ee
\end{itemize}
It is much less clear that one needs other curves such as $(w=z)$ or $(w=z^2)$ in this context; at the very least, the rectilinear lines and flags certainly enjoy a preferred status. To say much the same thing another way: one expects to encounter functions belonging to $\CC(w) \oxC \CC(z)$, such as $\frac{1}{w-a} \frac{1}{z-b}$, but the function $\frac{1}{w-z}$, for example, looks very strange in this context since the physical meanings of the coordinates $w$ and $z$ are so different.

It is these intuitions which lead us to consider the rectilinear space $(\Rect,\O)$ we sketched above. 

With this digression on motivations complete, let us return to describing the contents of the present paper.

\subsection{The rectilinear adelic complex} 
In view of the discussion above, one would like to have an explicit model for the derived spaces of sections $R\Gamma(-,\O)$ of the structure sheaf $\O$ on rectilinear space $\Rect$, and one would like this model to have the property of being well-adapted to eventually attaching data to rectilinear lines and flags, as well as closed points. 

That leads us to construct a model of $R\Gamma(-,\O)$ in the spirit of the adelic complexes for schemes due to Parshin \cite{Parshin1976} and Beilinson \cite{BeilinsonAdeles}. Their construction involves associating a commutative algebra to each flag of subschemes. 
The algebra associated to a given flag is defined by an elegant and rather intricate procedure of repeated localizations and completions.

In our case, what considering merely the topological product $\Rect= \AA^1\ttop\AA^1$ of affine lines buys us is that (a) we get a much smaller set of flags, consisting only of the rectilinear flags and (b) the algebras attached to flags are simpler to describe. The resulting \emph{rectilinear adelic complex} is given in \cref{sec: rect adel}, where the main result is \cref{thm: flasque res}. 

Subsequently, when we move to considering the homotopy Manin triples of \cref{thm: global mt}, we need only a \emph{finite} collection of flags, built from the finite set of those rectilinear lines which intersect our chosen collection of marked points $(w_i,z_i)_{i=1}^N$. 

In each case, the rectilinear flags form a semisimplicial set, and this gives rise to a semi\emph{co}simplicial commutative algebra. Then, by applying the Thom-Whitney functor whose definition we recall in \cref{sec: ThomWhitney}, we get a dg commutative algebra.

\subsection{Triangular decompositions of enveloping algebras}
Our final collection of results concerns the universal enveloping algebras. Recall that, in the usual case of Lie algebras, a Manin triple $(\a,\pairing --,\a_+,\a_-)$ encodes in particular a decomposition of $\a$ as the direct sum in vector spaces of two Lie subalgebras, $\a\cong_\CC \a_- \oplus \a_+$. That in turn gives rise to an isomorphism
\be U(\a) \cong U(\a_-) \oxC U(\a_+) \nn\ee
between the enveloping algebras; it is an isomorphism of vector spaces and, moreover, of $(U(\a_-), U(\a_+))$-bimodules \cite{Dixmier_1996}. 

In our present setting, of homotopy Manin triples of dg Lie algebras, we get something similar, at least in the special case that the homotopy equivalence of dg vector spaces $\a\simeq \a_- \oplus \a_+$ is actually a strong deformation retract  (we recall the definitions in \cref{sec: retracts} and \cref{sec: triangular decompositions}  below)
\be\begin{tikzcd} 
\a_- \oplus \a_+ \rar[shift left]{}  & \lar[shift left]{} \a \ar[loop right] 
\end{tikzcd}\nn\ee
That turns out to be true of our local homotopy Manin triple from \cref{sec: local manin triple}. The resulting triangular decomposition is given in \cref{sec: triangular decompositions}, \cref{cor: disc bimod}. On the other hand it is not true of the global homotopy Manin triple of \cref{sec: global manin triple}. The remainder of \cref{sec: triangular decompositions} is dedicated to introducing a certain modification of the global homotopy Manin triple for which this extra condition does hold; see \cref{thm: global retract} and its corollary \cref{cor: global bimod}.

(The need for strong deformation retracts here is related to the fact that we insist on staying with dg Lie algebras and their dg associative enveloping algebras, rather than going to the less concrete but perhaps conceptually more natural setting of \Linfinity algebras and their envelopes. This choice is motivated by nothing deeper than the authors' impression that PBW-type statements become rather subtle for \Linfinity algebras; see \cite{KhoroshkinTamaroff} and references therein including \cite{Baranovsky,MorenoFernandez}.)

\subsection{Outline}
This paper is structured as follows. 

After defining rectilinear space $(\Rect,\O)$ and the formal polydisc $(\DDx,\hat\O)$ in \cref{sec: rectilinear space}, we introduce  in \cref{sec: rectilinear flags} the semisimplicial set of rectilinear flags and associated semicosimplicial algebras. Then in \cref{sec: rect adel} we introduce the rectlinear adelic complex.

The Thom-Whitney, or Thom-Sullivan, construction is recalled in \cref{sec: ThomWhitney}. 

Then in \cref{sec: Manin triples} we give our notion of what it means to have a homotopy Manin triple, before giving our two classes of examples in \cref{sec: local manin triple} and \cref{sec: global manin triple}. These sections contain the main results of the paper, \cref{thm: local mt} and \cref{thm: global mt}. Finally, results about the enveloping dg associative algebras are collected in \cref{sec: triangular decompositions}. 

\bigskip
\begin{ack}
The authors would like to thank Jon Pridham for useful discussions. 
CY is grateful to Leron Borsten and Hyungrok Kim for useful discussions and many helpful suggestions.

This authors gratefully acknowledge the financial support of the Leverhulme Trust, Research Project Grant number RPG-2021-092. 
\end{ack}
\section{Rectilinear space and the formal polydisc}\label{sec: rectilinear space}
We work over $\CC$. The tensor product $\ox$ means $\ox_\CC$ throughout.

\subsection{The affine line $\AA^1$ and the formal disc $\Disc_1$}
Recall that one thinks of the polynomial algebra $\CC[z]$ as the algebra of functions on the affine line $\AA^1:= \Spec\CC[z]$, its prime spectrum. The prime ideals of $\CC[z]$ are $0$ and $(z-a)\CC[z]$ for every $a\in \CC$, so that as a set
\be \AA^1 = \{\eta\} \sqcup \CC ,\nn\ee
where $\eta$ is the generic point. The nonempty open sets in the Zariski topology are the complements $\AA^1 \setminus \{ c_1,\dots,c_k\}$ of finite collections of closed points $c_i\in \CC$, and the structure sheaf 
\be \O: \Open(\AA^1)^\op \to \CAlg; \quad \O(\AA^1 \setminus \{ c_1,\dots,c_k\}) = \CC(z)_{c_1,\dots,c_k} \nn\ee
assigns to such an open set the algebra $\CC(z)_{c_1,\dots,c_k}$ of rational expressions in $z$ singular at most at the missing points.
The stalk at $a\in \CC$ is $\O_a:=\directlim_{U\ni a} \O(U)= S_a^{-1} \CC[z]$, the localization of $\CC[z]$ away from the ideal generated by $(z-a)$ or in other words the algebra of rational expressions in $z$ with no singularity at $z=a$. At the generic point $\eta$ the stalk is $\O_\eta = \CC(z)$, 
the field of all rational expressions in $z$. 

The completion of $\CC[z]$ with respect to the maximal ideal $(z-a)\CC[z]$ is the algebra $\CC[[z-a]]$ of formal power series. One thinks of it as the algebra of functions on the \dfn{formal disc} at $a$, $\Disc_1(a) := \Spec\CC[[z-a]]$. The prime ideals of $\CC[[z-a]]$ are $0$ and the unique maximal ideal $(z-a)\CC[[z-a]]$, so that as a set the formal disc at $a$ has exactly two points, 
\be \Disc_1(a) = \{\eta, a\} ,\nn\ee
namely the generic point which we again call $\eta$, and the closed point $a$.
(When we wish to refer to an abstract copy of the formal disc, we shall sometimes write $\pt$ for the closed point $\Disc_1 = \{\eta,\pt\}$.)
The only nonempty open sets of $\Disc_1(a)$ are $\Disc_1(a)$ itself and the \dfn{punctured formal disc} at $a$, $\Discp_1(a) := \Disc_1\setminus \{a\} = \{\eta\}$. The structure sheaf, which shall denote by $\hat\O$, is given by 
\be \hat\O(\Disc_1(a)) = \CC[[z-a]], \qquad \hat\O(\Discp_1(a)) = \CC((z-a)) \label{odp}\ee
where $\CC((z-a))$ is the algebra of formal Laurent series. Its stalks are $\hat\O_a = \CC[[z-a]]$ and $\hat\O_\eta = \CC((z-a))$. 
Note that there are embeddings of algebras $\O_a \into \hat\O_a$ and $\O_\eta\into \hat\O_\eta$ given by expanding in formal (Laurent) series in the local coordinate $z-a$. 

Both the punctured affine line and the punctured formal disc are again affine schemes: $\AA^1\setminus\{a\} = \Spec \CC[(z-a)^{\pm 1}]$ and $\Discp_1(a) = \Spec \CC((z-a))$. 

\subsection{Rectilinear space $\Rect$}
Let us denote by 
\be \Rect := \AA^1 \ttop \AA^1 \nn\ee
the product in topological spaces of two copies of the affine line. Thus, $\Rect$ is the set-theoretic product $\AA^1\times \AA^1$, endowed with the product of the Zariski topologies. Let $w,z: \CC \times \CC \to \CC$ be the Cartesian coordinates. 
As a set, $\Rect$ consists of 
\begin{itemize}
\item All the points $(a,b) \in \CC^2$
\item All the lines $(w=a)$ for $a\in \CC$
\item All the lines $(z=b)$ for $b\in \CC$
\item The generic point $\Eta$.
\end{itemize}
Here, we are adopting the suggestive notations $(w=a)$ and $(z=b)$ for $(a,\eta)$ and $(\eta,b)$ respectively, and we write $\Eta$ for the generic point $(\eta,\eta)$. The closure of the line $(w=a)$ consists of the line and all its points: $\ol{\{(w=a)\}} = \{ (w=a) \} \sqcup \{(a,b): b\in \CC\}$; the closure of the generic point $\Eta$ is all of $\Rect$: $\ol{\{\Eta\}} = \Rect$.  
Complements of closures of lines form a base of the open sets. Thus, every nonempty open subset is the complement of finitely many closed points and the closures of finitely many lines:
\begin{subequations}\label{Udef}
\be U = \Rect \setminus \bigcup_{i=1}^m \ol{\{(w=a_i)\}} \setminus \bigcup_{j=1}^n \ol{\{(z=b_j)\}} \setminus \bigcup_{k=1}^p\{ (c_k,d_k) \} .\ee
Let $\O: \Open_{\Rect}^\op \to \CAlg$ be given by 
\begin{align} 
\O(U) := \CC(w)_{a_1,\dots,a_m} \oxC \CC(z)_{b_1,\dots,b_n}.
\end{align}
\end{subequations}
This is a sheaf in commutative algebras. Its restriction maps are all injections.

In this way, we get a ringed space $(\Rect,\O)$ which we call \dfn{rectilinear space}. Its algebra of global sections is the polynomial algebra $\mc O(\Rect) = \CC[w]\oxC \CC[z] = \CC[w,z]$. The stalk at any closed point $(a,b)\in \CC^2\into\Rect$ is
\be \mc O_{(a,b)} = S_a^{-1}\CC[w] \oxC S_b^{-1} \CC[z] ,\nn\ee
which is a local ring: its unique maximal ideal is $(w-a) \mc O_{(a,b)} + (z-b) \mc O_{(a,b)}$. 
However, the remaining stalks are not local rings: for example, the stalk of $\O$ at the generic point $\Eta\in \Rect$ is 
\be \O_\Eta = \CC(w) \oxC\CC(z).\nn\ee

\subsection{The formal polydisc $\DD$}
Let us denote by 
\be \DDp{{a,b}} = \Disc_1(a) \ttop \Disc_1(b) \nn\ee
the product in topological spaces of two copies of the formal disc. As a set  $\DDp{{a,b}}$ consists of four points: 
\begin{itemize}
\item The closed point $(a,b) \in \CC^2$
\item The line $(w=a)$
\item The line $(z=b)$
\item The generic point $\Eta$.
\end{itemize}
We call $\DDp{{a,b}}$ the \dfn{formal polydisc}, and 
\be \DDpx{{a,b}}:= \DDp{{a,b}} \setminus \{(a,b)\}, \nn\ee
the \dfn{punctured formal polydisc}, at $(a,b)$.  
Let $\hat\O$ be the sheaf in commutative algebras on $\DDp{{a,b}}$ given by
\begin{align}
\hat\O(\DDp{{a,b}}) = \hat\O( \DDpx{{a,b}} ) &=  \CC[[w-a]] \oxC \CC[[z-b]] \label{odp2}\\
\hat\O(\DDp{{a,b}} \setminus \ol{\{(w=a)\}}) &=  \CC((w-a)) \oxC \CC[[z-b]] \nn\\
\hat\O(\DDp{{a,b}} \setminus \ol{\{(z=b)\}}) &=  \CC[[w-a]] \oxC \CC((z-b)) \nn\\
\hat\O(\{ \Eta \}) &=  \CC((w-a)) \oxC \CC((z-b)) \nn
\end{align}
whose stalks are
\begin{alignat}{2} 
\hat\O_{\Eta}  &= \CC((w-a)) \oxC \CC((z-b))  &\qquad \hat\O_{(w=a)} &= \CC[[w-a]] \oxC \CC((z-b))  \nn\\
\hat\O_{(z=b)} &= \CC((w-a)) \oxC \CC[[z-b]]  &\qquad \hat\O_{(a,b)} &= \CC[[w-a]] \oxC \CC[[z-b]].  \nn
\end{alignat}
(It is the external tensor product of the structure sheaves on the two factors $\Disc_1$.)

When we identify $\DDp{{a,b}}$ as a subset of $\Rect$ in the obvious way, there are embeddings of algebras $\O_x \into \hat\O_x$ for every point $x\in \DDp{{a,b}}$,  given by expanding in formal series in both the local coordinates, $w-a$ and $z-b$. 

\subsection{Derived global sections and higher sheaf cohomology}\label{sec: sheaf cohomology}
There is a vital difference between the disc $\Disc_1$ and the polydisc $\DD$. If we remove the closed point from the $\Disc_1$, the algebra of global sections of the structure sheaf gets bigger, as we see in \cref{odp}:
\be \Gamma(\Disc_1,\hat\O) = \CC[[z]] \subsetneq \CC((z)) = \Gamma(\Discp_1,\hat\O). \nn\ee
By contrast, if we remove the closed point from the polydisc in dimension two, we don't get any more sections than we had before, as we see in \cref{odp2}:
\be \Gamma(\DD,\hat\O) = \CC[[w]] \oxC \CC[[z]] = \Gamma(\DDx,\hat\O). \nn\ee 

The same is true of $\Rect$ compared to the affine line $\AA^1$. If we remove a single point $c\in \CC$ from the affine line, the algebra of global sections of the structure sheaf goes from the polynomial algebra $\CC[z]$ to the algebra $\CC(z)_c = \CC[(z-c)^{\pm 1}]$ of Laurent polynomials. By contrast,  note the lack of dependence on the $(c_k,d_k)$ in \cref{Udef}: we may remove as many closed points as we wish and no more global sections appear.

This is a classical phenomenon which occurs also for the structure sheaf on the usual affine plane  $\AA^2:= \Spec\CC[w,z]$. It has an analog in the complex-analytic setting, known as Hartog's theorem (see e.g. \cite[Theorem 1.8]{Kodaira} or \cite{Hartshorne}): in complex dimension at least two, every holomorphic function on a punctured polydisc can be analytically continued to a function on the unpunctured polydisc. 

As stressed in \cite{FHK,Kapranov}, and in \cite{GWHigherKM}, one should think that the ``missing'' global sections over the punctured space have not vanished, but merely moved to higher cohomology.
Recall that the \dfn{derived space of global sections} $R\Gamma^\bul(\DDx,\hat\O)$ of the sheaf $\hat\O$ on $\DDx$ is the cochain complex defined, up to quasi-isomorphism, by the requirement that its cohomology computes the sheaf cohomology of $\hat\O$, 
\be H^\bul(R\Gamma(\DDx,\hat\O)) \cong H^\bul(\DDx,\hat\O).  \nn\ee
There is higher cohomology in the case of the punctured polydisc (cf. \cref{cor: cohomology} below)
\be
H^k(\DDx,\hat\O) = 
\begin{cases} \CC[[w]]\oxC\CC[[z]] & k =0 \\ w^{-1} z^{-1} \CC[w^{-1}, z^{-1}] & k = 1 
\\ 0 & k \notin \{ 0,1\}
  \end{cases} \nn
\ee
One can think of $w^{-1} z^{-1} \CC[w^{-1}, z^{-1}]$ here as the higher analog of the negative modes $z^{-1} \CC[z^{-1}] \subset \CC((z))$ in the one-dimensional case.

Our aim in the present paper is to construct explicit models of spaces of derived sections, in our rectilinear setting, with the following properties:
\begin{itemize}
\item we want models both for the local situation, i.e. for $R\Gamma(\DDx,\hat O)$, and for the global case $R\Gamma(\Rect\setminus\{\text{closed points}\},\O)$;
\item we want these models to make explicit the global-to-local maps 
\be R\Gamma(\Rect\setminus\bigcup_i\{(a_i,b_i,\O)\}) \to R\Gamma(\DDpx{{a_i,b_i}},\hat O),\nn\ee that are the higher analogs of taking formal Laurent expansions;
\item ultimately we want models in dg commutative algebras, rather than just dg vector spaces.
\end{itemize}
Our motivation is that we want explicit models, in dg Lie algebras, for the higher analogs of the usual current Lie algebras $\g \ox \CC((t)) = \g\ox \Gamma(\Discp_1,\hat O)$, and their global analogs, and the dg Lie algebra maps between them.

Models in dg commutative algebras for derived sections can be obtained in various ways; see \cite[Appendix A]{Kapranov}. The construction we use centres on the Thom-Sullivan-Whitney functor, whose definition we recall in \cref{sec: ThomWhitney}.

The starting point is the familiar definition of sheaf cohomology. Recall that the sheaf cohomology of $\hat\O$ on $\DDx$ is, by definition, the cohomology $H^\bul(\DDx,\hat O):= H^\bul(\F)$ of any resolution 
\be 0 \to \hat\O \to \F^0 \to \F^1 \to \dots \nn\ee
of $\hat\O$ by flasque sheaves. 
We shall construct such resolutions, of $\hat \O$ and of $\O$, in the spirit of the adelic complexes for schemes due to Parshin \cite{Parshin1976} and Beilinson \cite{BeilinsonAdeles}, but adapted to our simpler rectilinear spaces, $\DD$ and $\Rect$ (which are not schemes).

\begin{rem}
For more on adelic complexes, see e.g. \cite{Huber,Parshin2000,Osipov,Parshin2011}.
In the case of schemes, the algebras attached to flags of subschemes are defined by repeated localizations and completions -- see \cite{Osipov}, especially \S3.2 and \S3.3, and references therein. 
One of the ways in which our present rectilinear setting is simpler is that, because we just have a topological product of affine lines, the algebras we attach to flags below will be merely products of algebras appearing in that familiar case.  
Another is that we simply have a much smaller semisimplicial set of flags, because we need only the rectilinear flags.
\end{rem}

\section{Semicosimplicial algebras associated to rectilinear flags}\label{sec: rectilinear flags}
We must first introduce semisimplicial sets of rectilinear flags, and algebras associated to them: these will be the building blocks of our models for spaces of derived sections.

\subsection{Rectilinear flags}\label{sec: flags}
Given any subset $U\subseteq \Rect$ (open or not), let $\Flag_n(U)$ for $n=0,1,2$ denote the set of $n$-step flags in $U$:
\begin{align}
\Flag_n(U) &:= 
\left\{ (a_0,\dots,a_n)\in U^n: \ol{\{a_0\}} \subsetneq \dots \subsetneq \ol{\{a_n\}} \right\}.\nn
\end{align} 
(Thus, $\Flag_1(\Rect)$ consists of points-in-lines and lines-in-the-surface, and $\Flag_2(\Rect)$ consists of points-in-lines-in-the-surface.)

Let 
\begin{align} 
\del^n_i  : \Flag_{n+1}(U) &\to \Flag_n(U);\nn\\ 
(a_0,\dots,a_{n+1}) &\mapsto (a_0,\dots,\widehat a_i,\dots,a_{n+1}) \nn
\end{align}
be the map given by removing one space from a flag, for $i=0,\dots,n+1$ and $n=0,1,2$.
These maps endow $\Flag_\bul(U)$ with the structure of a \dfn{semisimplicial set}, as we recall below. For intuition, one should visualise  the example of the finite set of flags $\Flag_\bul(\DDp{{a,b}})$ for the formal polydisc $\DDp{{a,b}} \subset \Rect$ at the closed point $(a,b)$. It consists of  exactly four vertices, five edges and two 2-simplices, while the set $\Flag_\bul(\DDpx{{\wop,\zop}})\subset \Flag_\bul(\DDp{{\wop,\zop}})$ of flags in the punctured formal polydisc has exactly three vertices, two edges and no higher simplices:
\be
\begin{tikzpicture} 
\tdplotsetmaincoords{0}{180}
\begin{scope}[xshift=-5cm,yshift=4cm,scale=3,tdplot_main_coords,local bounding box=D2]
\coordinate (uU) at (1,-1,0);
\coordinate (w=U)  at (1,0,0)  ;  
\coordinate (z=u) at (0,-1,0)  ;  
\coordinate (p1p1) at (0,0,0) ;   
\draw[thick,black,fill=blue,fill opacity=.1] (p1p1) -- (w=U) -- (uU)   --cycle; 
\draw[thick,black,fill=blue,fill opacity=.1] (p1p1) -- (z=u) -- (uU)  --cycle; 

\node[fill,circle,inner sep=1.5,draw,label= left:{$(\wop,\zop)$}]   (NuU)   at (uU)  {};
\node[fill,circle,inner sep=1.5,draw,label= right:{$(z=\zop)$}]   (Nz=u)  at (z=u) {};
\node[fill,circle,inner sep=1.5,draw,label= left:{$(w=\wop)$}]   (Nw=U)  at (w=U) {};
\node[fill,circle,inner sep=1.5,draw,label= right:{$\Eta$}] (Np1p1) at (p1p1)  {};
\end{scope}
\begin{scope}[xshift=3cm,yshift=4cm,scale=3,tdplot_main_coords,local bounding box=D2p]
\coordinate (w=U)  at (1,0,0)  ;  
\coordinate (z=u) at (0,-1,0)  ;  
\coordinate (p1p1) at (0,0,0) ;   
\draw[thick,black,fill=blue,fill opacity=.1] (p1p1) -- (w=U);
\draw[thick,black,fill=blue,fill opacity=.1] (p1p1) -- (z=u);
\node[fill,circle,inner sep=1.5,draw,label= right:{$(z=\zop)$}]   (Nz=u)  at (z=u) {};
\node[fill,circle,inner sep=1.5,draw,label= left:{$(w=\wop)$}]   (Nw=U)  at (w=U) {};
\node[fill,circle,inner sep=1.5,draw,label= right:{$\Eta$}] (Np1p1) at (p1p1)  {};
\end{scope}
\path (D2.south) node[below]{$\Flag_\bul(\DDp{{\wop,\zop}})$};
\path (D2p.south) node[below]{$\Flag_\bul(\DDpx{{\wop,\zop}})$};
\end{tikzpicture}
\label{D2pics}\ee

\subsection{Semisimplicial sets}
Let $\Deltas$ denote the category whose objects are the finite totally-ordered sets 
\be [n] := \{0<1<\dots<n\}, \qquad n\in \ZZ_{\geq 0}, \nn\ee
and whose morphisms are the strictly order-preserving maps $\theta: [n] \to [N]$. Such maps are injections and exist only for $n\leq N$. They are generated by the \dfn{coface maps}
\be d^n_i : [n] \to [n+1];\quad j\mapsto \begin{cases} j & j<i \\ j+1 & j\geq i \end{cases}\nn\ee
for $i=0,1,\dots,n+1$ (together with the identity maps $\id_{[n]}$) for $n=0,1,2,\dots$.
One thinks of the category $\Deltas$ as follows:
\be 
\begin{tikzcd} 
\dots {[2]}
\rar[<-,shift left=-4pt]\rar[<-,shift left=4pt]\rar[<-] & 
{[1]} \rar[<-,shift left=2pt]\rar[<-,shift right=2pt]& {[0]}.
\end{tikzcd}
\nn\ee
A \dfn{semisimplicial object} $Z$ in a category $\C$ is a functor $Z: \Deltas^\op\to \C$.  
In particular, a \dfn{semisimplicial set} $S: \Deltas^\op \to \Set$ is a semisimplicial object $S$ in the category of sets. For each $n$, $S([n])$ is called the set of $n$-simplices of $S$. The maps $\del^n_i := S(d^n_i) : S([n+1]) \to S([n])$ are the \dfn{face maps} of $S$. 
\footnote{\dfn{Simplicial sets} are defined in the same way but with ``strictly order-preserving'', i.e. increasing, replaced by  ``weakly order-preserving'', i.e. non-decreasing, in the definition of $\Delta$. Simplicial sets have extra structure (degenerate simplices and degeneracy maps) which we shall not need here.}

In our present case, we have the functor 
\be \Flag(\Rect):\Deltas^\op \to \Set \nn\ee 
given on objects by $[n] \mapsto \Flag_n(\Rect)$ and on morphisms by $d^n_i \mapsto \del^n_i$. One may think of the semisimplicial structure on $\Flag_\bul(\Rect)$ as follows:
\be 
\begin{tikzcd} 
\Flag_2(\Rect) 
\rar[shift left=-4pt]\rar[shift left=4pt]\rar & 
\Flag_1(\Rect) \rar[shift left=2pt]\rar[shift right=2pt]& \Flag_0(\Rect).
\end{tikzcd}
\nn\ee

\subsection{The comma category $\Pre S$}
Given any semisimplicial set $S:\Deltas^\op \to \Set$, let $\Pre S$ denote the category whose objects are the simplices of $S$, and in which there is a unique morphism $f\to F$ if $f$ is a subsimplex of $F$, i.e. if $f= \phi(F)$ for some morphism $\phi$ of $S(\Deltas^\op)$, and no morphisms $f\to F$ otherwise. 
One can regard $\Pre S$ as a partially ordered set.
For example, the category $\Pre{\Flag_\bul(\DDp{{\wop,\zop}})}$ is the partially ordered set given by
\be\begin{tikzpicture}
\tdplotsetmaincoords{0}{180}
\begin{scope}[xshift=0,scale=6.5,tdplot_main_coords,local bounding box=G]
\coordinate (w=U)  at (1,0,0)  ;  
\coordinate (z=u) at (0,-1,0)  ;  
\coordinate (uU) at (1,-1,0)  ;  
\coordinate (p1p1) at (0,0,0) ;   
\coordinate ({w=U }{p1p1}) at (barycentric cs:p1p1=1,{w=U}=1);
\coordinate ({z=u }{p1p1}) at (barycentric cs:p1p1=1,{z=u}=1);
\coordinate ({uU  }{p1p1}) at (barycentric cs:p1p1=1,{uU}=1);
\coordinate ({uU  }{z=u }{p1p1}) at (barycentric cs:p1p1=1,{z=u}=1.3,{uU}=1);
\coordinate ({uU  }{w=U }{p1p1}) at (barycentric cs:p1p1=1,{w=U}=1.3,{uU}=1);
\coordinate ({uU  }{z=u }) at (barycentric cs:{z=u}=1,{uU}=1);
\coordinate ({uU  }{w=U }) at (barycentric cs:{w=U}=1,{uU}=1);

\node (NuU)  at (uU) {$\bigl((\wop,\zop)\bigr)$};
\node (Nw=U)  at(w=U) {$\bigl((w=\wop)\bigr)$};
\node (Nz=u) at (z=u) {$\bigl((z=\zop)\bigr)$};
\node (Np1p1) at (p1p1) {$\bigl(\Eta\bigr)$};
\node[rotate=-45] (N{uU  }{p1p1}) at ({uU  }{p1p1}) {$\bigl((\wop,\zop), \Eta\bigr)$};
\node (N{w=U }{p1p1}) at ({w=U }{p1p1}) {$\bigl((w=\wop) , \Eta\bigr)$};
\node[rotate=90] (N{z=u }{p1p1}) at ({z=u }{p1p1}) {$\bigl((z=\zop) , \Eta\bigr)$};
\node[rotate=90] (N{uU  }{w=U }) at ({uU  }{w=U }) {$\bigl((\wop,\zop) , (w=\wop)\bigr)$};
\node (N{uU  }{z=u }) at ({uU  }{z=u }) {$\bigl((\wop,\zop) , (z=\zop)\bigr)$};
\node[rotate=0] (N{uU  }{w=U }{p1p1}) at ({uU  }{w=U }{p1p1}) {$\bigl((\wop,\zop) , (w=\wop) , \Eta\bigr)$};
\node[rotate=0] (N{uU  }{z=u }{p1p1}) at ({uU  }{z=u }{p1p1}) {$\bigl((\wop,\zop) , (z=\zop) , \Eta\bigr)$};
\begin{scope}[every node/.style={midway,circle,inner sep = 0,draw,fill=white}]
\draw[->] (Np1p1) -- (N{z=u }{p1p1}); \draw[->] (Np1p1) -- (N{w=U }{p1p1}); 
\draw[->] (Nz=u) --(N{z=u }{p1p1}); \draw[->] (Nw=U) -- (N{w=U }{p1p1}); 
\draw[->] (Nz=u) --(N{uU  }{z=u }); \draw[->] (Nw=U) -- (N{uU  }{w=U }); 
\draw[->] (N{uU  }{w=U }) -- (N{uU  }{w=U }{p1p1});   
\draw[->] (N{uU  }{z=u }) -- (N{uU  }{z=u }{p1p1});   
\draw[->] (NuU) -- (N{uU  }{w=U });
\draw[->] (NuU) -- (N{uU  }{z=u });
\draw[->] (NuU) -- (N{uU  }{p1p1});
\draw[->] (Np1p1) -- (N{uU  }{p1p1});
\draw[->] (N{w=U }{p1p1}) -- (N{uU  }{w=U }{p1p1});   
\draw[->] (N{z=u }{p1p1}) -- (N{uU  }{z=u }{p1p1});   
\draw[->] (N{uU  }{p1p1}) -- (N{uU  }{z=u }{p1p1});   
\draw[->] (N{uU  }{p1p1}) -- (N{uU  }{w=U }{p1p1});   
\end{scope}
\end{scope}
\end{tikzpicture}
\label{flag2dp}\ee
There is another useful way of regarding the category $\Pre S$.
Recall that the Yoneda embedding
$\Deltas  \xrightarrow{\text{Yoneda}} [\Deltas^\op,\Set]$
embeds $\Deltas$ as a full subcategory of the category of semisimplicial sets, by sending $[n]\in \Deltas$ to  the \dfn{standard $n$-simplex} $\Delta^n := \Hom_{\Deltas}(-,[n])$.
We can then regard $\Pre S$ as the comma category associated to the diagram of functors
\be
\begin{tikzcd}       {}     & \mathbf 1\dar{S}  \\ 
\Deltas \rar{\text{Yoneda}} & {[\Deltas^\op,\Set]}
\end{tikzcd}
\nn\ee
That is: we can think that an object of $\Pre S$ is by definition a copy of the standard $n$-simplex $\Delta^n$ for some $n$ together with a map of semisimplicial sets $\Delta^n \to S$; and a morphism from $(\Delta^n\to S)$ to $(\Delta^N \to S)$ in $\Pre S$ is a morphism $\phi:\Deltas^n\to \Deltas^N$ of semisimplicial sets such that the diagram 
\be\begin{tikzcd} \Delta^n \arrow{rr}{\phi}\drar & & \Delta^N \dlar \\ & S &  \end{tikzcd}\nn\ee
commutes.
An advantage of this perspective on $\Pre S$ is that we get the functor
\be \Pre S \xrightarrow{\text{Forgetful}} \Deltas \xrightarrow{\text{Yoneda}} 
[\Deltas^\op, \Set] \nn\ee
which forgets about the maps of the simplices into $S$. We can think of this functor as a diagram in the category $[\Deltas^\op,\Set]$ of semisimplicial sets. (Its colimit is $S$ itself.) 

For example the image of $\Pre{\Flag_\bul(\DDp{{\wop,\zop}})}$ in the category of semisimplicial sets is the diagram
\be\begin{tikzpicture}
\tdplotsetmaincoords{0}{180}
\begin{scope}[xshift=0,scale=4,tdplot_main_coords,local bounding box=G]
\coordinate (w=U)  at (1,0,0)  ;  
\coordinate (z=u) at (0,-1,0)  ;  
\coordinate (uU) at (1,-1,0)  ;  
\coordinate (p1p1) at (0,0,0) ;   
\coordinate ({w=U }{p1p1}) at (barycentric cs:p1p1=1,{w=U}=1);
\coordinate ({z=u }{p1p1}) at (barycentric cs:p1p1=1,{z=u}=1);
\coordinate ({uU  }{p1p1}) at (barycentric cs:p1p1=1,{uU}=1);
\coordinate ({uU  }{z=u }{p1p1}) at (barycentric cs:p1p1=1,{z=u}=1.5,{uU}=1);
\coordinate ({uU  }{w=U }{p1p1}) at (barycentric cs:p1p1=1,{w=U}=1.5,{uU}=1);
\coordinate ({uU  }{z=u }) at (barycentric cs:{z=u}=1,{uU}=1);
\coordinate ({uU  }{w=U }) at (barycentric cs:{w=U}=1,{uU}=1);

\node (NuU)  at (uU) {$\Deltas^0$};
\node (Nw=U)  at(w=U) {$\Deltas^0$};
\node (Nz=u) at (z=u) {$\Deltas^0$};
\node (Np1p1) at (p1p1) {$\Deltas^0$};
\node (N{uU  }{p1p1}) at ({uU  }{p1p1}) {$\Deltas^1$};
\node (N{w=U }{p1p1}) at ({w=U }{p1p1}) {$\Deltas^1$};
\node (N{z=u }{p1p1}) at ({z=u }{p1p1}) {$\Deltas^1$};
\node (N{uU  }{w=U }) at ({uU  }{w=U }) {$\Deltas^1$};
\node (N{uU  }{z=u }) at ({uU  }{z=u }) {$\Deltas^1$};
\node (N{uU  }{w=U }{p1p1}) at ({uU  }{w=U }{p1p1}) {$\Deltas^2$};
\node (N{uU  }{z=u }{p1p1}) at ({uU  }{z=u }{p1p1}) {$\Deltas^2$};
\begin{scope}[every node/.style={midway,circle,inner sep = 0,draw,fill=white}]
\draw[->] (Np1p1) -- (N{z=u }{p1p1}); \draw[->] (Np1p1) -- (N{w=U }{p1p1}); 
\draw[->] (Nz=u) --(N{z=u }{p1p1}); \draw[->] (Nw=U) -- (N{w=U }{p1p1}); 
\draw[->] (Nz=u) --(N{uU  }{z=u }); \draw[->] (Nw=U) -- (N{uU  }{w=U }); 
\draw[->] (N{uU  }{w=U }) -- (N{uU  }{w=U }{p1p1});   
\draw[->] (N{uU  }{z=u }) -- (N{uU  }{z=u }{p1p1});   
\draw[->] (NuU) -- (N{uU  }{w=U });
\draw[->] (NuU) -- (N{uU  }{z=u });
\draw[->] (NuU) -- (N{uU  }{p1p1});
\draw[->] (Np1p1) -- (N{uU  }{p1p1});
\draw[->] (N{w=U }{p1p1}) -- (N{uU  }{w=U }{p1p1});   
\draw[->] (N{z=u }{p1p1}) -- (N{uU  }{z=u }{p1p1});   
\draw[->] (N{uU  }{p1p1}) -- (N{uU  }{z=u }{p1p1});   
\draw[->] (N{uU  }{p1p1}) -- (N{uU  }{w=U }{p1p1});   
\end{scope}
\end{scope}
\end{tikzpicture}
\label{flagd2pYoneda}\ee
Intuitively speaking, this tells us how to build $\Flag_\bul(\DDp{{\wop,\zop}})$ by sewing together standard simplices.


\subsection{Semicosimplicial algebras}
A \dfn{semicosimplicial object} $A$ in a category $\A$ is a functor $A: \Deltas\to \A$, i.e. an object of the functor category
\be [\Deltas, \A] .\nn\ee 
The relevant categories $\A$ for us are commutative algebras, Lie algebras, and the differential graded analogs of these. When it is not necessary to be more precise, we shall refer to objects of $\A$ as algebras and to  semicosimplicial objects in $\A$ as \dfn{semicosimplicial algebras}.

Thus, a semicosimplicial algebra $A$ has, by definition, an \emph{algebra} $A([n]) \in \A$ of $n$-cosimplices, for each $n\in \ZZ_{\geq 0}$, and coface morphisms $d^n_i : A([n]) \to A([n+1])$, $i=0,1,\dots,n+1$ between them:
\be 
\begin{tikzcd} 
\dots {A([2])}
\rar[<-,shift left=-4pt]\rar[<-,shift left=4pt]\rar[<-] & 
{A([1])} \rar[<-,shift left=2pt]\rar[<-,shift right=2pt]& {A([0])}.
\end{tikzcd}
\nn\ee

Actually, our semicosimplicial algebras will have a finer structure than this: they will arise from semisimplicial sets (of flags) by first attaching algebras to \emph{individual simplices} (i.e. individual flags) of a given semisimplicial set. It is useful to keep track of this structure, and to that end we make the following definition.

\subsection{$S$-algebras}\label{sec: S algebras} Given a semisimplicial set $S$, define the category of \dfn{$S$-objects in $\A$} or \dfn{$S$-algebras} to be the functor category 
\be [\Pre S,\A] ,\nn\ee 
i.e. the category whose objects are functors $\Pre S \to \A$ and whose morphisms are natural transformations between such functors. 
Given an $S$-object in $\A$ there is a natural way to recover a semicosimplicial object in $\A$. Namely, given a functor $A:\Pre S \to \A$, we may define a functor $\sca A:\Deltas \to \A$ as follows. We set
\be \left(\sca A\right)([n]) := \prod_{f\in S([n])} A_f \nn\ee
and if $\phi: [n] \to [N]$ in $\Deltas$ then the morphism $\left(\sca A\right)(\phi) : \left(\sca A\right)([n]) \to \left(\sca A\right)([N])$ is given by its restrictions to the factors $A_f$:
\be \left(\sca A\right)(\phi)|_{A_f} = \prod_{F\in S([N])\,:\, S(\phi)(F) = f} A(f\to F) \nn\ee
(Recall that $A(f\to F): A_f \to A_F$. A given $n$-simplex $f$ may belong to the boundary of infinitely many $N$-simplices of $S$, and for that reason we need the direct product $\prod$ rather than the direct sum $\bigoplus$. On the other hand, any given $N$-simplex $F$ has only finitely many boundary $n$-simplices, so there are only finitely many factors $A_f$ such that $f\to F$ is a morphism of $\Pre S$, and thus the restriction of $(\sca A)(\phi)\bigl(\sca A([n])\big)$ to the factor $A_F$ is a well-defined sum of finitely many terms.)

\begin{lem}\label{lem: sca1} 
This $\sca$ defines the action on objects of a functor
\be \sca = \sca_S : [\Pre S,\A] \to [\Deltas,\A] \nn\ee
from $S$-objects in $\A$ to semicosimplicial objects in $\A$. 
\qed\end{lem}
\begin{rem}
This functor $\sca$ is left adjoint,
\be\begin{tikzcd}[baseline=0,ampersand replacement=\&]
{[\Deltas,\A]} \ar[r,shift left=5pt, 
"u^*",""{name=A, below}] \& {[\Pre S, \A]} \ar[l,shift left=5pt,
"\sca",""{name=B,above}] \arrow[phantom, from=A, to=B, "\dashv" rotate=-90]
\end{tikzcd},\nn\ee 
to the pull-back $u^*: B \mapsto u^*(B) = B\circ u$ of the forgetful functor $u:\Pre S \to \Deltas$.
\end{rem}

\subsection{Semisimplicial subsets and the restriction morphism}\label{sec: semisimplicial subsets}
The other fact we need concerns semisimplicial \emph{subsets} of $S$.
Suppose $R$ is a semisimplicial subset of $S$, by which we mean that for each $n$ the set of $n$-simplices of $R$ is a subset of the set of $n$-simplices of $S$, $R([n]) \subset S([n])$ and that the face maps respect these embeddings of sets. 
That is, the embedding maps $i([n]): R([n]) \into S([n])$ define a morphism of semisimplicial sets $i:R \to S$, i.e. they are the components of a natural transformation $i$
\be 
\begin{tikzcd}[ampersand replacement=\&]
\Deltas^\op  \rar[bend left=50, ""{name = U, below},"R"{above}]
           \rar[bend right=50,""{name = D, above},"S"{below}] \& \Set
\arrow[Rightarrow, from=U, to=D,"i"]
\end{tikzcd}\nn\ee 
between the functors $R$ and $S$. We get a functor $\Pre i:\Pre R \to \Pre S$ between the corresponding comma categories.
Given an $S$-object $A$ in $\A$ we have then also its restriction $A|_R$, an $R$-object in $\A$. Namely, 
$A|_R$ is the composition
\be \Pre R \to \Pre S \xrightarrow{A} \A.\nn\ee
This defines a functor $[\Pre S,\A] \to [\Pre R,\A]; A \mapsto A|_R$. (That is, $A|_R := (\Pre i)^* A$.)
We can then form two semicosimplicial objects in $\A$, namely $\sca_S A$ and $\sca_R A|_R$.

\begin{lem}\label{lem: sca2} 
There is a morphism of semicosimplicial objects in $\A$, 
\be \pi: \sca_S A \to \sca_R A|_R\nn\ee
given by
\be \pi|_{A_f} = \begin{cases} \id_{A_f} & f\in R \\ 0 &f\notin R. \end{cases}\nn\ee  
\qed\end{lem}

\begin{rem}
Note that while we also have the obvious embedding maps $(\Pi_RA|_R)([n]) \into \Pi_S A([n])$, these do not in general define a morphism of semicosimplicial algebras $\Pi_RA|_R \to \Pi_SA$. Indeed, we get failures of naturality whenever $f\in S([n])$ and $F\in S([N])$ are such that $f \in R([n])$ and yet $F\notin R([N])$.
\end{rem}

\subsection{First example}\label{sec: ADdef}
We now turn to an example which will play a central role. Let $\CAlgemb$ denote the category whose objects are commutative ($\CC$-)algebras and whose morphisms are embeddings of commutative algebras. 

Let $\DD := \DDp{{0,0}}$ be the formal polydisc at the point $(0,0)$. 
To give a $\Flag_\bul(\DD)$-object in $\CAlgemb$, i.e. a functor 
\be \Pre{\Flag_\bul(\DD)} \to \CAlgemb, \nn\ee
is by definition to give a certain commuting diagram of commutative algebras and embeddings between them, cf. \cref{flag2dp} and \cref{flagd2pYoneda}. Let us define such an algebra, $\AD$, as follows.
\be
\AD :=
\begin{tikzpicture}[baseline =100]
\tdplotsetmaincoords{0}{180}
\begin{scope}[xshift=0,scale=7,tdplot_main_coords,local bounding box=G]
\coordinate (w=U)  at (1,0,0)  ;  
\coordinate (z=u) at (0,-1,0)  ;  
\coordinate (uU) at (1,-1,0)  ;  
\coordinate (p1p1) at (0,0,0) ;   
\coordinate ({w=U }{p1p1}) at (barycentric cs:p1p1=1,{w=U}=1);
\coordinate ({z=u }{p1p1}) at (barycentric cs:p1p1=1,{z=u}=1);
\coordinate ({uU  }{p1p1}) at (barycentric cs:p1p1=1,{uU}=1);
\coordinate ({uU  }{z=u }{p1p1}) at (barycentric cs:p1p1=1,{z=u}=1.3,{uU}=1);
\coordinate ({uU  }{w=U }{p1p1}) at (barycentric cs:p1p1=1,{w=U}=1.3,{uU}=1);
\coordinate ({uU  }{z=u }) at (barycentric cs:{z=u}=1,{uU}=1);
\coordinate ({uU  }{w=U }) at (barycentric cs:{w=U}=1,{uU}=1);

\node (NuU)  at (uU) {$\CC[[w]]\oxC\CC[[z]]$};
\node (Nw=U)  at(w=U) {$\CC[[w]]\oxC\CC((z))$};
\node (Nz=u) at (z=u) {$\CC((w))\oxC\CC[[z]]$};
\node[rotate=0] (Np1p1) at (p1p1) {$\RRR$};
\node[rotate=-45] (N{uU  }{p1p1}) at ({uU  }{p1p1}) {$\RRR$};
\node (N{w=U }{p1p1}) at ({w=U }{p1p1}) {$\RRR$};
\node (N{z=u }{p1p1}) at ({z=u }{p1p1}) {$\RRR$};
\node (N{uU  }{w=U }) at ({uU  }{w=U }) {$\CC[[w]]\oxC\CC((z)) $};
\node (N{uU  }{z=u }) at ({uU  }{z=u }) {$\CC((w))\oxC\CC[[z]] $};
\node (N{uU  }{w=U }{p1p1}) at ({uU  }{w=U }{p1p1}) {$\RRR$};
\node (N{uU  }{z=u }{p1p1}) at ({uU  }{z=u }{p1p1}) {$\RRR$};
\begin{scope}[every node/.style={midway,circle,inner sep = 0,draw,fill=white}]
\draw[->] (Np1p1) -- (N{z=u }{p1p1}); \draw[->] (Np1p1) -- (N{w=U }{p1p1}); 
\draw[->] (Nz=u) --(N{z=u }{p1p1}); \draw[->] (Nw=U) -- (N{w=U }{p1p1}); 
\draw[->] (Nz=u) --(N{uU  }{z=u }); \draw[->] (Nw=U) -- (N{uU  }{w=U }); 
\draw[->] (N{uU  }{w=U }) -- (N{uU  }{w=U }{p1p1});   
\draw[->] (N{uU  }{z=u }) -- (N{uU  }{z=u }{p1p1});   
\draw[->] (NuU) -- (N{uU  }{w=U });
\draw[->] (NuU) -- (N{uU  }{z=u });
\draw[->] (NuU) -- (N{uU  }{p1p1});
\draw[->] (Np1p1) -- (N{uU  }{p1p1});
\draw[->] (N{w=U }{p1p1}) -- (N{uU  }{w=U }{p1p1});   
\draw[->] (N{z=u }{p1p1}) -- (N{uU  }{z=u }{p1p1});   
\draw[->] (N{uU  }{p1p1}) -- (N{uU  }{z=u }{p1p1});   
\draw[->] (N{uU  }{p1p1}) -- (N{uU  }{w=U }{p1p1});   
\end{scope}
\end{scope}
\end{tikzpicture}
\label{ADdef}\ee
where $\RRR$ is the commutative algebra
\be \RRR := \CC((z)) \oxC \CC((w)). \nn\ee
That is, $\AD(f) := \RRR$ for all simplices $f$ with only the following exceptions:
\begin{align}
\AD( (0,0) ) &:= \CC[[w]]\oxC\CC[[z]] \nn\\
\AD( (w=0) ) := \AD( (0,0) , (w=0) ) &:= \CC[[w]]\oxC\CC((z)) \nn\\
\AD( (z=0) ) := \AD( (0,0) , (z=0) ) &:= \CC((w))\oxC\CC[[z]] .\nn
\end{align}
Let $\ADx$ denote the restriction, in the sense of \cref{lem: sca2}, of the $\Flag_\bul(\DD)$-algebra $\AD$ to a $\Flag_\bul(\DDx)$-algebra:
\be 
\ADx := \AD|_{\Flag_\bul(\DDx)}.\nn\ee

\subsection{The associated cochain complex of a cosimplicial algebra}\label{sec: cochains}
Let $\CCh(\A)$ denote the category of cochain complexes in $\A$. 
There is a functor 
\be \C: [\Deltas,\A] \to \CCh(\A);\qquad A \mapsto (\C^\bul(A),\dd)\nn\ee 
which assigns to any semicosimplicial object $A$ in $\A$ a cochain complex $(\C^\bul(A),\dd)$ concentrated in nonnegative degrees, its \dfn{associated 
complex}. (See e.g. \cite[\S8.2.1 and \S8.4.3]{Weibel}.) For each $n\geq 0$ the space $\C^n(A)$ is a copy of $A([n])$ put into cohomological degree $n$,
\be \C^n(A) := \downby n A([n]) \nn\ee
and the differential $\dd_\C = \sum_n \dd_\C^n$, $\dd_\C^n : \C^n(A) \to \C^{n+1}(A)$, is given by the alternating sum of the coface maps,
\be \dd_\C^n = \down \circ \left( A(d^n_0) - A(d^n_1) + \dots +(-1)^{n+1} A(d^n_{n+1}) \right).\nn\ee
Here we use the standard notation
\be \upby n : \CCh(\A) \to \CCh(\A); \quad \upby n V := [n] \ox V\nn\ee
where $[n]$ is the one-dimensional graded vector space concentrated in cohomological degree $-n$. In particular if $V$ is concentrated in degree 0 then $\downby n V$ is concentrated in degree $n$.

\subsection{Homotopy equivalences and deformation retracts}\label{sec: retracts}
Recall that a map $f: V \to W$ of cochain complexes $V,W\in \CCh(\A)$ is a \dfn{homotopy equivalence} if it is invertible up to homotopies, in the sense that there exists a map of cochain complexes $g: W \to V$ in the opposite direction such that 
\be g\circ f \simeq \id_V \qquad\text{and}\qquad f\circ g \simeq \id_W. \nn\ee
Here we used $\simeq$ to indicate that two cochain maps are \dfn{homotopic}, meaning that there exists a cochain homotopy between them. That is, in this case, there are maps in $\A$ 
\be h : V^n \to V^{n-1} \qquad\text{and}\qquad k : W^n \to W^{n-1} \nn\ee
for each $n$, such that 
\begin{align} 
g\circ f - \id_V &= [h, \dd_V] := h \circ \dd_V + \dd_V \circ h \nn\\
f\circ g - \id_W &= [k, \dd_W] := k \circ \dd_W + \dd_W \circ k. \nn
\end{align}
This situation is often denoted
\be\begin{tikzcd} 
\ar[loop left,"h"] V \rar[shift left]{f}  & \lar[shift left]{g} W \ar[loop right,"k"] 
\end{tikzcd}.\nn\ee
As a special case, if $g\circ f = \id_V$ holds exactly then $V$ is a \dfn{deformation retract} of $W$:
\be\begin{tikzcd} 
V \rar[shift left]{f}  & \lar[shift left]{g} W \ar[loop right,"h"] 
\end{tikzcd}\nn\ee
See for example e.g. \cite[\S1.5.5]{LodayVallette}. 

Every homotopy equivalence is a quasi-isomorphism of cochain complexes, i.e. it gives rise to an isomorphism in cohomology. For cochain complexes in vector spaces, i.e. for dg vector spaces, the converse is also true. To see this, it is enough to note $H^\bul(V)$ is always a deformation retract of $V$. See e.g. \cite[\S9.4.3]{LodayVallette}. (Recall we work over $\CC$, here and throughout.)

\section{The rectilinear adelic complex}\label{sec: rect adel}
We are now in a position to define the complex which will model the derived sections of the sheaf $\O$ on rectilinear space $\Rect$.  

The main result of this section is \cref{thm: flasque res}. Let us remark that the subsequent sections of the paper are self-contained and can be read independently of this section.

Recall that $\Flag_\bul(U)$ is the semisimplicial set of flags in a subset $U\subset \Rect$ of rectilinear space, as in \cref{sec: flags}. There is manifestly an embedding of semisimplicial sets $\Flag_\bul(U)\into \Flag_\bul(V)$, cf. \cref{sec: semisimplicial subsets}, whenever $U\subset V$, and these embeddings compose correctly, i.e. we get a functor
\be \Flag_\bul(-) : \Open_{\Rect} \to [\Deltas^\op,\Set]_{/\Flag_\bul(\Rect)}^{\emb}; 
\quad U \mapsto \Flag_\bul(U)\nn\ee 
to the category of embedded semisimplicial subsets of $\Flag_\bul(\Rect)$.\footnote{Here, given any semisimplicial set $S$, we let $[\Deltas^\op,\Set]_{/S}^{\emb}$ denote the category of its embedded semisimplicial subsets, i.e. the category whose objects are tuples $(R, R \into S)$ consisting of a semisimplicial set $R$ and an embedding of $R$ into $S$, and whose morphisms $(R,R\into S)\to(T,T\into S)$ are morphisms $R\to T$ such that the diagram 
$\begin{tikzcd}[column sep = tiny,row sep=tiny,ampersand replacement=\&] R \arrow{rr}\drar[hook] \& \& T \dlar[hook] \\ \& S \&  \end{tikzcd}$
commutes.}
We recognize the (inverse) limit 
\be \invlim_{U \ni (a,b)} \Flag_\bul(U) = \Flag_\bul(\DDp{{a,b}})\nn\ee 
as the semisimplicial set $\Flag_\bul(\DDp{{a,b}})$ of flags in the formal polydisc at the point $(a,b)\in \CC^2$. 

Now let us define a $\Flag_\bul(\Rect)$-object $\AAA$ in commutative algebras, i.e. a functor
\be \AAA: \Pre{\Flag_\bul(\Rect)} \to \CAlg \nn\ee 
as follows: every flag, i.e. every simplex, is sent to $\CC(w) \oxC \CC(z)$ with only the following exceptions:
\begin{subequations}\label{def: AA}
\be \AAA( (a,b) ) := S_a^{-1}\CC[w]\oxC  S_b^{-1}\CC[z] \ee
\be \AAA( (a,b) , (w=a) ) := \AAA( ( w=a) ) := S_a^{-1}\CC[w] \oxC \CC(z)\label{AAa}\ee
\be \AAA( (a,b) , (z=b) ) := \AAA( ( z=b) ) := \CC(w)  \oxC S_b^{-1}\CC[z] \label{AAb}\ee
\end{subequations}
for all $a,b\in \CC$. By restriction, we obtain also a $\Flag_\bul(U)$-object $\AAA_U$ in commutative algebras for every subset (open or not) $U\subset \Rect$.

At this point, we would like to apply the functor $\sca$ of \cref{sec: S algebras} to obtain a semicosimplicial algebra $\sca \AAA_U$ for each open $U$. There is however a crucial subtlety. To get the resolution we seek of the sheaf $\O$, it is necessary to modify the definitions of the algebras of $1$-simplices and $2$-simplices:
\begin{align} 
\sca'\AAA_U([0]) &:= \prod_{F\in \Flag_0(U)} \AAA( F ) \nn\\
\sca'\AAA_U([1]) &:= \biggl\{ \mathbf x = (x_{F}) \in\!\!\!\!\!\! \prod_{F\in \Flag_1(U)} \AAA( F ) : \label{adelic products}\\
&\qquad\quad  \parbox{\textwidth}{for all but finitely many flags of the form $F=(\{\pt\} \subset (line) )$ 
, \\ $\phantom{+}\qquad x_{F}$ actually belongs to $\AAA( \{\pt\} )$, and \\ for all but finitely many flags of the form $F=((line) \subset \Eta )$, \\ $\phantom{+}\qquad x_{F}$ actually belongs to $\AAA( (line) )$ } \hspace{-.2\textwidth}\biggr\} \nn\\
\sca'\AAA_U([2]) &:= \biggl\{ \mathbf x = (x_{F}) \in\!\!\!\!\!\! \prod_{F\in \Flag_2(U)} \AAA( F ) : \nn\\
&\qquad\qquad  \parbox{.7\textwidth}{for all but finitely many flags $F=(\{\pt\} \subset (line) \subset \Eta )$, \\ $\phantom{+}\qquad x_{F}$ actually belongs to $\AAA( \{\pt\} )$.} \biggr\} \nn
\end{align}
We obtain a sheaf in semicosimplicial algebras  $U\mapsto \sca'\AAA_U$. The restriction maps $\sca'\AAA_V \to \sca'\AAA_U$ for $U\subset V$ just consist in throwing away some terms in the products and are  manifestly surjective. Thus this sheaf is flasque. On taking the associated cochain complexes we obtain a flasque sheaf in cochain complexes in commutative algebras 
\be U\mapsto \C^\bul(\sca'\AAA_U). \nn\ee
\begin{thm}\label{thm: flasque res}
This sheaf $U\mapsto \C^\bul(\sca'\AAA_U)$ on $\Rect$ is a flasque resolution of $\O$. 

Thus, $\C^\bul(\sca'\AAA)$ is a model for the derived sections of $\O$:
\be R\Gamma^\bul(U,\O) \simeq \C^\bul(\sca'\AAA_U) ,\nn\ee
for each open $U\subset \Rect$.
\end{thm}
\begin{proof} 
The proof is given in \cref{sec: flasque resolution proof}. 
\end{proof}

The Thom-Whitney construction, \cref{sec: ThomWhitney}, provides another model, $\Th^\bul(\sca'\AAA)$, which comes equipped with the structure of a dg commutative algebra. 

For completeness, we note also the following. 
Let $\AD$ be the $\Flag(\DD)$-algebra from \cref{sec: ADdef}. It restricts to a $\Flag(U)$-algebra $\AU$ for each open $U\subset \DD$ and this defines a sheaf $U\mapsto \sca\AU$ in semicosimplicial commutative algebras on $\DD$.  
\begin{thm} 
This sheaf $U\mapsto \C^\bul(\sca\AU)$ on $\DD$ is a flasque resolution of $\hat\O$. 

Thus, $\C^\bul(\sca\AAA)$ is a model for the derived sections of $\hat\O$:
\be R\Gamma^\bul(U,\hat\O) \simeq \C^\bul(\sca\AU) ,\nn\ee
for each open $U\subset \DD$.\qed
\end{thm}

\subsection{Remark on completed local rings}\label{rem: completions}
In place of the definition \cref{def: AA} of $\AAA$, we could make the following alternative choice:
\begin{subequations}\label{def: AAhat}
\be \hat\AAA( (a,b) ) := \CC[[w-a]] \oxC  \CC[[z-b]] \ee
\be \hat\AAA( (a,b) , (w=a) ) := \CC[[w-a]] \oxC \CC((z-b)) \qquad \hat\AAA( ( w=a) ) := \CC[[w-a]] \oxC \CC(z)\nn\ee
\be \hat\AAA( (a,b) , (z=b) ) := \CC((w-a)) \oxC \CC[[z-b]] \qquad \hat\AAA( ( z=b) ) := \CC(w)  \oxC \CC[[z-b]] \nn\ee
\be \hat\AAA( (a,b) , (z=b),\Eta ) := \CC((w-a)) \oxC \CC((z-b)) \qquad \hat\AAA( ( z=b) ,\Eta ) := \CC(w)  \oxC \CC((z-b)) \nn\ee
\be \hat\AAA( (a,b) , (w=a),\Eta ) := \CC((w-a)) \oxC \CC((z-b)) \qquad \hat\AAA( ( w=a) ,\Eta) := \CC((w-a)) \oxC \CC(z)\nn\ee
\be \hat\AAA( \Eta ) := \CC(w) \oxC  \CC(z) \ee
\end{subequations}
(The proof of \cref{thm: flasque res} in \cref{sec: flasque resolution proof} goes through with an additional step: one checks that $\C\sca'\AAA_V\simeq \C\sca'\hat\AAA_V$ for suitably small open sets $V$.)

This choice of $\AAA$ vs. $\hat\AAA$ has an analog in the familiar case of the affine line $\AA^1$, where, at the level of global sections, one has the usual short exact sequence,
\be 0 \to \CC[x] \to \CC(x) \oplus \prod_{a\in\CC} \CC[[x-a]] \to \prod_{a\in \CC}' \CC((x-a)) \to 0 ,\nn\ee
but also the following one,
\be 0 \to \CC[x] \to \CC(x) \oplus \prod_{a\in\CC} S_a^{-1} \CC[x] \to \prod_{a\in \CC}' \CC(x) \to 0 .\nn\ee

\subsection{Global sections as the homotopy kernel.}
In the familiar case of adeles for complex dimension one, we may also consider puncturing the affine line $\AA^1$ at only a prescribed finite collection of closed points $\{a_1,\dots,a_N\}$ of our choice, and we get the following short exact sequence
\be 0 \to \CC[x] \to \CC(x)_{a_1,\dots,a_N} \oplus \bigoplus_{i=1}^N \CC[[x-a_i]] \to \bigoplus_{i=1}^N \CC((x-a_i)) \to 0 \label{line adeles}\ee
which, more conceptually, is
\be 0 \to \Gamma(\AA^1,\O) \to \Gamma(\AA^1\setminus \{a_1,\dots,a_N\},\O) \oplus \bigoplus_{i=1}^N \hat\O_{a_i} 
      \to \bigoplus_{i=1}^N \Gamma(\Discp_1(a_i,\hat \O)) \to 0. \nn\ee
One way to interpret this exact sequence is to say that the space of global sections $\Gamma(\AA^1,\O)$ is the kernel of the map into the tuples of sections over the punctured discs:
\be \Gamma(\AA^1,\O)  = \ker\left( \Gamma(\AA^1\setminus \{a_1,\dots,a_N\},\O) \oplus \bigoplus_{i=1}^N \hat\O_{a_i} 
      \to \bigoplus_{i=1}^N \Gamma(\Discp_1(a_i,\hat \O))\right)\nn\ee
(and this is true whether we choose to complete the local rings $\O_a$ or not, cf. \cref{rem: completions}). 

As explained in \cite{FHK}, this statement has a derived analog: the space of global sections becomes a certain \dfn{homotopy kernel}. In our case, the statement is \cref{prop: hoker} below. This is nothing but a restatement of \cite[Proposition 1.1.4]{FHK} in our rectilinear setting, and the argument below is merely an expanded version of the argument given there\footnote{Though any errors which have appeared in it are of course due to the present authors.}

As we saw in \cref{thm: flasque res},the cochain complex
\be C := \C^\bul(\sca'(\AAA)) \label{cdef}\ee 
models the derived space of global sections of $\O$ on $\Rect$
\be R\Gamma^\bul(\Rect,\O)\simeq \C^\bul(\sca'(\AAA)). \nn\ee
Now let 
\be \xx = \{(a_1,b_1),\dots,(a_N,b_N)\} \nn\ee 
denote a finite collection of marked points in $\CC\times\CC$. 
Flags in the semisimplicial set of rectilinear flags $\Flag_\bul(\Rect)$ either start at one of the marked points, or they do not; thus, by definition of the unnormalized cochains functor $\C^\bul$, we have
\begin{align} C^n = \prod'_{\substack{f \in\Flag_{n}(\Rect)}} \AAA(f)
        &= \prod'_{\substack{f \in\Flag_{n}(\Rect\setminus\xx)}} \AAA(f)  \oplus 
\prod'_{\substack{f\in\Flag_{n}(\Rect) \\ 
            : f_0 \subset \xx}} \AAA(f) \label{Cn}
\end{align}
Here, we recognise the first of the two summands on the right as the space $C_1^n$ of the complex
\be C_1^\bul := \C^\bul(\sca'(\AAA_{\Rect\setminus\xx})) \simeq R\Gamma(\Rect\setminus\xx,\O)\label{c1def}\ee
which we know models the derived space of sections of $\O$ on $\Rect\setminus\xx$.
We want to interpret the second summand on the right in \cref{Cn}. 
When $n=0$, it is nothing but 
\be \bigoplus_{(a,b)\in \xx} \AAA(\{(a,b)\}) = \bigoplus_{(a,b)\in \xx} \O_{(a,b)} =: C_2^0, \label{c2def}\ee
which we may choose to think of as the degree zero space of a complex $0\to C_2^0\to 0$. 
For $n\geq 1$, we need the following observation about the definition \cref{def: AA}: for any flag\footnote{Here and occasionally elsewhere, we use a suggestive but rather loose notation: the flag is strictly-speaking the tuple $(f_0,f_1,\dots)$ with $\ol{\{f_0\}} \subset \ol{\{f_1\}} \subset\dots$.} $f_0 \subset f_1 \subset \dots $ whose first space is a point, we have
\be \AAA(f_0\subset f_1\subset \dots) = \AAA( f_1 \subset\dots) .\nn\ee
(The only cases to check are those in  \cref{AAa,AAb} and $\AAA((a,b) \subset \Eta) = \AAA(\Eta)$.)
If $f_0=(a,b)$ is one of the marked points then the truncated flags $f_1\subset\dots$ on the right here belongs to our semisimplicial set of flags $\Flag_\bul(\DDpx{{a,b}})$ in the punctured polydisc at this marked point. Using this fact, we get that, for $n\geq 1$,
\be C_3^{n-1} := \prod'_{\substack{f\in\Flag_{n}(\Rect) \\ : f_0 \subset \xx}} \AAA(f) 
  = \bigoplus_{(a,b) \in \xx} \C^{n-1}(\sca(\AAA|_{\Flag(\DDpx{{a,b}})}))  \label{c3def}\ee
(note the degree shift). We know the complexes appearing on the right here, namely
\be \C^\bul(\sca(\AAA|_{\Flag_\bul(\DDpx{{a,b}})})) \simeq R\Gamma^\bul(\DDpx{{a,b}},\O),\nn\ee 
model the derived algebra of sections of $\O$ on the punctured formal polydiscs at the marked points $(a,b)\in \xx$.

The $0$-step flags in $C_2$ form a semisimplicial subset of $\Flag_\bul(\Rect)$,  i.e. $\Pre C_2$ is a full subcategory of of $\Pre{\Flag_\bul(\Rect)}$. (It just consists of isolated points.) We get the $C_2$-object in commutative algebras $\AAA|_{C_2}$, which just sends, cf. \cref{def: AA},
\be \AAA : \{(a,b)\} \mapsto S_a^{-1}\CC[w]\oxC  S_b^{-1}\CC[z] = \O_{(a,b)}\nn\ee
for each marked point $(a,b)\in \xx$.

Thus, as a graded vector space, we have that
\be C^\bul = \left(C_1 \oplus C_2 \oplus \down C_3\right)^\bul,\qquad \text{i.e.}\qquad 
 C^n = C_1^n \oplus C_2^n \oplus C_3^{n-1} \label{ccdef}\ee
for each $n$, for these complexes $C$, $C_1$, $C_2$ and $C_3$ we defined in \cref{cdef}, \cref{c1def}, \cref{c2def} and \cref{c3def}. The differential of the complex $C^\bul$ is given, in matrix form, by
\be \dd_C = \bmx \dd_{C_1} + \dd_{C_2} & 0 \\ 
                                     d & -\dd_{C_3} \emx\label{cddef}\ee
where $d$ is a (degree zero) cochain map 
\be d:C_1\oplus C_2 \to C_3\nn\ee 
defined by our choices above. Conceptually, $d|_{C_2}$ is the sum of the maps 
\be \O_{(a,b)} = \Gamma(\DDp{{a,b}},\O) \to R\Gamma(\DDpx{{a,b}},\O) \nn\ee
while $d|_{C_1}$ is the diagonal map
\be R\Gamma(\Rect\setminus\xx,\O) \to R\Gamma(\DDpx{{a,b}},\O) .\nn\ee
But presented as in \cref{ccdef}, \cref{cddef}, one recognises the complex $(C^\bul, \dd_C)$ as the \dfn{mapping cocone} $\Cocone(d)$ of the cochain map $d$.  (See e.g. \cite{nlab:mappingcone}, or \cite[Chapter 10]{Weibel}, and note that $\Cocone(d)=\down\Cone(d)$.) In turn the mapping cocone represents the \dfn{homotopy kernel} $\hoker(d)$ 
of the map $d$, so we arrive at the following statement.
\begin{prop}[Following \cite{FHK} Proposition 1.1.4]\label{prop: hoker} The space of global sections of $\O$ on rectilinear space $\Rect$ is the homotopy kernel of the map $d$:
\begin{align} &\Gamma(\Rect,\O) =\nn\\ 
&\quad\hoker\left( R\Gamma(\Rect\setminus\xx,\O) \oplus \bigoplus_{(a,b)\in \xx} \O_{(a,b)} \xrightarrow d \bigoplus_{(a,b)\in \xx} R\Gamma(\DDpx{{a,b}},\O)\right),\nn
\end{align}
\qed\end{prop}

\section{The Thom-Whitney-Sullivan functor}\label{sec: ThomWhitney}
Now we describe the tool we use to produce models of derived spaces of sections which come equipped with the structure of differential graded commutative or Lie algebras. In the literature it goes by the name of the Thom-Sullivan \cite{HS,BoG,Kapranov} or Thom-Whitney \cite{FMM} construction.

\subsection{Polynomial forms on the standard algebro-geometric simplex}\label{sec: omega}
There is a semisimplicial\footnote{It is actually simplicial, i.e. it has degeneracy as well as face maps, but we shall not need this.} commutative differential graded algebra 
\be \Omega:\Deltas^\op \to \dgCAlg \nn\ee 
defined as follows. For each $n\geq 0$, $\Omega([n])$ is the commutative differential graded algebra
\be \Omega([n]) :=  \CC[t_0,\dots,t_n; \dd t_0, \dots \dd t_n]\big/\langle \sum_{i=0}^n t_i -1, \sum_{i=0}^n \dd t_i \rangle \nn\ee 
with $t_i$ in degree $0$ and $\dd t_i$ in degree $1$, for each $i$, and equipped with the usual de Rham differential. One should think of $\Omega([n])$ as the complex of polynomial differential forms on the standard algebro-geometric $n$-simplex. 
For any map $\phi: [n] \to [N]$ of $\Deltas$, 
\be \Omega(\phi) : \Omega([N]) \to \Omega([n]) \nn\ee
is the map of cochain complexes defined by $t_i \mapsto \sum_{j\in \phi^{-1}(i)} t_j$.

\subsection{The functor $\Th$}\label{sec: tw functor}
Suppose $\g : \Deltas \to \dgLie$ is a semicosimplicial differential graded Lie algebra. 
We can construct the bigraded vector space $C^{\bul,\bul}$ whose spaces are
\begin{align} C^{p,q} &= \biggl\{ \mathbf a = (a_m)_{m\geq 0} \in \prod_{m\geq 0} \Omega([m])^p \ox \g([m])^q:\nn\\&\qquad\qquad\qquad
\left(\id \ox \g(\phi)\right) a_n
= \left(\Omega(\phi)\ox \id\right) a_N \quad \text{in}\quad \Omega([n])^p \ox \g([N])^q  \nn\\&\qquad\qquad\qquad
\qquad\text{for all $n\leq N$ and all maps $\phi:[n] \to [N]$ of $\Deltas$} \biggr\}.\label{cpq}
\end{align}
It is a bicomplex, with the de Rham differential
$\dd^p : C^{p,q} \to C^{p+1,q}$
and the internal differential of $\g$,
$\dd^q_\g : C^{p,q} \to C^{p,q+1}$.
The Thom-Whitney complex $(\Th^\bul(\g), \dth)$ is by definition the corresponding total complex
\be \Th^n(\g) = \bigoplus_{p+q=n} C^{p,q} \nn\ee
with differential 
\be \dth :=\dd + \dd_\g .\nn\ee
That is, 
\be \dth (\omega \ox a )= \dd \omega \ox a  + (-1)^{\gr \omega}  \omega \ox \dd_\g a  .\nn\ee 

As a cochain complex $\Th^\bul(\g)$ is quasi-isomorphic to the total complex $\Tot^\bul(\g)$, $\Tot^n(\g) = \bigoplus_{p+q=n} \C^p(\g^q)$, of the unnormalized cochain complex, cf. \cref{sec: cochains}, of the dg Lie algebra $\g$ \cite{Whitney} \cite[\S4]{HLHA}. A quasi-isomorphism  ${\int}: \Th(\g) \to \Tot(\g)$ is defined by integrating over the simplices; see \cite[\S5.2.6]{HS}. In fact an explicit deformation retract 
\be\begin{tikzcd} 
\ar[loop left,""] \Th(\g) \rar[shift left]{\int}  & \lar[shift left]{E} \Tot(\g) 
\end{tikzcd}.\nn\ee
is known \cite{Dupont}; see \cite[\S6]{FMM} and references therein. 
 
The great advantage of the Thom-Whitney complex is that it comes with the structure of a differential graded Lie algebra. The graded Lie bracket is given by
\be [\omega \ox a, \tau \ox b] := (-1)^{\gr a \gr \tau} \omega \wedge \tau \ox [a,b] \nn\ee
for all $a,b \in \g([n])$ and $\omega,\tau \in \Omega([n])$, for each $n$.
One obtains a functor, the Thom-Whitney functor, from semicosimplicial differential graded Lie algebras to differential graded Lie algebras,
\be \Th:[\Deltas, \dgLie] \to \dgLie. \nn\ee
Entirely analogously, one has a functor 
\be \Th: [\Deltas, \dgCAlg] \to \dgCAlg \nn\ee
(which we also denote $\Th$) from semicosimplicial dg commutative algebras to dg commutative algebras.

Any Lie algebra can be regarded as a differential graded Lie algebra concentrated in degree zero, and any commutative algebra can be regarded as a differential graded commutative algebra concentrated in degree zero. So the functors above restrict to functors
\be \Th:[\Deltas, \LieAlg] \to \dgLie,\quad\text{and}\quad \Th: [\Deltas, \CAlg] \to \dgCAlg  \nn\ee
which will actually be all that we need here.

\subsection{Thom-Whitney complex of an $S$-algebra}
Suppose $S$ is a semisimplicial set and 
\be \g: \Pre S \to \dgLie \nn\ee
an $S$-object in differential graded Lie algebras, in the sense of \cref{sec: S algebras}. On composing the the functor $\sca : [\Pre S, \dgLie] \to [\Deltas, \dgLie]$ with the Thom-Whitney functor, we get a functor
\be \Th \circ \sca : [\Pre S, \dgLie] \to \dgLie .\nn\ee
There is an intuitively clear geometrical interpretation of the differential graded Lie algebra $\Th(\g)$. Recall that an $S$-algebra assigns an algebra to each simplex of the semisimplicial set $S$, and specifies maps between them. We can realize $S$ geometrically and consider polynomial differential forms on $S$, with the form on each simplex valued in the corresponding algebra. It is natural to consider forms compatible with the maps between these algebras in the obvious sense.
And indeed we see $C^{p,q}$ of \cref{cpq} becomes 
\begin{align} C^{p,q} &= 
\biggl\{ \mathbf a = (a_{x})_{x\in \bigsqcup_n S([n])} \in \prod_{x\in \bigsqcup_n S([n])} \Omega([\dim x])^p \ox \g(x)^q:\nn\\&\qquad\qquad\qquad
\left(\id \ox \g(\phi)\right) a_x
= \left(\Omega(\phi)\ox \id\right) a_X \quad \text{in}\quad \Omega([\dim x])^p \ox \g(X)^q  \nn\\&\qquad\qquad\qquad
\qquad\text{for all maps $\phi:x \to X$ of $\Pre S$} \biggr\}.\label{cpqs}
\end{align}
That is, an element of $C^{p,q}$ consists of an $\g(x)$-valued polynomial differential form $a_x$ on each simplex of $x$ of $S$, such that whenever $x$ is a simplex of $S$ on the boundary of another simplex $X$ in $S$, then the pullback to $x$ of the form $a_X$ agrees with the image, under the map $\g(x) \to \g(X)$, of the form $a_x$.

If $R$ is a semisimplicial subset of $S$ then we have the morphism of semicosimplicial algebras $\pi: \sca A \to \sca (A|_R)$ of \cref{lem: sca2} and hence, by functoriality of $\Th$, a map
\be \Th(\pi): \Th(A) \to \Th(A|_R) \label{thpi}\ee
This is just the map which pulls back a differential form on $S$ to one on $R$.

\subsection{Conventions for coordinates on simplices}\label{sec: simplex coords}
For us, every 2-simplex corresponds to a flag of the form point
${\pt} \subset (line) \subset \Eta$.
(Compare \cref{D2pics}.) On each individual such simplex, we choose coordinates $(s,t)$ as follows.
\be
\begin{tikzpicture} 
\tdplotsetmaincoords{0}{270}
\begin{scope}[xshift=-5cm,yshift=4cm,scale=3,tdplot_main_coords,local bounding box=D2]
\coordinate (uU) at (0,0,0);
\coordinate (z=u) at (1,0,0)  ;  
\coordinate (p1p1) at (1,-1,0) ;   
\draw[->] (0,.5,0) -- (0,-1.7,0) node[below] {$s$};
\draw[->] (-.5,0,0) -- (1.3,0,0) node[left] {$t$};
\draw[dotted] (1,.5,0) -- (1,-1.5,0) node[right] {$t=1$};
\draw[dotted] (-.5,-1,0) -- (1.5,-1,0) node[left] {$s=1$};
\draw[dotted] (1.5,-1.5,0) -- (-.5,.5,0)  node[above left] {$s=t$};
\draw[thick,black,fill=blue,fill opacity=.1] (p1p1) -- (z=u) -- (uU)  --cycle; 

\node[fill,circle,inner sep=1.5,draw,label= below right:{$\pt$}]   (NuU)   at (uU)  {};
\node[fill,circle,inner sep=1.5,draw,label= below right:{$(line)$}]   (Nz=u)  at (z=u) {};
\node[fill,circle,inner sep=1.5,draw,label= below right:{$\Eta$}] (Np1p1) at (p1p1)  {};
\end{scope}
\end{tikzpicture}\nn\ee
Thus, on each simplex:
\begin{itemize}
\item $\dd s$ is a nonzero constant one-form that vanishes on the edge $\pt \subset (line)$
\item $\dd t$ is a nonzero constant one-form that vanishes on the edge $(line) \subset \Eta$
\item these one-forms agree, $\dd s = \dd t$, on the edge $\pt \subset \Eta$. 
\end{itemize}

\section{Homotopy Manin triples}\label{sec: Manin triples}
As we discussed in the introduction, our main goal in the present work is to give higher generalizations of certain Manin triples which are important in the theory of integrable systems.

To that end we must first clarify what such a generalization of a Manin triple should mean. We begin by expressing the usual definition of a Manin triple of Lie algebras in a form amenable to generalization. 
Recall that we work over $\CC$, here and throughout. A \dfn{Manin triple} $(\a,\a_\pm, \iota_\pm, \pairing --)$ is the data of
\begin{enumerate}
\item Lie algebras $\a$, $\a_+$, $\a_-$, 
\item Lie algebra maps 
$\a_+ \xrightarrow{\iota_+} \a \xleftarrow{\iota_-} \a_-$, and 
\item a map of vector spaces $\pairing -- :\a\ox\a \to \CC$,
\end{enumerate}
subject to the following conditions:
\begin{enumerate}[(i)]
\item the map of vector spaces $(\iota_+,\iota_-): \a_+ \oplus \a_- \to \a$ is an isomorphism.
\item the map $\pairing --: \a\ox\a \to \CC$ is 
\begin{enumerate}[-]
\item symmetric: $\pairing x y  = \pairing yx$ for all $x,y\in \a$. 
\item invariant: $\pairing {[x,y]} z + \pairing y{[x,z]} = 0$ for all $x,y,z\in \a$. 
\end{enumerate}
\item the map $\pairing --: \a\ox\a \to \CC$ is non-degenerate: If $\pairing x - =0$ as maps $\a\to \CC$ then $x=0$. 
\item both $\a_+$ and $\a_-$ are isotropic, i.e. the maps
\be \pairing{\iota_\pm(-)}{\iota_\pm(-)} : \a_\pm \ox \a_\pm \to \CC \nn\ee
are zero. 
\end{enumerate}

Having expressed the definition this way, it seems natural to make the following generalization to dg Lie algebras:

\begin{defn}\label{defn: homotopy manin triple} A \dfn{homotopy Manin triple (of dg Lie algebras)} $(\a, \a_\pm, \iota_\pm, \pairing - -,n)$ is the data of
\begin{enumerate}
\item dg Lie algebras $\a$, $\a_+$ and $\a_-$
\item dg Lie algebra maps 
$\a_+ \xrightarrow{\iota_+} \a \xleftarrow{\iota_-} \a_-$, and
\item a (degree zero) map of dg vector spaces $\pairing{-}{-}:\a \ox \a \to \downby n \CC$
\end{enumerate}
subject to the following conditions:
\begin{enumerate}[(i)]
\item\label{sum} the map of dg vector spaces $(\iota_+,\iota_-): \a_+ \oplus \a_- \to \a$ is a homotopy equivalence
\item\label{symmetric invariant} the map $\pairing --: \a\ox\a \to \downby n\CC$ is 
\begin{enumerate}[-]
\item (graded) symmetric:
 $\pairing x y  = (-1)^{\gr x \gr y }\pairing yx$ for all $x\in \a^{\gr x},y\in \a^{\gr y}$. 
\item invariant: \begin{align} \pairing{[x,y]}{z} + (-1)^{\gr x \gr y} \pairing{y}{[x,z]} &=0\nn\\
 \pairing{ \dd_\a x }{ y} + (-1)^{\gr x} \pairing{x}{\dd_\a y} &=0\nn 
\end{align}
for all $x\in \a^{\gr x}$, $y\in \a^{\gr y}$ and $z\in \a^{\gr z}$. 
\end{enumerate}
\item\label{nondegen} the map $\pairing --: \a\ox\a \to \downby n\CC$ is non-degenerate up to homotopy: If $\pairing x - \simeq 0$ then $x\simeq 0$ (i.e. $x$ is exact).
\item\label{isotropic} both $\a_+$ and $\a_-$ are isotropic, i.e. the maps
\be \pairing{\iota_\pm(-)}{\iota_\pm(-)} : \a_\pm \ox \a_\pm \to \CC \nn\ee
are homotopic to zero. 
\end{enumerate}
\end{defn}

Let us make several remarks about this definition.

\begin{rem}\label{rem: H}
Recall that in the category of dg vector spaces, every quasi-isomorphism is a homotopy equivalence (cf. \cref{sec: retracts}). Consequently, conditions (\ref{sum}) (\ref{nondegen}) and (\ref{isotropic}) respectively could be replaced with the following equivalent demands:
\begin{enumerate}
\item[(i')] the map $(\iota_+,\iota_-)$ induces an isomorphism of graded vector spaces 
\be H(\a_+) \oplus H(\a_-) \cong_{\grVect} H(\a)\nn\ee
\item[(iii')]\label{nondegenprime} the map of graded vector spaces 
\be H(\a) \ox H(\a) \to \downby n \CC \nn\ee
induced by $\pairing--$ is non-degenerate (on the nose).
\item[(iv')] both $H(\a_+)$ and $H(\a_-)$ are isotropic (on the nose) as subspaces of $H(\a)$. 
\end{enumerate}
\end{rem}

\begin{rem}
For a suitable notion of the dual $\a^*$ of $\a$, condition (\ref{nondegen}) is equivalent to 
\begin{enumerate}
\item[(iii'')] The map $\a \to \downby n \a^*$ induced by $\pairing - -$ is a homotopy equivalence.
\end{enumerate} 
(For us $\a$ is often a cochain complex in \emph{topological} vector spaces, and the appropriate dual cochain complex has spaces $\Homcont(\a^n,\CC)$ consisting of the \emph{continuous} linear maps to $\CC$ with its discrete topology. For example, in the case of $\CC[[t]]$ with its $t$-adic topology, one has $\Homcont(\CC[[t]],\CC) \cong t^{-1} \CC[t^{-1}]$, as one wants.)
\end{rem}

\begin{rem}
Let $\dgLie^\circ$ denote the full subcategory of $\dgLie$ whose objects are those dg Lie algebras that are both fibrant and cofibrant in some model structure (for example, in the standard projective model structure on $\dgLie$ induced from the standard projective model structure on $\dgVect$ \cite{HinichHomotopyAlgebras}). 
Recall that in $\dgLie^\circ$ every quasi-isomorphism is a homotopy equivalence, i.e. is invertible up to homotopies. One sees that if $\a,\a_\pm$ and $\b,\b_\pm$ are objects in $\dgLie^\circ$ and $\a\qisom \b$ and $\a_\pm \qisom \b_\pm$, then a Manin triple structure for $\a,\a_\pm$ induces a unique Manin triple structure for $\b,\b_\pm$. 
\end{rem}

\begin{rem}\label{rem: l8}
Kravchenko gives a definition of a Manin \Linfinity-triple, or strongly homotopy Manin triple, in \cite{Kravchenko}. The definitions are compatible in the following sense. 

In this paper we choose to work with dg Lie algebras rather than in the larger category of \Linfinity  algebras. But \cref{defn: homotopy manin triple} goes over to \Linfinity algebras unmodified except that one should generalise the invariance condition on $\pairing --$ to the cyclicity condition for all the brackets $\ell_k(-,\dots,-): \a^{\ox k} \to \upby{2-k} \CC$: 
\be \pairing{\ell_k(x_1,\dots,x_k)}{x_{k+1}} \pm \text{cyclic permutations} =0 .\nn\ee

Assume $(\a,\a_\pm,\iota_\pm,\pairing --)$ is a homotopy Manin triple of \Linfinity algebras in this sense.

Recall -- from e.g. \cite[\S9.4.3--9.4.5]{LodayVallette} or \cite[\S3,\S6]{FMM} and references therein -- that one has a homotopy transfer theorem for \Linfinity algebras: if $\a$ is an \Linfinity algebra then any homotopy equivalence of dg vector spaces $\a\simeq\b$ induces the structure of an \Linfinity algebra on the dg vector space $\b$, in such a way that $\a\simeq \b$ becomes a quasi-isomorphism of \Linfinity algebras. In particular the cohomology $H(\a)$ is a deformation retract of $\a$ in dg vector spaces, so it gets an \Linfinity algebra structure. This structure is compatible with the usual graded Lie algebra structure on $H(\a)$ -- in particular, it has vanishing differential -- but typically has non-vanishing higher brackets. In this way, every \Linfinity algebra $\a$ is quasi-isomorphic to a  \dfn{minimal model}, $H(\a)$. Since minimal models have vanishing differential, two minimal models are quasi-isomorphic precisely if they are isomorphic. 

Under certain conditions on the retract of $\a$ onto $H(\a)$, it is also possible to transfer the cyclic structure: see \cite[Appendix B]{BraunLazarev} and \cite{BraunLazarev2}. 
Assume we are in this situation.
Then one sees that $(H(\a), H(\a_\pm))$ gets the structure of a Manin $L_\8$-triple in the sense of \cite{Kravchenko}, namely: the induced bilinear form is nondegenerate, symmetric and cyclic and
\begin{enumerate}
\item $H(\a_\pm)$ are $L_\8$-subalgebras of $H(\a)$ such that $H(\g) \cong H(\g_+)\oplus H(\g_-)$ as graded vector spaces;
\item $H(\a_+)$ and $H(\a_-)$ are isotropic with respect to the induced bilinear form.
\end{enumerate}
Conversely, any such Manin $L_\8$-triple is clearly, in particular, a homotopy Manin triple in \Linfinity algebras in our sense.
\end{rem}


\section{Local homotopy Manin triple}\label{sec: local manin triple}

With the above definition of a homotopy Manin triple in place, we are ready in this section to give our first main an example of such a structure. Namely, we define a homotopy Manin triple associated to the punctured formal polydisc $\DDx$. See \cref{thm: local mt} below. It can be seen as a higher analog of the Manin triple \cref{yangman} from the introduction.

\subsection{The dg Lie algebra $\gDx$}\label{sec: 2disc}
Let $\g$ be a finite-dimensional simple Lie algebra. Recall from \cref{D2pics} the semisimplicial set $\Flag_\bul(\DDx)$, and from \cref{ADdef} our definition of the $\Flag_\bul(\DDx)$-algebra $\ADx$:
\be\begin{tikzpicture}[baseline =0]
\begin{scope}[xshift=0,scale=5,local bounding box=G]
\coordinate (w=U)  at (-1,0)  ;  
\coordinate (z=u) at (1,0)  ;  
\coordinate (p1p1) at (0,0) ;   
\coordinate ({w=U }{p1p1}) at (barycentric cs:p1p1=1.4,{w=U}=1);
\coordinate ({z=u }{p1p1}) at (barycentric cs:p1p1=1.4,{z=u}=1);
\node (Nw=U)  at(w=U) {$\CC[[w]]\oxC\CC((z))$};
\node (Nz=u) at (z=u) {$\CC((w))\oxC\CC[[z]]$};
\node (Np1p1) at (p1p1) {$\RRR$};
\node (N{w=U }{p1p1}) at ({w=U }{p1p1}) {$\RRR$};
\node (N{z=u }{p1p1}) at ({z=u }{p1p1}) {$\RRR$};
\begin{scope}[every node/.style={midway,circle,inner sep = 0,draw,fill=white}]
\draw[->] (Np1p1) -- (N{z=u }{p1p1}); \draw[->] (Np1p1) -- (N{w=U }{p1p1}); 
\draw[->] (Nz=u) --(N{z=u }{p1p1}); \draw[->] (Nw=U) -- (N{w=U }{p1p1}); 
\end{scope}
\end{scope}
\end{tikzpicture}
\nn\ee
where 
\be \RRR = \CC((w))\oxC\CC((z)). \nn\ee
We shall write
\be 
\gDx := \Th(\g \ox \ADx) \nn\ee
for the resulting differential graded Lie algebra of Thom-Whitney forms.
Explicitly,  $\gDx$ consists of pairs of differential forms:
\begin{align}
\gDx &= \Bigl\{ \bigl(\phi(s) = f(s) + F(s)\dd s,\,\psi(s) = g(s) + G(s)\dd s \bigr) : \nn\\
&\quad\qquad f,F,g,G \in \g \ox \CC((w))\oxC\CC((z))\nn\\
&\quad\qquad  f(1) = g(1) , \nn\\
&\quad\qquad f(0) \in \g \ox \CC[[w]]\oxC\CC((z)),\quad  g(0) \in \g \ox \CC((w))\oxC\CC[[z]] \Bigr\}.\label{gdxelements}
\end{align}
We should think of these forms as painted onto the edges of the semisimplicial set $\Flag_\bul(\DDx)$, cf. \cref{sec: simplex coords}:
\be\begin{tikzpicture} \begin{scope}[xshift=3cm,yshift=4cm,scale=5,local bounding box=D2p]
\coordinate (w=U)  at (-1,0)  ;  
\coordinate (z=u) at (1,0)  ;  
\coordinate (p1p1) at (0,0) ;   
\draw[thick,black] (p1p1) -- node[midway,label=below:{$\phi(s) = f(s) + F(s) \dd s$}] {} (w=U);
\draw[thick,black] (p1p1) -- node[midway,label=below:{$\psi(s) = g(s) + G(s) \dd s$}] {} (z=u);
\node[fill,circle,inner sep=1.5,draw,label= above:{$(z=0)$}]   (Nz=u)  at (z=u) {};
\node[fill,circle,inner sep=1.5,draw,label= above:{$(w=0)$}]   (Nw=U)  at (w=U) {};
\node[fill,circle,inner sep=1.5,draw,label= above:{$\Eta$}] (Np1p1) at (p1p1)  {};
\end{scope}
\end{tikzpicture}
\nn\ee

\subsection{The dg Lie algebras $\gp$ and $\gm$}\label{sec: def gp gm}
Let 
\be \gp := \g \ox \CC[[w]]\oxC\CC[[z]],\nn\ee 
regarded as a dg Lie algebra concentrated in degree zero.
Let 
\be \gm := \Th( \g \ox \ADx^{--} ),\nn\ee 
where we introduce another $\Flag_\bul(\DDx)$-algebra $\ADx^{--}$, given by
\be\begin{tikzpicture}[baseline =0]
\begin{scope}[xshift=0,scale=7,local bounding box=G]
\coordinate (w=U)  at (-1,0)  ;  
\coordinate (z=u) at (1,0)  ;  
\coordinate (p1p1) at (0,0) ;   
\coordinate ({w=U }{p1p1}) at (barycentric cs:p1p1=1,{w=U}=1.4);
\coordinate ({z=u }{p1p1}) at (barycentric cs:p1p1=1,{z=u}=1.4);
\node (Nw=U)  at(w=U) {$0$};
\node (Nz=u) at (z=u) {$0$};
\node (Np1p1) at (p1p1) {$w^{-1}z^{-1}\CC[w^{-1},z^{-1}]$};
\node (N{w=U }{p1p1}) at ({w=U }{p1p1}) {$w^{-1}z^{-1}\CC[w^{-1},z^{-1}]$};
\node (N{z=u }{p1p1}) at ({z=u }{p1p1}) {$w^{-1}z^{-1}\CC[w^{-1},z^{-1}]$};
\begin{scope}[every node/.style={midway,circle,inner sep = 0,draw,fill=white}]
\draw[->] (Np1p1) -- (N{z=u }{p1p1}); \draw[->] (Np1p1) -- (N{w=U }{p1p1}); 
\draw[->] (Nz=u) --(N{z=u }{p1p1}); \draw[->] (Nw=U) -- (N{w=U }{p1p1}); 
\end{scope}
\end{scope}
\end{tikzpicture}
\nn\ee
We have the maps of dg Lie algebras
\be \gp \xrightarrow{\Ip} \gDx \xleftarrow{\Im} \gm \nn\ee
where $\Ip:\gp \to \gDx$ sends $f\in \gp$ to the constant function $f\in \gDx$, and $\Im: \gm \to \gDx$ is the map of dg Lie algebras coming from the canonical embeddings, using \cref{lem: sca1} and functoriality of $\Th$. 

\subsection{The bilinear form}
Now we define on $\gDx$ a symmetric invariant form
\be \pairing {-}{-}: \gDx \ox \gDx \to \down\CC. \nn\ee
(Recall $\down\CC$ is a copy of $\CC$ put into cohomological degree $1$.)
We pick an orientation of $\Flag(\DDx)$: let 
\be \Sigma = \bigl( (w=0), \Eta)\bigr) - \bigl( (z=0), \Eta \bigr) 
\nn\ee 
be the $1$-chain whose boundary is $\del \Sigma = (w=0) - (z=0)$. Observe that, in the notation of \cref{sec: 2disc}, we have
\be \int_\Sigma (\dd s, - \dd s) = \int_{\Delta^1} \dd s + \int_{\Delta^1} \dd s = 2. \nn\ee 
For $a,b\in \g$ and $\omega,\lambda\in \Th(\ADx)$, we set
\be \pairing {a\ox \omega} {b\ox \lambda} := \down \half \bilin a b \int_\Sigma \res_w \res_z \omega \wedge \lambda \nn\ee
where $\bilin --$ denotes the standard invariant symmetric bilinear form on the simple Lie algebra $\g$, and where 
\be \res_t : \CC((t)) \to \CC;\quad \sum_k f_k t^k  \mapsto f_{-1} \nn\ee
is the residue map. Then we extend $\pairing{-}{-}$ by linearity to all of $\gDx\ox\gDx$. 


\subsection{Manin triple}
The main result of this section is then the following.
\begin{thm}\label{thm: local mt} These data 
\be (\gDx,\gp,\gm,\Ip,\Im,\pairing --)\nn\ee 
constitute a homotopy Manin triple in dg Lie algebras, in the sense of \cref{defn: homotopy manin triple}.
\end{thm}
\begin{proof}
Condition (\ref{sum}) is \cref{prop: discx retract} below. For condition (\ref{symmetric invariant}), graded symmetry is clear, and invariance is \cref{prop: inv}. For the nondegeneracy and isotropy conditions (\ref{nondegen}) and (\ref{isotropic}), it is convenient to establish the equivalent statements about cohomologies from \cref{rem: H}. We do so in \cref{prop: pairing}. 
\end{proof}

We start with the following, which is fundamental for us.
\begin{prop}\label{prop: discx retract}
At the level of dg vector spaces, $\gm \oplus \gp$ is a deformation retract of $\gDx$:
\be\begin{tikzcd} 
\gm \oplus \gp \rar[shift left]{I}  & \lar[shift left]{P} \gDx \ar[loop right,"h"] 
\end{tikzcd}\nn\ee
\end{prop}
\begin{proof}
The maps $\Ip$ and $\Im$ of dg Lie algebras define the map of dg vector spaces
\be I= \Ip \oplus \Im: \gm \oplus \gp \to \gDx.\nn\ee 
(Note that it is not a map of dg Lie algebras: the images of $\gp$ and $\gm$ in $\gDx$ are not mutually commuting.) 

We must define a map 
\be P: \gDx \to \gm \oplus \gp  \nn\ee
and a homotopy $h: \gDx \to \gDx$. 

Let 
\be \omega(s) = (\phi(s),\psi(s)) = (f(s) + F(s) \dd s, g(s) + G(s) \dd s)\nn\ee 
be the element of $\gDx$ we introduced above, cf. \cref{gdxelements}.
To define the map $P$ and the homotopy $h$, we first note that we get a unique decomposition
\be \omega = \omega^{++} + \omega^{-+} + \omega^{+-} + \omega^{--} \label{omegadecomp}\ee
coming from the direct sum decomposition of vector spaces
\begin{align} &\CC((z))\oxC\CC((w))\nn\\ &\cong_\CC \CC[[w]]\oxC\CC[[w]] \oplus   \CC[[z]]\oxC w^{-1} \CC[w^{-1}] \oplus z^{-1} \CC[z^{-1}]\oxC\CC [[w]] \oplus w^{-1} z^{-1} \CC[w^{-1},z^{-1}] .\nn
\end{align}
We observe that $\omega^{--}$ may be interpreted as an element of $\gm$, and define 
\be P(\omega) := \left( \omega^{++}|_{s=1} \,,\,\,  \omega^{--}\right) .\label{disc P}\ee
We should then also set $h(\omega^{--}) = 0$, for indeed 
$\omega^{--} = I \circ P (\omega^{--})$
holds exactly.

Next we define
\be h(\omega^{-+})(s) = \left( \int_{0}^{s} F^{-+}(s') \dd s' \,,\,\, \int_{0}^{1} F^{-+}(s') \dd s' + \int_{1}^{s} G^{-+}(s') \dd s' \right) \nn\ee
which we may sketch as 
 \be\begin{tikzpicture} \begin{scope}[xshift=3cm,yshift=4cm,scale=5,local bounding box=D2p,decoration={
    markings,
    mark=at position 0.25 with {\arrow{stealth}},
    mark=at position 0.75 with {\arrow{stealth}},
}]
\coordinate (w=U)  at (-1,0)  ;  
\coordinate (z=u) at (1,0)  ;  
\coordinate (p1p1) at (0,0) ;   
\draw[thick,black] (p1p1) -- node[midway,label=below:{$$}] {} (w=U);
\draw[thick,black] (p1p1) -- node[midway,label=below:{$$}] {} (z=u);
\node[fill,circle,inner sep=1.5,draw,label= above:{$(z=0)$}]   (Nz=u)  at (z=u) {};
\node[fill,circle,inner sep=1.5,draw,label= above:{$(w=0)$}]   (Nw=U)  at (w=U) {};
\node[fill,circle,inner sep=1.5,draw,label= above:{$\Eta$}] (Np1p1) at (p1p1)  {};
\draw[ultra thick,postaction={decorate},blue] ($ (w=U) + .0*(0,-1) $) -- ($ (p1p1)!.6!(z=u)  + .0*(0,-1) $)   ;
\end{scope}
\end{tikzpicture}
\nn\ee
The choice of base point for these integrals is fixed by the following consideration.
The coefficient function $F^{-+}(s)$ of the one-form $F^{-+}(s) \dd s$ obeys no conditions at $s=0$ in general, and yet we need the function $h(F^{-+}(s) \dd s)$ to vanish there, since $ w^{-1} \CC[w^{-1}]\oxC\CC[[z]] \cap \CC[[w]]\oxC\CC((z)) = 0$. 

We then indeed have that
\begin{align} [\dd, h] (f^{-+}(s),g^{-+}(s)) = h \circ \dd (f^{-+}(s),g^{-+}(s)) &= \bigl(f^{-+}(s) , g^{-+}(s) - g^{-+}(1) + f^{-+}(1)\bigr) \nn\\
&=(f^{-+}(s) , g^{-+}(s)) \nn
\end{align}
and 
\be [\dd, h] (F^{-+}(s)\dd s, G^{-+}(s)\dd s) = \dd \circ h  (F^{-+}(s)\dd s, G^{-+}(s)\dd s)  =  (F^{-+}(s)\dd s, G^{-+}(s)\dd s).\nn\ee
That is, $[\dd, h] \omega^{-+} = \omega^{-+}$. 

For the same reasons, for the component $\omega^{+-}$ we are forced to integrate from the other end. We set
\be h(\omega^{+-})(s) = \left( \int_{0}^{1} G^{+-}(s') \dd s' + \int_{1}^{s} F^{+-}(s') \dd s' \,,\,\,  \int_{0}^{s} G^{+-}(s') \dd s' \,,\,\,  \right) \nn\ee
which we may sketch as
\be\begin{tikzpicture} \begin{scope}[xshift=3cm,yshift=4cm,scale=5,local bounding box=D2p,decoration={
    markings,
    mark=at position 0.25 with {\arrow{stealth}},
    mark=at position 0.85 with {\arrow{stealth}},
}]
\coordinate (w=U)  at (-1,0)  ;  
\coordinate (z=u) at (1,0)  ;  
\coordinate (p1p1) at (0,0) ;   
\draw[thick,black] (p1p1) -- node[midway,label=below:{$$}] {} (w=U);
\draw[thick,black] (p1p1) -- node[midway,label=below:{$$}] {} (z=u);
\node[fill,circle,inner sep=1.5,draw,label= above:{$(z=0)$}]   (Nz=u)  at (z=u) {};
\node[fill,circle,inner sep=1.5,draw,label= above:{$(w=0)$}]   (Nw=U)  at (w=U) {};
\node[fill,circle,inner sep=1.5,draw,label= above:{$\Eta$}] (Np1p1) at (p1p1)  {};
\draw[ultra thick,postaction={decorate},blue] ($ (z=u) + .0*(0,-1) $) -- ($ (p1p1)!.3!(w=U)  + .0*(0,-1) $)   ;
\end{scope}
\end{tikzpicture}
\nn\ee
and then we have $[\dd, h] \omega^{+-} = \omega^{+-}$.

For the component $\omega^{++}$ the new feature is that $f^{++}(s)$ can have a non-zero constant term. We set
\be h(\omega^{++})(s) = \left( \int_{1}^{s} F^{++}(s') \dd s' \,,\,\, \int_{1}^{s} G^{++}(s') \dd s' \right) \nn\ee  
which we may sketch as
\be\begin{tikzpicture} \begin{scope}[xshift=3cm,yshift=4cm,scale=5,local bounding box=D2p,decoration={
    markings,
    mark=at position 0.75 with {\arrow{stealth}},
}]
\coordinate (w=U)  at (-1,0)  ;  
\coordinate (z=u) at (1,0)  ;  
\coordinate (p1p1) at (0,0) ;   
\draw[thick,black] (p1p1) -- node[midway,label=below:{$$}] {} (w=U);
\draw[thick,black] (p1p1) -- node[midway,label=below:{$$}] {} (z=u);
\node[fill,circle,inner sep=1.5,draw,label= above:{$(z=0)$}]   (Nz=u)  at (z=u) {};
\node[fill,circle,inner sep=1.5,draw,label= above:{$(w=0)$}]   (Nw=U)  at (w=U) {};
\node[fill,circle,inner sep=1.5,draw,label= above:{$\Eta$}] (Np1p1) at (p1p1)  {};
\draw[ultra thick,postaction={decorate},blue] ($ (p1p1) + .0*(0,-1) $) -- ($ (p1p1)!.3!(w=U)  + .0*(0,-1) $)   ;
\draw[ultra thick,postaction={decorate},blue] ($ (p1p1) + .0*(0,-1) $) -- ($ (p1p1)!.6!(z=u)  + .0*(0,-1) $)   ;
\end{scope}
\end{tikzpicture}
\nn\ee
We then have
\begin{align} [\dd, h] (f^{++}(s),g^{++}(s)) &= h \circ \dd (f^{++}(s),g^{++}(s)) \nn\\ 
&= (f^{++}(s) - f^{++}(1), g^{++}(s) - g^{++}(1)) \nn\\
&=(f^{++}(s) , g^{++}(s)) - ( f^{++}(1), g^{++}(1)) \nn\\
&= (\id - I \circ P) (f^{++}(s),g^{++}(s)) \nn
\end{align}
and
\be [\dd, h] (F^{++}(s)\dd s, G^{++}(s)\dd s) = \dd \circ h  (F^{++}(s)\dd s, G^{++}(s)\dd s)  =  (F^{++}(s)\dd s, G^{++}(s)\dd s) .\nn\ee
This completes the proof that $[\dd, h] = \id - I \circ P$. It is clear on inspection that the relation $\id = P \circ I$ holds exactly. 
\end{proof}

The next statement shows that the cohomology of $\gm$ is a copy of $ \g \ox w^{-1} z^{-1} \CC[w^{-1},z^{-1}] $ in cohomological degree one.
\begin{prop}\label{prop: H1 retract}
There is a deformation retract dg vector spaces
\be\begin{tikzcd} 
\down \g \ox w^{-1} z^{-1} \CC[w^{-1},z^{-1}] \rar[shift left]{i}  & \lar[shift left]{p} \gm \ar[loop right,"h"] 
\end{tikzcd}\nn\ee
\end{prop}
\begin{proof}
We let 
\be i(\down a) := \half a ( \dd s, - \dd s ), \nn\ee
\be p\bigl( (f(s) + F(s) \dd s, g(s) + G(s) \dd s) \bigr) := \down \int_0^1 (F(s') - G(s')) \dd s' \nn\ee
and finally
\begin{align} &h\bigl( (f(s) + F(s) \dd s, g(s) + G(s) \dd s) \bigr) \nn\\
&:= \left( \int_0^s F(s') \dd s' + \half s \int_0^1 (F(s') - G(s') \dd s'),\, \int_0^s G(s') \dd s' -  \half s \int_0^1 (F(s') - G(s') \dd s')\right). \nn
\end{align}
Then $p(i(\down a)) = \down \int_0^1 a \dd s' = \down a \int_0^1 \dd s' = \down a$, so $p\circ i = \id$, and 
\begin{align} 
&   (h\circ \dd)\bigl(  (f(s) + F(s) \dd s, g(s) + G(s) \dd s) \bigr) \nn\\
&=   h\bigl( ( f'(s) \dd s, g'(s) \dd s) \bigr) \nn\\
&=   \bigl(f(s) + \half s( f(1) - f(0) - g(1) + g(0) ) , \nn\\
&\qquad     g(s) - \half s( f(1) - f(0) - g(1) + g(0) ) \bigr) = (f(s), g(s)), \nn
\end{align} 
since $f(1) = g(1)$ and $f(0) = 0 = g(0)$, while
\begin{align}
&   (\dd \circ h)\bigl(  (f(s) + F(s) \dd s, g(s) + G(s) \dd s) \bigr) \nn\\
&=   \dd\left( \int_0^s F(s') \dd s' + \half s \int_0^1 (F(s') - G(s') \dd s'),\, \int_0^s G(s') \dd s'  - \half s \int_0^1 (F(s') - G(s') \dd s'\right). \nn\\
&=   \bigl( F(s) \dd s, G(s) \dd s)  + \half \int_0^1 (F(s') - G(s') \dd s') (\dd s , - \dd s) \nn
\end{align}
and here we recognize the last term as $i(p( \bigl( (f(s) + F(s) \dd s, g(s) + G(s) \dd s) \bigr) ))$, so that 
\be \id - i \circ p = [\dd,h] \nn\ee
as required.
\end{proof}
On combining this with \cref{prop: discx retract}, we get the following corollary.
\begin{cor}\label{cor: cohomology} At the level of graded vector spaces
\begin{align} H^k(\gDx) = \begin{cases} \g\ox\CC[[w]]\oxC\CC[[z]] & k =0 \\ \g \ox w^{-1} z^{-1} \CC[w^{-1}, z^{-1}] & k = 1 
\\ 0 & k \notin \{ 0,1\}
  \end{cases} \nn\end{align}
\end{cor}

\begin{rem}
Since $H^\bul(\gDx)$ is a retract in of $\gDx$ in $\dgVect$, it is possible to endow it with the structure of an \Linfinity algebra via the homotopy transfer theorem, c.f. \cref{rem: l8}.
\end{rem}

\begin{prop}\label{prop: inv}
The symmetric bilinear form $\pairing{-}{-}$ is $\gDx$-invariant, i.e.
\begin{align} \pairing{[x,y]}{z} + (-1)^{\gr x \gr y} \pairing{y}{[x,z]} &=0\nn\\
 \pairing{ \dd x }{ y} + (-1)^{\gr x} \pairing{x}{\dd y} &=0\nn 
\end{align}
for all $x,y,z\in \gDx$.
\end{prop}
\begin{proof}
The first part is clear, given the $\g$-invariance of $\bilin--$. 
For the second part, we recall that $\dd (a\ox \omega) = a\ox \dd \omega$ in our setting, and we have
\begin{align} 
&\pairing {a\ox \dd \omega} {b\ox \lambda} + (-1)^{\gr{a\ox \omega}} \pairing{a\ox\omega}{b \ox \dd\lambda} \nn\\
&= \down \half \bilin{a}{b} \int_\Sigma\res_w\res_z \dd \left(\omega \wedge \lambda\right) \nn\\
&= \down \half \bilin{a}{b} \res_w\res_z \left(\left(\omega \wedge \lambda\right)|_{w=0} - \left(\omega\wedge\lambda\right)|_{z=0}\right). \nn
\end{align}
Obviously both the pullbacks here $(\omega\wedge\lambda)|_{w=0}$ and $(\omega\wedge\lambda)|_{z=0}$ receive contributions only from the degree zero components of $\omega$ and $\lambda$. Both vanish after taking the residues $\res_w\res_z$, because $\omega$ and $\lambda$ obey the boundary conditions in \cref{gdxelements}.  
\end{proof}


\begin{prop}\label{prop: pairing}
The pairing $\pairing -- $ induces a non-degenerate pairing
\be H(\gDx) \ox H(\gDx) \to \down \CC, \nn\ee
with respect to which both $H(\gp)$ and $H(\gm)$ are isotropic.
\end{prop}
\begin{proof}
Recall that $H^0(\gDx) = \gp := \g\ox \CC[[w,z]]$ and the map $\Ip:\gp \to \gDx$ of \cref{prop: discx retract} maps elements of $\gp$ to constant $\gp$-valued $0$-forms on $\Flag\DDx$.
Similarly, the map 
\be H^1(\gDx) \cong \down \g \ox w^{-1} z^{-1} \CC[w^{-1},z^{-1}] \xrightarrow{i} \gm \xrightarrow{\Im} \gDx\nn\ee from \cref{prop: H1 retract} and \cref{prop: discx retract} maps elements of $H^1(\gDx)$ to multiples of the $1$-form $(\dd s,-\dd s)$. 
The pairing $\pairing{-}{-}$ restricts to a manifestly non-degenerate pairing between them:
\begin{align} H^0(\gDx) \ox  H^1(\gDx) &\to \down \CC\nn\\
\quad (x,y) &\mapsto \pairing{\Ip(x)}{\Im(i(y))} \nn
\end{align}
while $H^0(\gDx)$ and $H^1(\gDx)$ are isotropic. (This establishes the result, because $H(\gp) = H^0(\gDx)$ and $H(\gm) = H^1(\gDx)$.)  
\end{proof}
This completes the proof of \cref{thm: local mt}.
  
\subsection{The unpunctured polydisc}
Now we turn to the dg Lie algebra for the unpunctured formal polydisc: 
\be \gD := \Th(\g \ox \AD). \nn\ee
The next statement implies in particular that $\gD$ and $\gp:= \g \ox \CC[[w]]\oxC\CC[[z]]$ are quasi-isomorphic. The proof is an instructive warm-up for the global cases in the next section.
\begin{prop}\label{prop: disc retract}
There is a deformation retract of dg vector spaces
\be\begin{tikzcd} 
\gp \rar[shift left]{I}  & \lar[shift left]{P} \gD \ar[loop right,"h"] 
\end{tikzcd}\nn\ee
in which both $I$ and $P$ are maps of dg Lie algebras. 

Moreover, we may choose the map $P$ to be given by pull-back of the non-singular part to the $0$-simplex $\Eta$:
\be P(\omega) := \omega^{++}|_{\Eta} \nn\ee
\end{prop}
\begin{proof}
It is helpful to sketch the idea first:
\be
\begin{tikzpicture} 
\tdplotsetmaincoords{0}{180}
\begin{scope}[rotate=45,xshift=-5cm,yshift=4cm,scale=2,tdplot_main_coords,local bounding box=D2,decoration={
    markings,
    mark=at position 0.5 with {\arrow{stealth}},
    mark=at position 0.75 with {\arrow{stealth}},
}]
\coordinate (uU) at (1,-1,0);
\coordinate (w=U)  at (1,0,0)  ;  
\coordinate (z=u) at (0,-1,0)  ;  
\coordinate (p1p1) at (0,0,0) ;   
\draw[thick,black,fill=blue,fill opacity=.1] (p1p1) -- (w=U) -- (uU)   --cycle; 
\draw[thick,black,fill=blue,fill opacity=.1] (p1p1) -- (z=u) -- (uU)  --cycle; 

\node[fill,circle,inner sep=1.5,draw,label= left:{$\{(0,0)\}$}]   (NuU)   at (uU)  {};
\node[fill,circle,inner sep=1.5,draw,label= right:{$(z=0)$}]   (Nz=u)  at (z=u) {};
\node[fill,circle,inner sep=1.5,draw,label= left:{$(w=0)$}]   (Nw=U)  at (w=U) {};
\node[fill,circle,inner sep=1.5,draw,label= right:{$\Eta$}] (Np1p1) at (p1p1)  {};
\draw[very thick,postaction={decorate},blue,shorten >= 2ex] (uU) -- ($(uU)!.3!(z=u)$)  -- ($(p1p1)!.3!(z=u)$)  ;
\draw[very thick,postaction={decorate},blue,shorten >= 2ex] (uU) -- ($(uU)!.7!(z=u)$)  -- ($(p1p1)!.7!(z=u)$)  ;
\draw[very thick,postaction={decorate},blue,shorten >= 2ex] (uU) -- ($(uU)!.6!(w=U)$)  -- ($(p1p1)!.6!(w=U)$)  ;
\end{scope}
\end{tikzpicture}
\ee
An element of $\gD$ is given by a pair of forms, one on each of the two 2-simplices of $\DD$,
\be \omega = ( \Phi, \Psi ) \nn\ee
valued in $\g\ox \CC((w))\oxC\CC((z))$. We give the definition of $h$ on $\Phi$; the definition on $\Psi$ is similar. We choose to use the coordinates $s$ and $u = t-s$, \cref{sec: simplex coords}. We have
\be \Phi(s,u) = f(s,u) + f_s(s,u) \dd s + f_u(s,u) \dd u + f_{su}(s,u) \dd s \wedge \dd u \nn\ee
for some coefficient functions. We define
\begin{align} h(\Phi)(s,u) &:= \left(\int_0^s f_s(s',u) \dd s'\right) + \left(\int_0^s f_{su}(s',u) \dd s'\right) \dd u \nn\\
&\quad + \left( \int_0^u f_u(0,u') \dd u'\right) \label{hdef}
\end{align}

Again, the choice of base point for these integrals is fixed by the fact that $\int_0^s f_s(s',u) \dd s'$ must obey a boundary condition at $s=0$ while $f_s(s,u)$ obeys no such condition (since $\dd s$ was zero on $s=0$). Note also that on the ``far edge'', i.e. $(z=0) \subset \Eta$, the boundary condition is empty since the space attached to this edge, $\g \ox \CC((w))\oxC\CC((z))$, is the same as that attached to the 2-simplex it borders. 

We then compute
\begin{align} (\dd\circ h(\Phi))(s,u) &= f_s(s,u) \dd s + \left(\int_0^s \del_u f_s(s',u) \dd s'\right) \dd u 
                                       + f_{su}(s,u) \dd s \wedge \dd u\nn\\
&\quad + f_u(0,u) \dd u \nn  
\end{align} 
while
\begin{align} (h\circ \dd(\Phi))(s,u) &=  h\Bigl(  (\del_s f)(s,u) \dd s + (\del_u f)(s,u) \dd u + (\del_s f_u - \del_u f_s) \dd s \wedge \dd u \Bigr) \nn\\ 
&= f(s,u) - f(0,u) + f_u(s,u) \dd u - f_u(0,u) \dd u -  \left(\int_0^s \del_u f_s(s',u) \dd s'\right) \dd u \nn\\
& \quad + f(0,u) - f(0,0) \nn
\end{align} 
so that
\begin{align} ([\dd,h] (\Phi))(s,u) &= \Phi(s,u) - f(0,0).  \nn
\end{align}
Then we define $P: \gD \to \gp$ to be the map $\omega \mapsto f(0,0) = \omega|_{\{(0,0)\}}$, picking out the pullback of $\omega$ to the vertex $\{(0,0)\}$, and $I: \gp \to \gD$ to be the embedding of an element of $\gp$ as a constant function. This ensures that $P \circ I = \id$ and that the equality above becomes
\be   [\dd,h] = \id - I \circ P, \nn\ee
which completes the proof that $\gp$ is a retract of $\gD$.

It remains to establish the ``moreover'' part of the proposition. Above, we chose to set $P(\omega) = \omega|_{\{(0,0)\}}$. 
In what follows, it will be helpful to note the following alternative choice for the maps $I,P,h$. We keep the definition of $I$. We let $P:\gD \to \gp$ be the map $\omega\mapsto f(s=1,u=0)^{++}= \omega^{++}|_{\Eta}$, picking out the pullback of the regular part of $\omega$ at the vertex $\Eta$. And we let $h$ be given by
\be h_{new}(\Phi) = h_{old}(\Phi) + \int_1^0 f_s(s',0)^{++} \dd s'\nn\ee
Then $(\dd\circ h(\Phi))(s,u)$ is unaltered (we have added a constant), while
while $(h\circ \dd(\Phi))(s,u)$ receives the extra term
\be \int_{1}^0 \del_s f(s',0)^{++} \dd s' = f(0,0)^{++} - f(1,0)^{++} = f(0,0) - f(1,0)^{++} \nn\ee
so that 
\begin{align} ([\dd,h] (\Phi))(s,u) &= \Phi(s,u) - f(s=1,u=0)^{++}  \nn
\end{align}
as we now need with our new definition of $P$. 
\end{proof}
\subsection{Pictorial notation for homotopies}\label{sec: pictures}
The proof above is a prototype for the sort of computation we shall need in many places below. Let us introduce an obvious pictorial notation, in which the homotopy $h$ from \cref{hdef} is denoted simply as
\be
\begin{tikzpicture} 
\tdplotsetmaincoords{0}{180}
\begin{scope}[rotate=45,xshift=-5cm,yshift=4cm,scale=1.5,tdplot_main_coords,local bounding box=D2,decoration={
    markings,
    mark=at position 0.5 with {\arrow{stealth}},
}]
\coordinate (uU) at (1,-1,0);
\coordinate (w=U)  at (1,0,0)  ;  
\coordinate (z=u) at (0,-1,0)  ;  
\coordinate (p1p1) at (0,0,0) ;   
\draw[fill=blue,fill opacity=.1] (p1p1) -- (w=U) -- (uU)   --cycle; 
\draw[fill=blue,fill opacity=.1] (p1p1) -- (z=u) -- (uU)   --cycle; 

\node[inner sep=1.5,label= left:{$\{(0,0)\}$}]   (NuU)   at (uU)  {};
\node[inner sep=1.5,label= above:{$(z=0)$}]   (Nz=u)  at (z=u) {};
\node[inner sep=1.5,label= below:{$(w=0)$}]   (Nw=U)  at (w=U) {};
\node[inner sep=1.5,label= right:{$\Eta$}] (Np1p1) at (p1p1)  {};

\draw[very thick,opacity=.3,postaction={decorate},blue] ($(uU)$)  -- ($(z=u)$) ;
\draw[very thick,opacity=.3,postaction={decorate},blue] ($(uU)$)  -- ($(w=U)$) ;
\draw[very thick,opacity=.3,postaction={decorate},blue] ($(uU)$)  -- ($(p1p1)$) ;
\foreach \r in {0.25,0.5,0.75}
{
\draw[very thick,opacity=.3,postaction={decorate},blue] ($(uU)!\r!(z=u)$)  -- ($(p1p1)!\r!(z=u)$) ;
\draw[very thick,opacity=.3,postaction={decorate},blue] ($(uU)!\r!(w=U)$)  -- ($(p1p1)!\r!(w=U)$) ;
}
\end{scope}
\end{tikzpicture}
\nn\ee
We shall occasionally also refer to $h$, somewhat loosely, as a \dfn{retract to the vertex $(0,0)$.} The precise paths we chose for the integrals were not crucial. One may verify that it is equally possible to make the choice
\be
\begin{tikzpicture} 
\tdplotsetmaincoords{0}{180}
\begin{scope}[rotate=45,xshift=-5cm,yshift=4cm,scale=1.5,tdplot_main_coords,local bounding box=D2,decoration={
    markings,
   mark=at position 0.75 with {\arrow{stealth}},
}]
\coordinate (uU) at (1,-1,0);
\coordinate (w=U)  at (1,0,0)  ;  
\coordinate (z=u) at (0,-1,0)  ;  
\coordinate (p1p1) at (0,0,0) ;   
\draw[fill=blue,fill opacity=.1] (p1p1) -- (w=U) -- (uU)   --cycle; 
\draw[fill=blue,fill opacity=.1] (p1p1) -- (z=u) -- (uU)   --cycle; 

\node[inner sep=1.5,label= left:{$\{(0,0)\}$}]   (NuU)   at (uU)  {};
\node[inner sep=1.5,label= above:{$(z=0)$}]   (Nz=u)  at (z=u) {};
\node[inner sep=1.5,label= below:{$(w=0)$}]   (Nw=U)  at (w=U) {};
\node[inner sep=1.5,label= right:{$\Eta$}] (Np1p1) at (p1p1)  {};

\draw[very thick,opacity=.3,postaction={decorate},blue] ($(uU)$)  -- ($(z=u)$) ;
\draw[very thick,opacity=.3,postaction={decorate},blue] ($(uU)$)  -- ($(w=U)$) ;
\draw[very thick,opacity=.3,postaction={decorate},blue] ($(uU)$)  -- ($(p1p1)$) ;
\foreach \r in {0.25,0.5,0.75}
{
\draw[very thick,opacity=.3,postaction={decorate},blue] ($(uU)!\r!(p1p1)$)  -- ($(z=u)!\r!(p1p1)$) ;
\draw[very thick,opacity=.3,postaction={decorate},blue] ($(uU)!\r!(p1p1)$)  -- ($(w=U)!\r!(p1p1)$) ;
}
\end{scope}
\end{tikzpicture}
\nn\ee
which amounts to replacing the definition \cref{hdef} of $h$ with
\begin{align}  & \left( \int_{0}^u f_u(s,u') \dd u'\right) - \left(\int_{0}^u f_{su}(s,u') \dd u'\right) \dd s \nn\\
   & \quad + \left(\int_0^s f_s(s',0) \dd s' \right).\nn 
\end{align}
In what follows, we shall not write out such integrals explicitly, and we shall suppress the details of routine manipulations similar to those in the proof above.


\section{Global homotopy Manin triple}\label{sec: global manin triple}
In this section we give the second of our two main examples of homotopy Manin triples: see \cref{thm: global mt}. We shall use throughout the pictorial notation for homotopies introduced in \cref{sec: pictures} above. 

The Manin triples of this section are defined by a collection of marked points in rectilinear space, as we now describe.

\subsection{Marked points}
We continue to let $w,z: \CC\times\CC \to \CC$ be the Cartesian coordinates.
Pick $N\geq 1$. 
Let $\zz = z_1,\dots,z_N$ be pairwise distinct points in $\CC$. 
Let $\ww = w_1,\dots,w_N$ be pairwise distinct points in $\CC$. 
Let $\Rectn$ denote the subset of $\Rect$ consisting of
\begin{itemize}
\item The closed points $(w_i,z_j)$ for all $i\neq j$, $i,j\in \{1,\dots,N\}$
\item The lines  $(w=w_i)$ for $i\in \{1,\dots,N\}$
\item The lines  $(z=z_i)$ for $i\in \{1,\dots,N\}$
\item The generic point $\Eta$.
\end{itemize}
We may sketch these data as follows:
\be\begin{tikzpicture}[baseline=-50]
\begin{scope}[scale=.3,local bounding box = Global]
\fill[gray!15] (1,-1) rectangle (16,11); 
\draw (0,0) -- (17,0) node[right] (z0) {$z=z_1$};
\draw (0,10) -- (17,10) node[right] (z8) {$z=z_N$};
\draw (0,3) -- (17,3) node[right] (zu) {$z=z_2$};
\draw (2,12) -- (2,-2) node[below] (wU) {$w=w_1$};
\draw (6.5,12) -- (6.5,-2) node[below] (w1) {$w=w_2$};
\draw (15,12) -- (15,-2) node[below] (w1) {$w=w_N$};
\draw (10,-2) node[below] () {$\dots$};
\draw (17,7.5) node[right] () {$\vdots$}; 
\draw[fill=black]  (2,3) circle (.2);\draw[fill=black]  (2,10) circle (.2);\draw[fill=white]  (2,0) circle (.2);
\draw[fill=white]  (6.5,3) circle (.2);\draw[fill=black]  (6.5,10) circle (.2);\draw[fill=black]  (6.5,0) circle (.2);
\draw[fill=black]  (15,3) circle (.2);
\draw[fill=white]  (15,10) circle (.2);
\draw[fill=black]  (15,0) circle (.2);
\end{scope}
\end{tikzpicture}\nn\ee
This gives rise to the semisimplicial subset  $\Flag_\bul(\Rectn)$ of $\Flag_\bul(\Rect)$. In contrast to the latter, it has finite sets of $n$-simplices, for each $n$.


\subsection{Sketches of $\Flag_\bul(\Rectn)$}\label{sec: sketches}

It is helpful to be able to visualise at least parts of the semisimplicial set  $\Flag_\bul(\Rectn)$. 
For example we may restrict our attention to the following lines and points, and draw the corresponding simplices of $\Flag_\bul(\Rectn)$:
\be\begin{tikzpicture}[baseline=-50]
\begin{scope}[scale=.3,local bounding box = Global]
\fill[gray!15] (1,-1) rectangle (10,4); 
\draw (0,0) -- (11,0) node[right] (z0) {$z=z_1$};
\draw (0,3) -- (11,3) node[right] (zu) {$z=z_3$};
\draw (2,5) -- (2,-2) node[below=-.1cm] (wU) {$w=w_1$};
\draw[] (9,5) -- (9,-2) node[below] (w1) {$w=w_2$};
\draw[fill=black]  (2,3) circle (.2);
\draw[fill=white]  (2,0) circle (.2);
\draw[fill=black]  (9,3) circle (.2);
\draw[fill=black]  (9,0) circle (.2);
\end{scope}
\end{tikzpicture}\qquad\begin{tikzpicture}
\tdplotsetmaincoords{0}{180}
\begin{scope}[scale=1.5,tdplot_main_coords,local bounding box=G]
\threefourpiconetwothree
\end{scope}
\end{tikzpicture}\nn\ee
At least in the case of three marked points, it is actually possible draw the whole semisimplicial set  $\Flag_\bul(\Rect(3))$:
\Rgpicsthree

\subsection{Global dg Lie algebra $\gN$}\label{sec: gNdef}
Let $\ARN$ be the $\Flag_\bul(\Rectn)$-object in commutative algebras given as follows.
Let us write $\CC(w)_{\ww}^\8$ for the $\CC$-algebra
\be \CC(w)_{\ww}^\8:= \left\{\parbox{0.66\textwidth}{rational expressions in $w$ vanishing at $\8$ and with poles at most at the points $\ww= \{w_1,\dots,w_N\}$}\right\}.\nn\ee 
We assign, to every simplex of $\Flag_\bul(\Rectn)$, the commutative algebra
\be \CC(w)_{\ww}^\8 \oxC \CC(z)_{\zz}^\8 \nn\ee
with only the following exceptions: 
\be \ARN( \{(w_i,z_j)\} ) := \CC(w)_{\ww\setminus\{w_i\}}^\8 \oxC \CC(z)_{\zz \setminus \{z_j\}}^\8 \nn\ee
\be \ARN( \{(w_i,z_j)\} \subset \{z=z_j\}) := \ARN( \{ z=z_j\}) := \CC(w)_{\ww}^\8 \oxC \CC(z)_{\zz \setminus\{z_j\}}^\8 \nn\ee
\be \ARN( \{(w_i,z_j)\} \subset \{w=w_i\}) := \ARN( \{ w=w_i\}) := \CC(w)_{\ww\setminus\{w_i\}}^\8 \oxC \CC(z)_{\zz}^\8\nn\ee
for all $i,j$ such that the given flag belongs to $\Flag_\bul(\Rectn)$. 

We continue to let $\g$ be a finite-dimensional simple Lie algebra.
Let $\gN$ denote the dg Lie algebra 
\be \gN := \Th(\g \ox \ARN).\nn\ee
Note that the exceptions above are precisely the simplices on the boundary of $\Flag_\bul(\Rectn)$. Thus, concretely, $\gN$ is the dg algebra of polynomial differential forms on $\Flag_\bul(\Rectn)$, valued in $\g\ox \CC(w)_{\ww}^\8 \CC(z)_{\zz}^\8$, subject to these boundary conditions. 

\subsection{Local dg Lie algebras}
Let us introduce the dg Lie algebras
\be \gloc := \bigoplus_{i=1}^N \gp(w_i,z_i) = \bigoplus_{i=1}^N \g \ox \CC[[w-w_i,z-z_i]],\qquad \glocx := \bigoplus_{i=1}^N \gDpx {{w_i,z_i}} \nn\ee
and define two maps of dg Lie algebras,
\be \iglob: \gN \to \glocx, \qquad \iloc: \gloc \to \glocx. \nn\ee
The map $\iloc$ is given, summand by summand, as in \cref{sec: 2disc} above. To define $\iglob$, observe that $\Flag_\bul(\DDpx {{w_i,z_i}})$ is a semisimplicial subset of $\Flag_\bul(\Rectn)$, for each $i=1,\dots,N$. We get the restriction of $\ARN$ to a $\Flag_\bul(\DDpx {{w_i,z_i}})$-algebra $\ARN|_{\Flag_\bul(\DDpx 0)}$, and there is an evident map of $\Flag_\bul(\DDpx {{w_i,z_i}})$-algebras 
\be \ARN|_{\Flag_\bul(\DDpx {{w_i,z_i}})} \to \ADpx {{w_i,z_i}}\nn\ee
given by taking formal Laurent expansions. 
Hence, by \cref{lem: sca1} and \cref{lem: sca2} we get the map of semicosimplicial algebras
\be \sca \ARN \to \sca \ARN|_{\Flag_\bul(\DDpx {{w_i,z_i}})} \to \sca \ADpx {{w_i,z_i}}\nn\ee
and therefore, by the functoriality of $\Th$, a map of dg Lie algebras 
\be \iglob^i : \gglob \to \gDpx {{w_i,z_i}}\label{iglobi}\ee
for each marked point $(w_i,z_i)$. 
The resulting diagonal map is $\iglob$:
\be \iglob := \bigoplus_{i=1}^N \iglob^i \nn\ee

\subsection{The pairing}\label{sec: pairing}
For each marked point $(w_i,z_i)$, we have an invariant bilinear form $\pairing{-}{-}^i$ on $\gDpx{{w_i,z_i}}$ defined as in \cref{sec: pairing} and \cref{prop: pairing}. 
Let us use the same notation $\pairing{-}{-}$ for the resulting diagonal pairing 
\be \pairing {A} {B} := \sum_{i=1}^N \pairing{A_i}{B_i}^i \nn\ee
between two elements $A= (A_i)$ and $B=(B_i)$ of  $\glocx = \bigoplus_{i=1}^N\gDpx{{w_i,z_i}}$.

\subsection{Manin triple} The main result of this section is the following.
\begin{thm}\label{thm: global mt} The data 
\be (\glocx,\gloc,\gglob,\iloc,\iglob,\pairing --)\nn\ee 
constitute a homotopy Manin triple in dg Lie algebras, in the sense of \cref{defn: homotopy manin triple}.
\end{thm}
\begin{proof}
Most of the rest of this section is devoted to establishing condition (\ref{sum}), the homotopy equivalence of $\gglob\oplus \gloc \simeq \glocx$, which is \cref{thm: he} below. For condition (\ref{symmetric invariant}) and (\ref{nondegen}) there is essentially nothing new to check, given \cref{prop: inv} and \cref{prop: pairing}. For the isotropy condition (\ref{isotropic}), we once more choose to establish the equivalent statements about cohomologies from \cref{rem: H}. We do so in \cref{sec: cohomology and pairing}. 
\end{proof}

We have the map from the direct sum of $\gglob$ and $\gloc$ as dg vector spaces:
\be I = (\iglob,\iloc) : \gglob \oplus \gloc \to \glocx .\nn\ee
Let us stress that this is not a map of dg Lie algebras from the direct sum of $\gglob$ and $\gloc$ as dg Lie algebras, for the same reason that the analogous map in the usual one-dimensional case is not a map of Lie algebras: the images of $\gglob$ and $\gloc$ in $\glocx$ are not mutually commuting.

The following statement justifies our definition of $\gN$: it establishes that, up to homotopies, $\gN$ provides a dg vector space complement to $\gloc$ in $\glocx$.

\begin{thm}\label{thm: he}
This map $I$ is a homotopy equivalence of dg vector spaces. That is, there is a homotopy equivalence of dg vector spaces:
\be\begin{tikzcd}[column sep = huge]
\ar[loop left,"\hglob+\hoffdiag"] \gN \oplus \gloc \rar[shift left]{I}  & \lar[shift left]{P} \glocx \ar[loop right,"\hlocx"] 
\end{tikzcd}\nn\ee
\end{thm}
\begin{proof}
We shall construct a map of cochain complexes 
\be \ptot = \pglob \oplus \ploc : \glocx \to \gglob \oplus \gloc \nn\ee
inverse to $\itot$ up to homotopies. We shall first define $\pglob$ and then check that $\pglob \circ \iglob$ is homotopic to the identity on $\gglob$. Then we shall define $\ploc$ and check the remaining homotopy relations. 

Our first step is to define the map $\pglob$. 
Morally, the idea here in our case in dimension two is the same as in the case of dimension one from, e.g., \cite{FFR}: we want to use the ``singular parts'' of an element of $\glocx$ to construct an element of $\gglob$.

An element of $\omega \in \glocx$ is a tuple $\omega = (\omega_{i})_{i=1}^N$, $\omega_{i} \in \gDpx{{w_i,z_i}}$, and we shall define $\pglob(\omega)$ summand by summand,
\be \pglob(\omega) := \sum_{i=1}^N \pglob^i(\omega_i) ,\qquad \pglob^i: \gDpx{{w_i,z_i}} \to \gglob.\nn\ee
Given $\omega_{i}\in \gDpx{{w_i,z_i}}$ we have, just in \cref{omegadecomp}, the decomposition
\be \omega_{i} = \omega_{i}^{++} + \omega_{i}^{-+} + \omega_{i}^{+-} + \omega_{i}^{--} \label{omegaidecomp}\ee
coming from the decomposition of the vector space $\CC((w-w_i))\oxC\CC((z-z_i))$ into the parts polar/regular parts with respect to each of the local coordinates., $w-w_i$ and $z-z_i$.
In particular, the ${}^{--}$ part can be interpreted as a rational function vanishing at infinity, via the embedding of commutative algebras
\be (w-w_i)^{-1} (z-z_i)^{-1} \CC[(w-w_i)^{-1},(z-z_i)^{-1}] \into 
\CC(w)_{\ww}^\8 \oxC \CC(z)_{\zz}^\8.\nn\ee
Now, we have
\be \omega_i^{--}(s) = (\phi_i^{--}(s), \psi_i^{--}(s)) \nn\ee
as in \cref{gdxelements}, and we define $\pglob^i(\omega_i)$ to be the element given as follows,
\be\label{pglob} \pglob^i(\omega_i) := 
\begin{tikzpicture}[baseline=0]
\tdplotsetmaincoords{0}{180}
\begin{scope}[xshift=-5cm,scale=2.2,tdplot_main_coords,local bounding box=G]
\threefourpicikl
\node (N{0wi }{w=wi}{p1p1}) at ({0wi }{w=wi}{p1p1}) {$\psi_i^{--}(1)$}; 
\node (N{0wi }{z=0 }{p1p1}) at ({0wi }{z=0 }{p1p1}) {$\phi_i^{--}(1)$}; 
\node (N{0U  }{z=0 }{p1p1}) at ({0U  }{z=0 }{p1p1}) {$\phi_i^{--}(t)$};
\node (N{0U  }{w=U }{p1p1}) at ({0U  }{w=U }{p1p1}) {$\phi_i^{--}(s)$};
\node (N{uwi }{z=u }{p1p1}) at ({uwi }{z=u }{p1p1}) {$\psi_i^{--}(s)$}; 
\node (N{uwi }{w=wi}{p1p1}) at ({uwi }{w=wi}{p1p1}) {$\psi_i^{--}(t)$}; 
\end{scope}
\end{tikzpicture}
\ee
for each $k,\ell \neq i$. If $k=\ell$ then the lower right square is absent. 

Observe that this obeys the continuity conditions on internal edges: in particular, along the edges of the form $\bigl( (w_k,z_\ell) ,  \Eta\bigr)$ it is continuous because $\phi^{--}(1) = f^{--}(1) = g^{--}(1) = \psi^{--}(1)$. 
Note also why $\pglob(\omega_i)$ obeys the boundary conditions: since $\omega_i^{--}$ has poles only at $w=w_i$ and $z=z_i$, the boundary conditions are non-trivial only on the edges of the form $\bigl( (w_k,z_i), (z=z_i)\bigr)$ and $\bigl((w_i,z_\ell), (w=w_i)\bigr)$, and here they are obeyed because $\phi^{--}(0) = 0$ and $\psi^{--}(0)=0$, by the boundary conditions on $\omega_i^{--}$ itself. 

Let us stress that while $\iglob:\gglob \to \glocx$ is a map of dg Lie algebras, $\pglob: \glocx \to \gglob$ is a map of dg vector spaces only.

\begin{lem}\label{lem: hglob}
There exists a homotopy 
$\hglob: \gglob \to \gglob$ such that 
\be \id_{\gglob} - \pglob \circ \iglob = \dd \circ \hglob + \hglob \circ \dd  \nn\ee
holds as an equality of cochain maps $\gglob\to \gglob$.
\end{lem}
\begin{proof}
We have the decomposition of vector spaces
\be \CC(w)_{\ww}^\8 \oxC \CC(z)_{\zz}^\8 \cong_\CC \bigoplus_{i,j}  
(w-w_i)^{-1} (z-z_j)^{-1} \CC[(w-w_i)^{-1},(z-z_j)^{-1}]\nn\ee
coming from taking partial fractions in each global coordinate, $z$ and $w$. In this way 
an element $\mu \in \gN$ has partial fraction decomposition
\be \mu = \sum_{i,j=1}^N \mu_{ij} \label{pfrac}\ee
where $\mu_{ij}$ is a polynomial differential form on $\ARN$ with coefficents in $(w-w_i)^{-1} (z-z_j)^{-1} \CC[(w-w_i)^{-1},(z-z_j)^{-1}]$.
We shall define $\hglob(\mu)$ summand by summand,
\be \hglob(\mu) := \sum_{i,j=1}^N \hglob^{ij}(\mu_{ij}) \label{hglob}\ee 
The argument is similar to that in the proof of \cref{prop: disc retract}, and we shall use the pictorial notation from \cref{sec: pictures}.

First consider a summand $\mu_{k\ell}$ with $k\neq \ell$. By definition of $\pglob$, $\pglob\circ\iglob(\mu_{k\ell}) = 0$ vanishes, so we must arrange our homotopy to contract $\id(\mu_{k\ell}) = \mu_{k\ell}$ to zero (just as, in \cref{prop: discx retract}, we had to contract $\omega^{\pm\mp}$ to zero).
The boundary conditions mean that $\mu_{k\ell}$ must vanish when pulled back to any of the edges drawn as thick dashed lines in the sketch at the left below, and we  define $\hglob^{k\ell}$ to be the map drawn on the right:
\be 
\begin{tikzpicture}[baseline=0]
\tdplotsetmaincoords{0}{180}
\begin{scope}[xshift=-5cm,scale=1.301,tdplot_main_coords,local bounding box=G]
\fourfourpicijkl
\draw[ultra thick,dashed] (uwi) -- (0wi) -- (0U);
\end{scope}
\end{tikzpicture}
\qquad  h_{k\ell}(\mu) := 
\begin{tikzpicture}[baseline=0]
\tdplotsetmaincoords{0}{180}
\begin{scope}[xshift=-5cm,scale=1.301,tdplot_main_coords,local bounding box=G,decoration={
    markings,
    mark=at position 0.25 with {\arrow{stealth}},
    mark=at position 0.50 with {\arrow{stealth}},
    mark=at position 0.75 with {\arrow{stealth}},
}]
\fourfourpicijkl
\draw[very thick,opacity=.3,postaction={decorate},blue] (0wi) -- (uwi);
\draw[very thick,opacity=.3,postaction={decorate},blue] (0wi) -- (0U);
\draw[very thick,opacity=.3,postaction={decorate},blue] (0wi) -- (p1p1);
\draw[very thick,opacity=.3,postaction={decorate},blue] (p1p1) -- (z=u);
\draw[very thick,opacity=.3,postaction={decorate},blue] (p1p1) -- (w=U);
\draw[very thick,opacity=.3,postaction={decorate},blue] (p1p1) -- (uU);
\foreach \r in {0.25,0.5,0.75}
{
\draw[very thick,opacity=.3,postaction={decorate},blue] ($(0wi)!\r!(p1p1)$)  -- ($(uwi)!\r!(z=u)$) ;
\draw[very thick,opacity=.3,postaction={decorate},blue] ($(0wi)!\r!(p1p1)$)  -- ($(0U)!\r!(w=U)$) ;
}
\end{scope}
\begin{scope}[xshift=-5cm,scale=1.301,tdplot_main_coords,local bounding box=G,decoration={
    markings,
    mark=at position 0.55 with {\arrow{stealth}}}]
\foreach \r in {0.25,0.5,0.75}
{
\draw[very thick,opacity=.3,postaction={decorate},blue] ($(uU)!\r!(p1p1)$)  -- ($(uU)!\r!(z=u)$) ;
\draw[very thick,opacity=.3,postaction={decorate},blue] ($(uU)!\r!(p1p1)$)  -- ($(uU)!\r!(w=U)$) ;
}
\end{scope}
\end{tikzpicture}
\nn\ee
Here we have sketched part of the semisimplicial set $\Flag_\bul(\Rectn)$ corresponding to some $i,j$ with $i\neq j$ and $i,j \notin\{k,\ell\}$. Special cases occur when $k=j$ or $i=\ell$ or both, and when $i=j$, but these special cases just correspond to omitting one or more of the squares, and in a way which, one sees, does not obstruct us in retracting back to the point $(w_k,z_\ell)$ as above.  
One may then verify, by direct calculations similar to those in \cref{prop: disc retract}, that
\be [\dd,\hglob](\mu_{k\ell}) = \mu_{k\ell} - \mu_{k,\ell}|_{(w_k,z_\ell)} = \mu_{k\ell} - 0 = \mu_{k\ell}\nn\ee
as we wanted. 

It remains to consider the diagonal summands $\mu_{ii}$. The new feature is that the semisimplicial set of edges of $\Flag_\bul(\Rectn)$ on which $\mu_{ii}$ vanishes is no longer contractible to a single vertex: 
\be 
\begin{tikzpicture}[baseline=0]
\tdplotsetmaincoords{0}{180}
\begin{scope}[xshift=-5cm,scale=1.301,tdplot_main_coords,local bounding box=G,decoration={
    markings,
    mark=at position 0.75 with {\arrow{stealth}},
}]
\threefourpicikl
\draw[ultra thick,dashed] (z=u) -- (uwi);
\draw[ultra thick,dashed] (w=U) -- (0U);
\end{scope}
\end{tikzpicture}\nn\ee
(if $k=\ell$ here then the lower right square is absent). We set  
\be \hglob^{ii}(\mu) := 
\begin{tikzpicture}[baseline=0]
\tdplotsetmaincoords{0}{180}
\begin{scope}[xshift=-5cm,scale=1.301,tdplot_main_coords,local bounding box=G,decoration={
    markings,
    mark=at position 0.75 with {\arrow{stealth}},
}]
\threefourpicikl
\draw[very thick,opacity=.3,postaction={decorate},blue] (z=u) -- (uwi);
\draw[very thick,opacity=.3,postaction={decorate},blue] (w=U) -- (0U);
\draw[very thick,opacity=.3,postaction={decorate},blue] (p1p1) -- (0wi);
\draw[very thick,opacity=.3,postaction={decorate},blue] (p1p1) -- (0wi);
\foreach \r in {0.25,0.5,0.75}
{
\draw[very thick,opacity=.3,postaction={decorate},blue] ($(z=u)!\r!(p1p1)$)  -- ($(uwi)!\r!(w=wi)$) ;
\draw[very thick,opacity=.3,postaction={decorate},blue] ($(w=U)!\r!(p1p1)$)  -- ($(0U)!\r!(z=0)$) ;
\draw[very thick,opacity=.3,postaction={decorate},blue] ($(p1p1)!\r!(0wi)$)  -- ($(z=0)!\r!(0wi)$) ;
\draw[very thick,opacity=.3,postaction={decorate},blue] ($(p1p1)!\r!(0wi)$)  -- ($(w=wi)!\r!(0wi)$) ;
}
\end{scope}
\end{tikzpicture}
\nn\ee
Again, one may verify by direct calculation that
\be [\dd,\hglob](\mu_{ii}) = \mu_{ii} - \pglob(\iglob(\mu_{ii})) \nn\ee
with $\pglob$ as we defined it in \cref{pglob} above.
\end{proof}

We continue with the proof of \cref{thm: he}. The next step is to construct a map of dg vector spaces
\be \ploc: \glocx \to \gloc \nn\ee
such that
\be \id_{\glocx} -\iglob\circ \pglob \simeq \iloc\circ\ploc, \label{iglocx}\ee
i.e. such that for any $\omega\in \glocx$, the difference
\be \wt\omega := \omega - \iglob\circ \pglob(\omega)\nn\ee 
is equal to $\iloc\circ\ploc(\omega)$ up to homotopy.

We define $\ploc$ by
\be \ploc(\omega)_i := \left( \omega_i - \iglob(\pglob(\omega))_i \right)^{++}|_{s=1} = \left( \wt\omega_i^{++}|_{s=1} \right)  \label{ploc}\ee
-- cf. \cref{disc P} and \cref{omegaidecomp}. 
Then, by construction, we have $\wt \omega_i^{--}=0$.
It remains to find a homotopy $\hlocx:\glocx\to \glocx$ contracting $\wt\omega_i^{+-}$, $\wt\omega_i^{-+}$, and the non-constant terms in $\wt\omega_i^{++}$. This can be done much as in the proof of \cref{prop: discx retract}. We may write $\wt\omega$ as
\be \wt\omega = \bigl( \wt\omega_i \bigr)_{i} = \bigl(\wt\phi_i(s) = \wt f_i(s) + \wt F_i(s)\dd s, \, \wt\psi_i(s) = \wt g_i(s) + \wt G_i(s) \dd s \bigr)_{i\in\{0,p,\8\}} \nn\ee
and set
\be \hlocx(\omega_i^{++})(s) = \left( \int_{1}^{s} \wt F_i^{++}(s') \dd s' \,,\,\, \int_{1}^{s} \wt G_i^{++}(s') \dd s' \right), \nn\ee
where we stress that we are defining $\hlocx(\omega)$ and it is the components of $\wt\omega$ that appear on the right. 
Using the fact that $\iglob\circ\pglob$ is a cochain map, and so commutes with $\dd$, we then have 
\begin{align} [\dd, \hlocx] (f_i^{++}(s),g_i^{++}(s)) &= (\hlocx \circ \dd)(f_i^{++}(s),g_i^{++}(s))\nn\\ 
&= \bigl(\wt f_i^{++}(s) - \wt f_i^{++}(1) , \wt g_i^{++}(s) - \wt g_i^{++}(1) \bigr) \nn\\
&=(\wt f_i^{++}(s) , \wt g_i^{++}(s)) - (1,1) \wt f_i^{++}(1)  \nn
\end{align}
and 
\begin{align} [\dd, \hlocx] (F_i^{++}(s)\dd s, G_i^{++}(s)\dd s) 
&= (\dd \circ \hlocx )(F_i^{++}(s)\dd s, G_i^{++}(s)\dd s)\nn\\  
&=  (\wt F_i^{++}(s)\dd s, \wt G_i^{++}(s)\dd s).\nn\end{align}
That is, 
\be [\dd, \hlocx] (\omega_i^{++}) = \wt \omega_i^{++} - \wt \omega_i^{++}|_{s=1} = (\id - \iglob\circ\pglob - \iloc\circ\ploc)(\omega_i^{++})  \nn\ee 
as we want. 

The argument for $\omega_i^{+-}$ and $\omega_i^{-+}$ is similar, again following the prototype in the proof of \cref{prop: discx retract}. (In fact, observe that $\omega_i^{+-} = \wt\omega_i^{+-}$ and  $\omega_i^{-+} = \wt\omega_i^{-+}$.)

This establishes that \cref{iglocx} holds, as we wanted. That is, $I\circ P$ is homotopic to the identity map $\id_{\glocx}$. 

To complete the proof of \cref{thm: he} it remains to check that $P\circ I$ is homotopic to the identity map $\id_{\gglob\oplus\gloc}$. We already defined the required homotopy for the global part $\pglob\circ\iglob$, in \cref{lem: hglob}. The restriction of $P\circ I$ to $\gloc$ is the identity map on the nose, i.e. we have
\be P\circ I|_{\gloc} = \ploc\circ \iloc = \id_{\gloc},\nn\ee
and it is manifest that  $\pglob\circ\iloc=0$ vanishes. However, the map $P\circ I$ does have a nonzero off-diagonal component
\be \ploc\circ\iglob : \gglob \to \gloc\nn\ee
which we must show is homotopic to zero.

\begin{lem} For each $i=1,\dots,N$, we have
\be \ploc(\iglob(\mu))_i = \iota_{w-w_i,z-z_i} \sum_{\substack{k,\ell=1\\ k\neq \ell \neq i \neq k}}^N  \mu_{k\ell}|_\Eta ,\nn\ee
where $\mu= \sum_{i,j=1}^N\mu_{ij}$ is the partial fraction decomposition of an element $\mu\in \gglob$ as in \cref{pfrac}.
\end{lem}
\begin{proof} First, note that the operation of restricting (i.e. pulling back) a polynomial differential form on the semisimplicial set $\Flag(\Rectn)$ to the vertex $\Eta$ factors through the operation of first restricting it to $\Flag(\DDpx{{w_i,z_i}})$ for any one of the punctured formal polydiscs $\DDpx{{w_i,z_i}}$. 

We have  $\iglob(\mu)^{--}_i = \mu_{ii}$. Therefore the restriction of the polynomial differential form $\pglob(\iglob(\mu))\in \gglob$  to the vertex $\Eta$ of $\Flag(\Rectn)$ is given by $\sum_{j=1}^N \mu_{jj}|_\Eta$, according to our definition \cref{pglob} of $\pglob$. Thus the restriction of $\iglob(\pglob(\iglob(\mu)))_i$ to the vertex $\Eta$ of $\Flag(\DDpx{{w_i,z_i}})$ is given by $\iota_{w-w_i,z-z_i} \sum_{j=1}^N \mu_{jj}|_\Eta$. 
By definition \cref{ploc} of $\ploc$, we find, for each $i=1,\dots,N$, that
\begin{align} \ploc(\iglob(\mu))_i &= \left( \iglob(\mu)_i - \iglob(\pglob(\iglob(\mu)))_i \right)^{++}|_{\Eta}  \nn\\
& = \iota_{w-w_i,z-z_i} \left(\sum_{\substack{k,\ell=1\\k,\ell\neq i}}^N \mu_{k\ell}|_\Eta -
  \sum_{\substack{j=1\\j\neq i}}^N \mu_{jj}|_\Eta\right) 
= \iota_{w-w_i,z-z_i} \sum_{\substack{k,\ell=1\\ k\neq \ell \neq i \neq k}}^N  \mu_{k\ell}|_\Eta\nn
\nn\end{align}
as required.
\end{proof}
We may then define
\be \hoffdiag : \gglob \to \gloc \nn\ee
to be the map given in terms of the partial fraction decomposition \cref{pfrac} of $\mu\in \gglob$ by
\begin{align} \hoffdiag(\mu)_i &:= \iota_{w-w_i,z-z_i} \sum_{\substack{k,\ell=1\\ k\neq \ell \neq i \neq k}}^N \hglob^{k\ell}(\mu_{k\ell})|_\Eta \nn\\
&= \iota_{w-w_i,z-z_i} \sum_{\substack{k,\ell=1\\ k\neq \ell \neq i \neq k}}^N \int_{(w_k,z_\ell)}^\Eta \mu_{k\ell}|_{\bigl( (w_k,z_\ell), \Eta\bigr)} 
\end{align}
(the integral is over the edge of $\Flag(\Rectn)$ joining $(w_k,z_\ell)$ to $\Eta$). 
Then indeed $\dd(\hoffdiag(\mu))_i = 0$ and 
\be \hoffdiag(  \dd (\mu))_i = \ploc(\iglob(\mu))_i - 0 = \ploc(\iglob(\mu))_i .\nn\ee
Thus we have shown that $\ploc \circ \iglob \simeq 0$, as we wanted. 

This completes the proof of \cref{thm: he}. 
\end{proof}

\subsection{The cohomology and the pairing}\label{sec: cohomology and pairing}
As in the proof of \cref{prop: pairing}, the pairing $\pairing{-}{-}$ restricts to a non-degenerate pairing 
\begin{align} H^0(\glocx) \ox  H^1(\glocx) &\to \down \CC\nn
\end{align}
between $H^0(\glocx) = \gloc$ and $H^1(\glocx)$, and these subspaces are again both manifestly isotropic. 
Therefore it is enough to establish the following. 
\begin{prop} 
There is a deformation retract of dg vector spaces
\be\begin{tikzcd} 
H^1(\glocx) 
\rar[shift left]{\iota}  & \lar[shift left]{\pi} \gglob \ar[loop right,"h"] 
\end{tikzcd}\nn\ee
(We stress that \emph{neither} $\pi$ nor $\iota$ here are maps of dg Lie algebras.) 
\end{prop}
\begin{proof}
Let $\gm^i:= \Th(\g\ox \ADpx{{w_i,z_i}}^{--})$ denote the copy of the dg Lie algebra $\gm$, cf. \cref{sec: def gp gm}, associated to the punctured formal disc at the marked point $(w_i,z_i)$. 
As in \cref{prop: H1 retract} we have the deformation retract of dg vector spaces
\be\begin{tikzcd} 
H^1(\gDpx{{w_i,z_i}}) 
\rar[shift left]{}  & \lar[shift left]{} \gm^i \ar[loop right,""] 
\end{tikzcd}\nn\ee
for each $i$. Deformation retracts compose. So to prove the result it is enough to show that there is a deformation retract of dg vector spaces
\be\begin{tikzcd} 
 \bigoplus_{i=1}^N \gm^i \rar[shift left]{f}  & \lar[shift left]{g} \gglob \ar[loop right,"\hglob"] 
\end{tikzcd}\nn\ee
And indeed, we have a dg Lie algebra map $I_{\gm}^i: \gm^i \to \gDpx{{w_i,z_i}}$ defined as in \cref{sec: def gp gm}, and the dg vector space map $\pglob^i: \gDpx{{w_i,z_i}} \to \gglob$ from \cref{pglob}. We let 
\be f:= \bigoplus_{i=1}^N f^i,\qquad f^i := \pglob^i \circ I_{\gm}^i :  \gm^i \to \gglob. \nn\ee
In the other direction, let
\be g = \bigoplus_{i=1}^N g^i,\qquad g^i: \mu \mapsto \iglob^i(\mu_{ii}) \nn\ee
where $\iglob^i$ was defined in \cref{iglobi} and $\mu_{ii}$ is as in \cref{pfrac}. On inspection, one sees that $g^i\circ f^i = \id_{\gm^i}$ for each $i$ (and that $g^i\circ f^j = 0$ for $i\neq j$).     
The homotopy $\hglob$ here is the same as in \cref{lem: hglob}. As there, it contracts all the off-diagonal pole terms $\mu_{ij}$ in the partial fraction decomposition of $\mu$, \be [\dd,\hglob](\mu_{ij}) =\mu_{ij},\qquad i\neq j\nn\ee
and retracts each diagonal term $\mu_{ii}$ to the semisimplicial subset $\Flag(\DDpx{{w_i,z_i}}$. By direct calculation one checks that
\be [\dd,\hglob](\mu_{ii}) = \mu_{ii} - f^i(g^i( \mu_{ii})) .\nn\ee
\end{proof}

\section{Triangular decompositions of enveloping algebras}\label{sec: triangular decompositions}
The main results of this section are \cref{cor: disc bimod} and \cref{cor: global bimod}. 
\subsection{Local case} 
In this section we establish the following corollary of \cref{prop: discx retract}.
\begin{cor}\label{cor: disc bimod}
There is a deformation retract of $(U(\gm),U(\gp))$-bimodules
\be 
\begin{tikzcd} 
U(\gm) \ox U(\gp) \rar[shift left]{U(I)}  & \lar[shift left]{U(P)} U(\gDx) \ar[loop right,"\wt h"] 
\end{tikzcd}\nn\ee
\end{cor}
\begin{proof}
Consider the deformation retract
\be\begin{tikzcd} 
\gm \oplus \gp \rar[shift left]{I}  & \lar[shift left]{P} \gDx \ar[loop right,"h"] 
\end{tikzcd}\nn\ee
of \cref{prop: discx retract}. The cochain map
$I\circ P: \gDx \to \gDx$ 
is a projector (because 
$(I \circ P) \circ (I \circ P) = I \circ ( P\circ I) \circ I = I \circ \id_{\gm \oplus \gp} \circ P = I \circ P$).
Its image is a dg subspace which we can and shall regard as an embedded copy of the dg vector space  $\gm \oplus \gp$. We get also the projector $\id_{\gDx} - I \circ P$ onto a dg subspace 
\be \g^\perp := (\id_{\gDx} - I \circ P)(\gDx),\nn\ee
and we obtain the direct sum decomposition of cochain complexes,
\be \gDx = \gm \oplus \g^\perp \oplus \gp \label{gdecom}\ee
To be concrete,  recall the decomposition \cref{omegadecomp} of an element $\omega \in \gDx$. The decomposition above is
\be \omega = \left(\omega^{--}, \omega^{+-} + \omega^{-+} + \omega^{++} - \omega^{++}|_{s=1}, \omega^{++}|_{s=1} \right) \nn\ee

We are therefore in the setting of the following lemma, which is essentially the dg analog of Proposition 2.2.7 and (a special case of) Proposition 2.2.9 of \cite{Dixmier_1996}.
\begin{lem}\label{lem: pbw}
Suppose $\a^- \into \a$ and $\a^+ \into \a$ are embeddings of dg Lie algebras, and $\a^\perp \into \a$ an embedding of dg vector spaces, such that 
\be \a \cong \a^- \oplus \a^\perp \oplus \a^+\nn\ee
as dg vector spaces. 
Then
\be U(\a) \cong U(\a^-) \ox \Sym(\a^\perp) \ox U(\a^+)\nn\ee
as dg vector spaces and, moreover, as $(U(\a^-), U(\a^+))$-bimodules. 
\qed\end{lem}
In view of this lemma, we have the isomorphism of  $(U(\gm), U(\gp))$-bimodules
\be U(\gDx) \cong U(\gm) \ox \Sym(\g^\perp) \ox U(\gp).\nn\ee
To prove \cref{cor: disc bimod} it remains to show that our homotopy $h$ from \cref{prop: discx retract} gives rise to a retract of $(U(\gm), U(\gp))$-bimodules
\be\begin{tikzcd} 
U(\gm) \ox U(\gp) \rar[shift left]{U(I)}  & \lar[shift left]{U(P)} 
U(\gm) \ox \Sym(\g^\perp) \ox U(\gp)  \arrow[loop right, distance=4em, start anchor={[yshift=3ex]east}, end anchor={[yshift=-3ex]east},"\wt h"]
\end{tikzcd}\nn\ee
To do that, we adapt an argument taken from the proof of \cite[Proposition 2.5.5]{GwilliamThesis}. The natural embedding and projection maps of dg vector spaces
\be \begin{tikzcd}  \CC \cong \Sym^0(\g^\perp) \rar[hook,shift left]{}  & \lar[shift left,two heads]{} \Sym(\g^\perp) 
\end{tikzcd} 
\nn\ee
give rise to maps $U(I)$ and $U(P)$ of free $(U(\gm),U(\gp)$-bimodules
\be\begin{tikzcd} 
U(\gm) \ox U(\gp)  \rar[shift left]{U(I)}  & \lar[shift left]{U(P)} 
U(\gm) \ox \Sym(\g^\perp) \ox U(\gp). 
\end{tikzcd}\nn\ee
Now, our homotopy $h: \gDx \to \gDx$ from \cref{prop: discx retract} by construction preserves the decomposition \cref{gdecom}, and is zero on the summands $\gp$ and $\gm$. So it defines a map $\g^\perp \to \g^\perp$ which, abusively, we continue to call $h$. 
(In the language of \cite{GwilliamThesis}, our homotopy obeys the \dfn{side conditions} and our retract is thus a \dfn{strong deformation retract} or \dfn{contraction}.)
As maps $\g^\perp \to \g^\perp$, we have
\be [\dd, h] = \id .\nn\ee
Recall that any map of dg vector spaces $V \to V$ extends uniquely to a derivation of the free dg commutative algebra $\Sym(V)$. In particular, we may extend $h$ and $\id$ to derivations of $\Sym(\g^\perp)$. But the extension of $\id$ to a derivation is just the map which multiplies each element of the dg subspace $\Sym^n(\g^\perp)$ by a factor of $n$, for each $n\geq 0$. So we obtain that
\be [\dd, h]|_{\Sym^n(\g^\perp)} = n \,\id_{\Sym^n(\g^\perp)} \nn\ee
for each $n\geq 0$.
We may therefore define the required homotopy 
\be \wt h : U(\gm) \ox \Sym(\g^\perp) \ox U(\gp) \to U(\gm) \ox \Sym(\g^\perp) \ox U(\gp) \nn\ee 
to be 
\be \wt h := \id \ox \left(\sum_{n\geq 1} \frac 1 n h|_{\Sym^n \g^\perp} \right)\ox \id \nn\ee
for then $[\dd, \wt h]$ is indeed the identity on each subspace $U( \gm ) \ox \Sym^n(\g^\perp) \ox U(  \gp )$ with $n\geq 1$, and is by definition zero on $U( \gm ) \ox \Sym^0(\g^\perp) \ox U(  \gp )$. 
Thus finally we get that
\be [\dd, \wt h ] = \id - U(I) \circ U(P), \nn\ee
as required.
It is evident that $\wt h$ is a map of $(U(\gm), U(\gp))$-bimodules.
This completes the proof of \cref{cor: disc bimod}. 
\end{proof}

Next, we would like to do something similiar in the global case, while staying in the world of dg algebras. For that purpose, we would really wish to have a deformation retract, rather than merely a homotopy equivalence as we have in \cref{thm: he}.

We shall obtain a result in this direction in \cref{thm: global retract} below. There are two steps: 

First, we shall introduce, on both sides of the map $I$, additional summands associated to all unpunctured discs $\DDp{{w_i,z_j}}$ at the ``off-diagonal'' points $(w_i,z_j)$, $i\neq j$. Indeed, arguing as for \cref{thm: he} it is not hard to see that there is also a homotopy equivalence
\be\begin{tikzcd}[column sep = huge,every cell/.style={font=\everymath\expandafter{\the\everymath\displaystyle}}]
\arrow[loop left, distance=4em, start anchor={[yshift=-3ex]west}, end anchor={[yshift=3ex]west},"\hglob"] 
\gN \oplus \bigoplus_{i,j=1}^N \gp(w_i,z_j) \quad\rar[shift left]{I}  & \lar[shift left]{J}\quad \bigoplus_{i=1}^N \gDpx{{w_i,z_i}} \oplus \bigoplus_{\substack{i\neq j}} \gDp{{w_i,z_j}} 
\arrow[loop right, distance=4em, start anchor={[yshift=3ex]east}, end anchor={[yshift=-3ex]east},"h"] 
\end{tikzcd}\nn\ee
(Note that by paying the price of introducing these extra points, we get rid of the off-diagonal terms in the homotopy on the left.)
 
Second, we introduce new models for all the summands on the right (at both the punctured and unpunctured discs). These models are chosen to be sufficiently large that we can reconstruct an element of $\gN$ on the nose, rather than merely up to homotopy, from its image under $I$.

We turn to these enlarged models now.

\subsection{Big models for the local algebras}
Now we model the disc algebras as $\Flag_\bul(\Rectn)$-objects.

Pick any $(i,j) \in \{1,\dots,N\}^2$. We set $\ARN^{ij}$ to be the $\Flag_\bul(\Rectn)$-object in commutative algebras given as follows. We assign, to every flag/simplex of $\Flag_\bul(\Rectn)$, the commutative algebra
\be \CC((w-w_i))\oxC\CC((z-z_j)) ,\nn\ee
with only the following exceptions:
\be \ARN^{ij}( \{\pt\} \subset \{ w=w_i\} ) :=\ARN^{ij}( \{ w=w_i\} ) := \CC[[w-w_i]]\oxC\CC((z-z_j)) \nn\ee
\be \ARN^{ij}( \{\pt\} \subset \{ z=z_j\} ) :=\ARN^{ij}( \{ z=z_j\} ) := \CC((w-w_i))\oxC\CC[[z-z_j]] \nn\ee
and (when $i\neq j$)
\be \ARN^{ij}( \{(w_i,z_j)\} ) := \CC[[w-w_i]]\oxC\CC[[z-z_j]] .\nn\ee
Let $\gN^{\ij}$ denote the dg Lie algebra
\be \gN^{ij} := \Th(\g \ox \ARN^{ij}) \nn\ee
Thus, $\gN^{ij}$ is the dg algebra of polynomial differential forms on $\Flag_\bul(\Rectn)$, valued in $ \g\ox\CC((w-w_i))\oxC\CC((z-z_j))$, and subject to boundary conditions above -- which now, in contrast to $\gN$ in \cref{sec: gNdef}), are only on the boundary simplices corresponding to the flags intersecting the point $(w_i,z_j)$.  

The next lemma says that $\gN^{ij}$ provides another model for $\gDp{{w_i,z_j}}$ (for $i\neq j$) and that $\gN^{ii}$ provides another model for $\gDpx{{w_i,z_i}}$. 

Let $\gp^{ij} := \g\ox \CC[[w-w_i]]\oxC\CC[[z-z_j]]$. Recall $\gp^{\ij} \simeq \gDp{{w_i,z_j}}$, \cref{prop: disc retract}.

\begin{lem}\label{lem: big disc retracts}
There are deformation retracts of dg vector spaces, 
\be\begin{tikzcd} 
\gp^{ij} \rar[shift left]{g_{ij}}  & \lar[shift left]{f_{ij}} \gN^{ij} \ar[loop right,"h_{ij}"] 
\end{tikzcd}\nn\ee 
for $i\neq j$, and
\be\begin{tikzcd} 
\gDpx{{w_i,z_i}} \rar[shift left]{g_{ii}}  & \lar[shift left]{f_{ii}} \gN^{ii} \ar[loop right,"h_{ii}"] 
\end{tikzcd}\nn\ee 
where the maps $g_{ij}, f_{ij}, g_{ii}, f_{ii}$ are maps of dg Lie algebras.
\end{lem}
\begin{proof}
We consider first $\gN^{ii}$. 
Since $\Flag_\bul(\DDpx {{w_i,z_i}})$ is a semisimplicial subset of $\Flag_\bul(\Rectn)$, we get the restriction of $\ARN$ to a $\Flag_\bul(\DDpx {{w_i,z_i}})$-algebra, and we recognise the latter as a copy of  $\ADpx{{w_i,z_i}}$.
Hence, by \cref{lem: sca1} and \cref{lem: sca2} we get the map of semicosimplicial algebras
$\sca \ARN  \to \sca \ADpx {{w_i,z_i}}$
and therefore, by the functoriality of $\Th$, a map of dg Lie algebras 
\be f_{ii} : \gN \to \gDpx {{w_i,z_i}}. \nn\ee
Now we define the map $g_{ii}$. Consider an element 
\be \omega(s) = (\phi(s),\psi(s)) \in \gDpx {{w_i,z_i}}\nn\ee 
cf. \cref{gdxelements}. 


Let $g_{ii}(\omega)\in \gN^{ii}$ be the polynomial differential form on $\Flag_\bul(\Rectn)$ given as follows:
\be g_{ii}(\omega) := 
\begin{tikzpicture}[baseline=0]
\tdplotsetmaincoords{0}{180}
\begin{scope}[xshift=-5cm,scale=2,tdplot_main_coords,local bounding box=G]
\threefourpicikl
\node (N{0wi }{w=wi}{p1p1}) at ({0wi }{w=wi}{p1p1}) {$\psi(1)$}; 
\node (N{0wi }{z=0 }{p1p1}) at ({0wi }{z=0 }{p1p1}) {$\phi(1)$}; 
\node (N{0U  }{z=0 }{p1p1}) at ({0U  }{z=0 }{p1p1}) {$\phi(t)$};
\node (N{0U  }{w=U }{p1p1}) at ({0U  }{w=U }{p1p1}) {$\phi(s)$};
\node (N{uwi }{z=u }{p1p1}) at ({uwi }{z=u }{p1p1}) {$\psi(s)$}; 
\node (N{uwi }{w=wi}{p1p1}) at ({uwi }{w=wi}{p1p1}) {$\psi(t)$}; 
\end{scope}
\end{tikzpicture}
\nn\ee
(If $k=\ell$ then the lower right square is absent.) 

First, observe that this does define a map of dg Lie algebras $\gDpx{{w_i,z_i}} \to \gN^{ii}$. And it is evident that 
\be f_{ii}\circ g_{ii} (\omega) = \omega.\nn\ee
It remains to show that, for all $\mu \in \gN^{ii}$,
\be g_{ii}\circ f_{ii} (\mu) = \mu + [h_{ii}, \dd] (\mu) \nn\ee
for some suitable homotopy $h_{ii}: \gN^{ii} \to \gN^{ii}$. And indeed, by direct calculation, one checks that a suitable homotopy is given as follows, in the pictorial notation we introduced in \cref{sec: pictures}:
\be h_{ii}(\mu) := 
\begin{tikzpicture}[baseline=0]
\tdplotsetmaincoords{0}{180}
\begin{scope}[xshift=-5cm,scale=1.301,tdplot_main_coords,local bounding box=G,decoration={
    markings,
    mark=at position 0.75 with {\arrow{stealth}},
}]
\threefourpicikl
\draw[ultra thick,dashed] (z=u) -- (uwi);
\draw[ultra thick,dashed] (w=U) -- (0U);
\draw[very thick,opacity=.3,postaction={decorate},blue] (p1p1) -- (0wi);
\foreach \r in {0.25,0.5,0.75}
{
\draw[very thick,opacity=.3,postaction={decorate},blue] ($(z=u)!\r!(p1p1)$)  -- ($(uwi)!\r!(w=wi)$) ;
\draw[very thick,opacity=.3,postaction={decorate},blue] ($(w=U)!\r!(p1p1)$)  -- ($(0U)!\r!(z=0)$) ;
\draw[very thick,opacity=.3,postaction={decorate},blue] ($(p1p1)!\r!(0wi)$)  -- ($(z=0)!\r!(0wi)$) ;
\draw[very thick,opacity=.3,postaction={decorate},blue] ($(p1p1)!\r!(0wi)$)  -- ($(w=wi)!\r!(0wi)$) ;
}
\end{scope}
\end{tikzpicture}
\nn\ee
(Here we have highlighted the edges on which the boundary conditions are non-trivial.)

Now we turn to $\gN^{ij}$ for $i\neq j$. Deformation retracts compose, so in view of \cref{prop: disc retract}, it is enough to show that there is a deformation retract of dg Lie algebras
\be\begin{tikzcd} 
\gDp{{w_i,z_j}} \rar[shift left]{g_{ij}}  & \lar[shift left]{f_{ij}} \gN^{ij} \ar[loop right,"h_{ij}"] 
\end{tikzcd}\nn\ee 
This can be done very similarly to the previous case of $\gN^{ii}$, with the following modification to the definition of the map $g$:
\be g_{ij}\left(
\begin{tikzpicture}[baseline=25]
\tdplotsetmaincoords{0}{180}
\begin{scope}[xshift=-5cm,scale=2,tdplot_main_coords,local bounding box=G]
\squarecoords
\draw[fill=blue,fill opacity=.1] (p1p1) -- (w=U) -- (uU)   --cycle; 
\draw[fill=blue,fill opacity=.1] (p1p1) -- (z=u) -- (uU)  --cycle; 
\begin{scope}[every node/.style={font=\scriptsize}]
\node[inner sep=1.5,label= left:{$(w_i,z_j)$}]   (NuU)   at (uU)  {};
\node[inner sep=1.5,label= left:{$(w=w_i)$}]   (Nw=U)  at (w=U) {};
\node[inner sep=1.5,label= above:{$(z=z_j)$}]   (Nz=u)  at (z=u) {};
\node[inner sep=1.5,label= below:{$\Eta$}] (Np1p1) at (p1p1) {};
\end{scope}  
\node (N{uU  }{w=U }{p1p1}) at ({uU  }{w=U }{p1p1}) {$\phi(s,t)$}; 
\node (N{uU  }{z=u }{p1p1}) at ({uU  }{z=u }{p1p1}) {$\psi(s,t)$}; 
\end{scope}
\end{tikzpicture}
\right) := 
\begin{tikzpicture}[baseline=0]
\tdplotsetmaincoords{0}{180}
\begin{scope}[xshift=-5cm,scale=2,tdplot_main_coords,local bounding box=G]
\fourfourpicijkl
\node (N{uU  }{w=U }{p1p1}) at ({uU  }{w=U }{p1p1}) {$\phi(s,t)$}; 
\node (N{uU  }{z=u }{p1p1}) at ({uU  }{z=u }{p1p1}) {$\psi(s,t)$}; 
\node (N{0wi }{w=wi}{p1p1}) at ({0wi }{w=wi}{p1p1}) {$\psi(1,1)$}; 
\node (N{0wi }{z=0 }{p1p1}) at ({0wi }{z=0 }{p1p1}) {$\phi(1,1)$}; 
\node (N{0U  }{z=0 }{p1p1}) at ({0U  }{z=0 }{p1p1}) {$\phi(t,1)$};
\node (N{0U  }{w=U }{p1p1}) at ({0U  }{w=U }{p1p1}) {$\phi(s,1)$};
\node (N{uwi }{z=u }{p1p1}) at ({uwi }{z=u }{p1p1}) {$\psi(s,1)$}; 
\node (N{uwi }{w=wi}{p1p1}) at ({uwi }{w=wi}{p1p1}) {$\psi(t,1)$}; 
\end{scope}
\end{tikzpicture}
\nn\ee
(where, again, the lower right square is absent if $k=\ell$.)
\end{proof}

\subsection{Global deformation retract}
Let 
\be \galoc := \bigoplus_{i,j=1}^N \gp^{ij} \qquad\text{and}\qquad 
    \galocx := \bigoplus_{i,j=1}^N \gN^{ij} \nn\ee
We have the dg Lie algebra maps
\be \iloc : \galoc \to \galocx \qquad\text{and}\qquad 
   \iglob : \gN \to \galocx .\nn\ee
\begin{thm}\label{thm: global retract}
There is a deformation retract of dg vector spaces
\be\begin{tikzcd}[column sep = huge] 
\gN \oplus \galoc\quad \rar[shift left]{I=(\iglob,\iloc)}  & \lar[shift left]{P=\pglob\oplus \ploc} \quad\galocx \ar[loop right,"h"] 
\end{tikzcd}\nn\ee
\end{thm}
\begin{proof}
We first define the map $\pglob$. 
An element of $\omega \in \galocx$ is a tuple
\be \omega = (\omega_{ij})_{i,j=1}^N ,\qquad  \omega_{ij} \in \gN^{ij}.\nn\ee
Given $\omega_{ij}\in \gN^{ij}$ we have, just in \cref{omegadecomp}, the decomposition
\be \omega_{ij} = \omega_{ij}^{++} + \omega_{ij}^{-+} + \omega_{ij}^{+-} + \omega_{ij}^{--} \nn\ee
coming from the direct sum decomposition of the vector space $\CC((w-w_i))\oxC\CC((z-z_j))$ into the parts polar/regular parts with respect to each of the local coordinates, $w-w_i$ and $z-z_j$.
In particular, the ${}^{--}$ part can be interpreted as a rational function vanishing at infinity, via the embedding of commutative algebras
\be (w-w_i)^{-1} (z-z_j)^{-1} \CC[(w-w_i)^{-1},(z-z_j)^{-1}] \into 
\CC(w)_{\ww}^\8 \oxC \CC(z)_{\zz}^\8.\nn\ee
Making implicit use of these embeddings, we set
\be \pglob(\omega) :=  \sum_{i,j=1}^N \omega_{ij}^{--} \nn\ee
It is then manifest that
\be \pglob \circ \iglob = \id_{\gN}. \nn\ee
Indeed, we have the decomposition of vector spaces
\be \CC(w)_{\ww}^\8 \oxC \CC(z)_{\zz}^\8 \cong_\CC \bigoplus_{i,j}  
(w-w_i)^{-1} (z-z_j)^{-1} \CC[(w-w_i)^{-1},(z-z_j)^{-1}]\nn\ee
coming from taking partial fractions in each global coordinate, $z$ and $w$. In this way 
an element $\mu \in \gN$ has partial fraction decomposition
\be \mu = \sum_{i,j=1}^N \mu_{ij} \nn\ee
where $\mu_{ij}$ is a polynomial differential form on $\ARN$ with coefficents in $(w-w_i)^{-1} (z-z_j)^{-1} \CC[(w-w_i)^{-1},(z-z_j)^{-1}]$. We see that $\iglob(\mu)_{ij}^{--} = \mu_{ij}$ for each $i,j$, and so
\be (\pglob\circ \iglob)(\mu) = \sum_{i,j=1}^N \iglob(\mu)_{ij}^{--} = \sum_{i,j=1^N} \mu_{ij} = \mu. \nn\ee  

Observe that 
\be \pglob \circ \iloc = 0. \nn\ee

Next we define the map $\ploc$. An element of $\galocx$ is a tuple $\omega = (\omega_{ij})_{i,j=1}^N$ as above, and we let
\be \ploc( \omega)_{ij} := f_{ij}( \omega_{ij} - (\iglob(\pglob( \omega)))_{ij}) \nn\ee
where $f_{ij}$ are the dg Lie algebra maps of \cref{lem: big disc retracts}, yielding the tuple
\be \ploc(\omega) = (\ploc(\omega)_{ij})_{i,j=1}^N \in \galoc \nn\ee
Since $\pglob \circ \iloc= 0$ we see that \cref{lem: big disc retracts} then ensures that
\be \ploc \circ \iloc = \id_{\galoc}. \nn\ee

Now we define the homotopy $h: \galocx \to \galocx$. It acts summand by summand, i.e. $h = \sum_{i,j=1}^N h_{ij}$, $h_{ij} : \gN^{ij} \to \gN^{ij}$. 
Our definitions above ensure that 
\be I \circ P (\omega_{ij}) = \iglob\circ \pglob(\omega_{ij}) + \iloc\circ\ploc(\omega_{ij}) = \omega_{ij}^{--} + \omega_{ij}^{++}|_{\Eta}.
\nn\ee
Indeed, $\iglob \circ \ploc$ and $\iloc\circ \pglob$ are trivially zero (the codomains have trivial intersection with the domains). 
The situation is thus similar to that in the proof of \cref{prop: discx retract}, in that we must contract $\omega_{ij}^{\pm\mp}$ and the non-constant part of $\omega_{ij}^{++}$.

Consider first $\omega_{ij}^{++}$. It obeys no nontrivial boundary conditions, and we may simply retract to the vertex $\Eta$,
\be h(\omega_{ij}^{++}) :=
\begin{tikzpicture}[baseline=0]
\tdplotsetmaincoords{0}{180}
\begin{scope}[xshift=-5cm,scale=1.301,tdplot_main_coords,local bounding box=G,decoration={
    markings,
    mark=at position 0.75 with {\arrow{stealth}},
}]
\fourfourpicijkl
\draw[very thick,opacity=.3,postaction={decorate},blue] (p1p1) -- (0wi);
\draw[very thick,opacity=.3,postaction={decorate},blue] (p1p1) -- (0U);
\draw[very thick,opacity=.3,postaction={decorate},blue] (p1p1) -- (uU);
\draw[very thick,opacity=.3,postaction={decorate},blue] (p1p1) -- (uwi);
\draw[very thick,opacity=.3,postaction={decorate},blue] (p1p1) -- (z=u);
\draw[very thick,opacity=.3,postaction={decorate},blue] (p1p1) -- (w=wi);
\draw[very thick,opacity=.3,postaction={decorate},blue] (p1p1) -- (z=0);
\draw[very thick,opacity=.3,postaction={decorate},blue] (p1p1) -- (w=U);
\foreach \r in {0.25,0.5,0.75}
{
\draw[very thick,opacity=.3,postaction={decorate},blue] ($(p1p1)!\r!(0wi)$)  -- ($(z=0)!\r!(0wi)$) ;
\draw[very thick,opacity=.3,postaction={decorate},blue] ($(p1p1)!\r!(0wi)$)  -- ($(w=wi)!\r!(0wi)$) ;
\draw[very thick,opacity=.3,postaction={decorate},blue] ($(uU)!\r!(p1p1)$)  -- ($(uU)!\r!(z=u)$) ;
\draw[very thick,opacity=.3,postaction={decorate},blue] ($(uU)!\r!(p1p1)$)  -- ($(uU)!\r!(w=U)$) ;
\draw[very thick,opacity=.3,postaction={decorate},blue] ($(uwi)!\r!(p1p1)$)  -- ($(uwi)!\r!(z=u)$) ;
\draw[very thick,opacity=.3,postaction={decorate},blue] ($(uwi)!\r!(p1p1)$)  -- ($(uwi)!\r!(w=wi)$) ;
\draw[very thick,opacity=.3,postaction={decorate},blue] ($(0U)!\r!(p1p1)$)  -- ($(0U)!\r!(w=U)$) ;
\draw[very thick,opacity=.3,postaction={decorate},blue] ($(0U)!\r!(p1p1)$)  -- ($(0U)!\r!(z=0)$) ;
}
\end{scope}
\end{tikzpicture}
\nn\ee
Again, in this picture, the squares at the top left, top right, lower left, and lower right are absent in the special cases $i=j$, $k=j$, $i=\ell$ and $k=\ell$ respectively.

Now we consider $\omega_{ij}^{\pm\mp}$. Here we must distinguish the cases $i\neq j$ and $i=j$. 

First suppose $i\neq j$. 
Consider $\omega_{ij}^{-+}$. It vanishes on the edges $\bigl(\pt, (w=w_i)\bigr)$, since $(w-w_i)^{-1} \CC[(w-w_i)^{-1}] \cap \CC[[w-w_i]] = 0$. Similarly, $\omega_{ij}^{+-}$ vanishes on the edges $\bigl( \pt, (z=z_j) \bigr)$. In either case we may define the action of the homotopy in the same way:
\be h(\omega_{ij}^{-+}) :=
\begin{tikzpicture}[baseline=0]
\tdplotsetmaincoords{0}{180}
\begin{scope}[xshift=-5cm,scale=1.301,tdplot_main_coords,local bounding box=G,decoration={
    markings,
    mark=at position 0.75 with {\arrow{stealth}},
}]
\fourfourpicijkl
\draw[ultra thick,dashed] (uU) -- (0U);
\draw[very thick,opacity=.3,postaction={decorate},blue] (p1p1) -- (0wi);
\draw[very thick,opacity=.3,postaction={decorate},blue] (uU) -- (p1p1);
\draw[very thick,opacity=.3,postaction={decorate},blue] (p1p1) -- (w=wi);
\draw[very thick,opacity=.3,postaction={decorate},blue] (p1p1) -- (z=0);
\draw[very thick,opacity=.3,postaction={decorate},blue] (uU) -- (z=u);
\draw[very thick,opacity=.3,postaction={decorate},blue] (z=u) -- (uwi);

\foreach \r in {0.25,0.5,0.75}
{
\draw[very thick,opacity=.3,postaction={decorate},blue] ($(z=u)!\r!(p1p1)$)  -- ($(uwi)!\r!(w=wi)$) ;
\draw[very thick,opacity=.3,postaction={decorate},blue] ($(w=U)!\r!(p1p1)$)  -- ($(0U)!\r!(z=0)$) ;
\draw[very thick,opacity=.3,postaction={decorate},blue] ($(p1p1)!\r!(0wi)$)  -- ($(z=0)!\r!(0wi)$) ;
\draw[very thick,opacity=.3,postaction={decorate},blue] ($(p1p1)!\r!(0wi)$)  -- ($(w=wi)!\r!(0wi)$) ;
\draw[very thick,opacity=.3,postaction={decorate},blue] ($(uU)!\r!(p1p1)$)  -- ($(z=u)!\r!(p1p1)$) ;
\draw[very thick,opacity=.3,postaction={decorate},blue] ($(uU)!\r!(p1p1)$)  -- ($(w=U)!\r!(p1p1)$) ;
}
\end{scope}
\end{tikzpicture}
 h(\omega_{ij}^{+-}) :=
\begin{tikzpicture}[baseline=0]
\tdplotsetmaincoords{0}{180}
\begin{scope}[xshift=-5cm,scale=1.301,tdplot_main_coords,local bounding box=G,decoration={
    markings,
    mark=at position 0.75 with {\arrow{stealth}},
}]
\fourfourpicijkl
\draw[ultra thick,dashed] (uU) -- (uwi);
\draw[very thick,opacity=.3,postaction={decorate},blue] (p1p1) -- (0wi);
\draw[very thick,opacity=.3,postaction={decorate},blue] (uU) -- (p1p1);
\draw[very thick,opacity=.3,postaction={decorate},blue] (p1p1) -- (w=wi);
\draw[very thick,opacity=.3,postaction={decorate},blue] (p1p1) -- (z=0);
\draw[very thick,opacity=.3,postaction={decorate},blue] (uU) -- (w=U);
\draw[very thick,opacity=.3,postaction={decorate},blue] (w=U) -- (0U);
\foreach \r in {0.25,0.5,0.75}
{
\draw[very thick,opacity=.3,postaction={decorate},blue] ($(z=u)!\r!(p1p1)$)  -- ($(uwi)!\r!(w=wi)$) ;
\draw[very thick,opacity=.3,postaction={decorate},blue] ($(w=U)!\r!(p1p1)$)  -- ($(0U)!\r!(z=0)$) ;
\draw[very thick,opacity=.3,postaction={decorate},blue] ($(p1p1)!\r!(0wi)$)  -- ($(z=0)!\r!(0wi)$) ;
\draw[very thick,opacity=.3,postaction={decorate},blue] ($(p1p1)!\r!(0wi)$)  -- ($(w=wi)!\r!(0wi)$) ;
\draw[very thick,opacity=.3,postaction={decorate},blue] ($(uU)!\r!(p1p1)$)  -- ($(z=u)!\r!(p1p1)$) ;
\draw[very thick,opacity=.3,postaction={decorate},blue] ($(uU)!\r!(p1p1)$)  -- ($(w=U)!\r!(p1p1)$) ;
}
\end{scope}
\end{tikzpicture}
\nn\ee

Next suppose $i=j$. Then the top left square in the pictures above is absent, but on the other hand we are guaranteed that the lower left and upper right squares are present, i.e. $i\neq \ell$ and $k\neq j$. We retract back to a boundary where the form vanishes as follows:
\be h(\omega_{ij}^{-+}) :=
\begin{tikzpicture}[baseline=0]
\tdplotsetmaincoords{0}{180}
\begin{scope}[xshift=-5cm,scale=1.3,tdplot_main_coords,local bounding box=G,decoration={
    markings,
    mark=at position 0.75 with {\arrow{stealth}},
}]
\threefourpicikl
\draw[ultra thick,dashed] (w=U) -- (0U);
\draw[very thick,opacity=.3,postaction={decorate},blue] (p1p1) -- (uwi);
\foreach \r in {0,0.25,0.5,0.75,1}
{
\draw[very thick,opacity=.3,postaction={decorate},blue] ($(uwi)!\r!(p1p1)$)  -- ($(uwi)!\r!(z=u)$) ;
\draw[very thick,opacity=.3,postaction={decorate},blue] ($(uwi)!\r!(p1p1)$)  -- ($(uwi)!\r!(w=wi)$) ;
\draw[very thick,opacity=.3,postaction={decorate},blue] ($(w=U)!\r!(0U)$)  -- ($(p1p1)!\r!(z=0)$) ;
\draw[very thick,opacity=.3,postaction={decorate},blue] ($(p1p1)!\r!(z=0)$)  -- ($(w=wi)!\r!(0wi)$) ;
}
\end{scope}
\end{tikzpicture}
h(\omega_{ij}^{+-}) :=
\begin{tikzpicture}[baseline=0]
\tdplotsetmaincoords{0}{180}
\begin{scope}[xshift=-5cm,scale=1.3,tdplot_main_coords,local bounding box=G,decoration={
    markings,
    mark=at position 0.75 with {\arrow{stealth}},
}]
\threefourpicikl
\draw[ultra thick,dashed] (z=u) -- (uwi);
\draw[very thick,opacity=.3,postaction={decorate},blue] (p1p1) -- (0U);
\foreach \r in {0,0.25,0.5,0.75,1}
{
\draw[very thick,opacity=.3,postaction={decorate},blue] ($(0U)!\r!(p1p1)$)  -- ($(0U)!\r!(w=U)$) ;
\draw[very thick,opacity=.3,postaction={decorate},blue] ($(0U)!\r!(p1p1)$)  -- ($(0U)!\r!(z=0)$) ;
\draw[very thick,opacity=.3,postaction={decorate},blue] ($(z=u)!\r!(uwi)$)  -- ($(p1p1)!\r!(w=wi)$) ;
\draw[very thick,opacity=.3,postaction={decorate},blue] ($(p1p1)!\r!(w=wi)$)  -- ($(z=0)!\r!(0wi)$) ;
}
\end{scope}
\end{tikzpicture}
\nn\ee
Finally, we set of course $h(\omega_{ij}^{--}) = 0$. 

With these definitions, we obtain
\be \omega_{ij} - I\circ P(\omega_{ij}) = [\dd,h](\omega_{ij}) \nn\ee
for each $i,j$, as required.
This completes the construction of the homotopy $h: \galocx \to \galocx$, and hence the proof of \cref{thm: global retract}.
\end{proof}

\begin{cor}\label{cor: global bimod}
There is a deformation retract of $(U(\gN),U(\galoc))$-bimodules
\be 
\begin{tikzcd} 
U(\gN) \ox U(\galoc) \rar[shift left]{U(I)}  & \lar[shift left]{U(P)} U(\galocx) \ar[loop right,"\wt h"] 
\end{tikzcd}\nn\ee
\end{cor}
\begin{proof}
This follows from \cref{thm: global retract} in the same way that \cref{cor: disc bimod} followed from \cref{prop: discx retract}.
\end{proof}

\appendix
\section{Proof of \cref{thm: flasque res}}\label{sec: flasque resolution proof}
We must show $\C^\bul(\sca'\AAA)$ is a flasque resolution of $\O$. As we noted before the statement of \cref{thm: flasque res},  $\C^\bul(\sca'\AAA)$ is certainly flasque. What remains is to check that the stalks of $\C^\bul(\sca'\AAA)$ resolve the stalks of $\O$. 

Let us consider the stalks at a point $p\in \Rect$. It is enough to show that, given any open $U\ni p$, there exists an open $V$ with $p\in V \subset U$ such that $\Gamma(V,\C^\bul(\sca'\AAA) = \C^\bul(\sca'\AAA_V)$ is a resolution of $\O(V)$. 

We may suppose $V$ is of the form
\be V = \Rect \setminus \bigcup_{i=1}^m \ol{\{(w=a_i)\}} \setminus \bigcup_{j=1}^m \ol{\{(z=b_j)\}}\nn\ee
i.e. that it has no isolated missing points, only missing lines. 

(Indeed, every open neighbourhood of $p$ is of the form $U=\Rect \setminus \bigcup_{i=1}^m \ol{\{(w=a'_i)\}} \setminus \bigcup_{j=1}^n \ol{\{(z=b'_j)\}} \setminus \bigcup_{k=1}^p\{ (c'_k,d'_k) \}$, as in \cref{Udef}. Suppose $p$ is a closed point $(a,b)\in \CC\times\CC$. Then given such a $U$ we may take $V = \Rect \setminus \bigcup_{i=1}^n \ol{\{(w=a'_i)\}} \setminus \bigcup_{j=1}^m \ol{\{(z=b'_j)\}} \setminus \bigcup_{\substack{k=1\\ c'_k\neq a}}^p \ol{\{ (w=c'_k) \}} \setminus \bigcup_{\substack{k=1\\ d'_k \neq b}}^p \ol{\{ (z=d'_k)\}}$, which is an open subset of $U$, containing the point $(a,b)$, and of the form we wanted. If instead $p$ is a line $(w=a)$ or $(z=b)$, or the generic point $\Eta$, the argument is similar.)

Let us regard the space 
\be \O(V) = \CC(w)_{a_1,\dots,a_m}\oxC\CC(z)_{b_1,\dots,b_n} \nn\ee
of sections of $\O$ over $V$ as a complex concentrated in degree zero:
\be 0 \to \O(V) \to 0 \nn\ee
It is enough to show that $\C^\bul(\sca'\AAA_V)$ is homotopy equivalent to this complex.

To do that, we shall use the fact that the associated complex  $\C^\bul(\sca'\AAA_V)$ and the Thom-Whitney complex $\Th^\bul(\sca'\AAA_V)$ are homotopy equivalent, as we recalled in \cref{sec: tw functor}. We shall show that $\O(V)$ is a deformation retract of $\Th^\bul(\sca'\AAA_V)$:
\be\begin{tikzcd} 
\bigl(0 \to \O(V) \to 0 \bigr) \rar[shift left]{i}  & \lar[shift left]{p} \Th^\bul(\sca'\AAA_V)  \ar[loop right,"h"] 
\end{tikzcd}\label{OVretract}\ee
The argument below is similar to that in the proofs of \cref{thm: he} and \cref{thm: global retract}. The only new ingredient conceptually is that the semisimplicial set $\Flag_\bul(V)$ now consists of \emph{uncountably many} copies of $\Flag_\bul(\DD)$, one for each closed point $(c,d)\in V$, appropriately sewn together along edges of the form $(line) \subset \Eta$. The restricted products $\sca'$ are needed to keep the sums finite in the definition of the homotopy $h$.

A cochain $\omega\in \Th^\bul(\sca'\AAA_V)$ is a polynomial differential form on $\Flag_\bul(V)$, valued in $\CC(w)\oxC \CC(z)$ and subject to certain boundary conditions: namely for each $(c,d)\in V$, $\AAA_V$ restricts on the embedded copy of $\Flag_\bul(\DDp{{c,d}})$ to the  $\Flag_\bul(\DDp{{c,d}})$-algebra
\be
\begin{tikzpicture}[baseline =100]
\tdplotsetmaincoords{0}{180}
\begin{scope}[xshift=0,scale=7.2,tdplot_main_coords,local bounding box=G]
\coordinate (w=U)  at (1,0,0)  ;  
\coordinate (z=u) at (0,-1,0)  ;  
\coordinate (uU) at (1,-1,0)  ;  
\coordinate (p1p1) at (0,0,0) ;   
\coordinate ({w=U }{p1p1}) at (barycentric cs:p1p1=1,{w=U}=1);
\coordinate ({z=u }{p1p1}) at (barycentric cs:p1p1=1,{z=u}=1);
\coordinate ({uU  }{p1p1}) at (barycentric cs:p1p1=1,{uU}=1);
\coordinate ({uU  }{z=u }{p1p1}) at (barycentric cs:p1p1=1,{z=u}=1.3,{uU}=1);
\coordinate ({uU  }{w=U }{p1p1}) at (barycentric cs:p1p1=1,{w=U}=1.3,{uU}=1);
\coordinate ({uU  }{z=u }) at (barycentric cs:{z=u}=1,{uU}=1);
\coordinate ({uU  }{w=U }) at (barycentric cs:{w=U}=1,{uU}=1);

\node (NuU)  at (uU) {$S_c^{-1}\CC[w]\oxC S_d^{-1}\CC[z]$};
\node (Nw=U)  at(w=U) {$S_c^{-1}\CC[w]\oxC\CC(z)$};
\node (Nz=u) at (z=u) {$\CC(w)\oxC S_d^{-1}\CC[z]$};
\node[rotate=0] (Np1p1) at (p1p1) {$\CC(w)\oxC\CC(z)$};
\node[rotate=-45] (N{uU  }{p1p1}) at ({uU  }{p1p1}) {$\CC(w)\oxC\CC(z)$};
\node (N{w=U }{p1p1}) at ({w=U }{p1p1}) {$\CC(w)\oxC\CC(z)$};
\node (N{z=u }{p1p1}) at ({z=u }{p1p1}) {$\CC(w)\oxC\CC(z)$};
\node (N{uU  }{w=U }) at ({uU  }{w=U }) {$S_c^{-1}\CC[w]\oxC \CC(z) $};
\node (N{uU  }{z=u }) at ({uU  }{z=u }) {$\CC(w)\oxC S_d^{-1}\CC[z] $};
\node (N{uU  }{w=U }{p1p1}) at ({uU  }{w=U }{p1p1}) {$\CC(w)\oxC\CC(z)$};
\node (N{uU  }{z=u }{p1p1}) at ({uU  }{z=u }{p1p1}) {$\CC(w)\oxC\CC(z)$};
\begin{scope}[every node/.style={midway,circle,inner sep = 0,draw,fill=white}]
\draw[->] (Np1p1) -- (N{z=u }{p1p1}); \draw[->] (Np1p1) -- (N{w=U }{p1p1}); 
\draw[->] (Nz=u) --(N{z=u }{p1p1}); \draw[->] (Nw=U) -- (N{w=U }{p1p1}); 
\draw[->] (Nz=u) --(N{uU  }{z=u }); \draw[->] (Nw=U) -- (N{uU  }{w=U }); 
\draw[->] (N{uU  }{w=U }) -- (N{uU  }{w=U }{p1p1});   
\draw[->] (N{uU  }{z=u }) -- (N{uU  }{z=u }{p1p1});   
\draw[->] (NuU) -- (N{uU  }{w=U });
\draw[->] (NuU) -- (N{uU  }{z=u });
\draw[->] (NuU) -- (N{uU  }{p1p1});
\draw[->] (Np1p1) -- (N{uU  }{p1p1});
\draw[->] (N{w=U }{p1p1}) -- (N{uU  }{w=U }{p1p1});   
\draw[->] (N{z=u }{p1p1}) -- (N{uU  }{z=u }{p1p1});   
\draw[->] (N{uU  }{p1p1}) -- (N{uU  }{z=u }{p1p1});   
\draw[->] (N{uU  }{p1p1}) -- (N{uU  }{w=U }{p1p1});   
\end{scope}
\end{scope}
\end{tikzpicture}
\label{AAdef}\ee
Every element of $\CC(w) \oxC \CC(z)$ has a partial fraction decomposition, in $w$ and then $z$, which a sum of finitely many terms. Thus we have the direct sum decomposition of vector spaces
\begin{align} \CC(w) \oxC \CC(z) &\cong_\CC \O(V) \oplus \bigoplus_{c\notin\{a_1,\dots,a_m\}} (w-c)^{-1} \CC[(w-c)^{-1}] \oxC \CC(z)\nn\\&\qquad\qquad 
\oplus \bigoplus_{d\notin\{b_1,\dots,b_n\}} \CC(w)_{a_1,\dots,a_m} \ox (z-d)^{-1} \CC[(z-d)^{-1}]\nn\end{align}
In this way, our polynomial differential form $\omega$ decomposes uniquely as a sum 
\be \omega = \omega_{\O(V)} + \sum_{c\notin\{a_1,\dots,a_m\}}\omega_c + \sum_{d\notin\{b_1,\dots,b_n\}} \omega'_d.\nn\ee
Here, it must be the case that after pulling $\omega$ back to any individual simplex $S$ of $\Flag_\bul(V)$ only finitely many summands are nonzero, so that $\omega|_S$ correctly takes values in $\CC(w) \oxC \CC(z)$. \emph{Which} summands contribute, though, can depend on which of the uncountably many simplices one considers, so it need not be true that only finitely many summands are nonzero full stop. 

We define the homotopy $h$ to act diagonally with respect to this decomposition, 
\be h(\omega) := h_{\O(V)}(\omega_{O(V)}) + \sum_{c\notin\{a_1,\dots,a_m\}} h_c(\omega_c) + \sum_{d\notin\{b_1,\dots,b_n\}} h'_d(\omega'_d),\nn\ee 
as follows. The summand $\omega_c$ must vanish when pulled back to the boundaries $(\pt,(w=c))$, and, in the same notation we used in the main text above, we define $h_c(\omega_c)$ to be a retract back to this boundary:
\be h_c(\omega_c) :=
\begin{tikzpicture}[baseline=0]
\tdplotsetmaincoords{0}{180}
\begin{scope}[xshift=-5cm,scale=1.3001,tdplot_main_coords,local bounding box=G,decoration={
    markings,
    mark=at position 0.75 with {\arrow{stealth}},
}]
\fourfourpiccdef
\draw[ultra thick,dashed] (w=U) -- (uU);
\draw[ultra thick,dashed] (w=U) -- (0U);
\foreach \r in {0.125,0.25,.375,0.5,.625,0.75,.875}
{
\draw[very thick,opacity=.3,postaction={decorate},blue] ($(uU)!\r!(0U)$)  -- ($(uwi)!\r!(0wi)$) ;
}
\end{scope}
\end{tikzpicture}
\nn\ee
for every $d,f\notin\{b_1,\dots,b_m\}$ and $e\notin\{a_1,\dots,a_m,c\}$.
What has to be checked is that for any simplex $S$, the resulting sum $ \sum_{c\notin\{a_1,\dots,a_m\}} h_c(\omega_c)|_S$ has only finitely many non-zero terms. This is what the conditions in the definition \cref{adelic products} of $\sca'$ ensure. For example, consider a flag $((w=e),\Eta)$. By our definition $h(\omega)|_{((w=e),\Eta)}$ receives contributions from \emph{all} flags $((w=c),\Eta)$ in this sum. But for all but finitely many such flags, the pull-back $\omega|_{((w=c),\Eta)}$ actually takes values not just in $\AAA((w=c),\Eta) = \CC(w)\oxC\CC(z)$ but in $\AA((w=c)) = S_c^{-1} \CC[w] \oxC \CC(z)$, which means that the pullback of the summand $\omega_c$ to this flag actually vanishes. Similarly, consider a flag $((e,d),(w=e), \Eta)$. By definition $h(\omega)|_{(e,d),(w=e),\Eta}$ receives a contribution all flags $((c,d),(w=c),\Eta)$ and $((c,d),(z=d),\Eta)$ for $c\notin\{b_1,\dots,b_m,e\}$, but at most finitely many of these contributions are actually nonzero.   

We define $h_d'(\omega_d')$ similarly to be given by retracting to the boundary of $\Flag_\bul(V)$ defined by $(z=d)$. 

We define $h_{\O(V)}$ to be the retraction to the vertex $\Eta$: for every $c,d\notin\{a_1,\dots,a_m\}$ and $e,f\notin\{b_1,\dot,b_n\}$, 
\be h_{\O(V)}(\omega_{\O(V)}) :=
\begin{tikzpicture}[baseline=0]
\tdplotsetmaincoords{0}{180}
\begin{scope}[xshift=-5cm,scale=1.301,tdplot_main_coords,local bounding box=G,decoration={
    markings,
    mark=at position 0.75 with {\arrow{stealth}},
}]
\fourfourpiccdef
\draw[very thick,opacity=.3,postaction={decorate},blue] (p1p1) -- (0wi);
\draw[very thick,opacity=.3,postaction={decorate},blue] (p1p1) -- (0U);
\draw[very thick,opacity=.3,postaction={decorate},blue] (p1p1) -- (uU);
\draw[very thick,opacity=.3,postaction={decorate},blue] (p1p1) -- (uwi);
\draw[very thick,opacity=.3,postaction={decorate},blue] (p1p1) -- (z=u);
\draw[very thick,opacity=.3,postaction={decorate},blue] (p1p1) -- (w=wi);
\draw[very thick,opacity=.3,postaction={decorate},blue] (p1p1) -- (z=0);
\draw[very thick,opacity=.3,postaction={decorate},blue] (p1p1) -- (w=U);
\foreach \r in {0.25,0.5,0.75}
{
\draw[very thick,opacity=.3,postaction={decorate},blue] ($(p1p1)!\r!(0wi)$)  -- ($(z=0)!\r!(0wi)$) ;
\draw[very thick,opacity=.3,postaction={decorate},blue] ($(p1p1)!\r!(0wi)$)  -- ($(w=wi)!\r!(0wi)$) ;
\draw[very thick,opacity=.3,postaction={decorate},blue] ($(uU)!\r!(p1p1)$)  -- ($(uU)!\r!(z=u)$) ;
\draw[very thick,opacity=.3,postaction={decorate},blue] ($(uU)!\r!(p1p1)$)  -- ($(uU)!\r!(w=U)$) ;
\draw[very thick,opacity=.3,postaction={decorate},blue] ($(uwi)!\r!(p1p1)$)  -- ($(uwi)!\r!(z=u)$) ;
\draw[very thick,opacity=.3,postaction={decorate},blue] ($(uwi)!\r!(p1p1)$)  -- ($(uwi)!\r!(w=wi)$) ;
\draw[very thick,opacity=.3,postaction={decorate},blue] ($(0U)!\r!(p1p1)$)  -- ($(0U)!\r!(w=U)$) ;
\draw[very thick,opacity=.3,postaction={decorate},blue] ($(0U)!\r!(p1p1)$)  -- ($(0U)!\r!(z=0)$) ;
}
\end{scope}
\end{tikzpicture}
\nn\ee
In this case there are no potentially infinite sums: for example, $h_{\O(V)}(\omega_{\O(V)})|_{((c,d),(w=c),\Eta)}$ receives contributions only from $\omega_{\O(V)}|_{((c,d),(w=c),\Eta)}$ and $\omega_{\O(V)}|_{((c,d),\Eta)}$. 

We let $i$ be the map embedding an element of $\O(V)$ as a constant $0$-form on $\Flag(V)$. (Observe that it obeys all the boundary conditions). We let $p$ be the map which picks out the component in $\O(V)$ of the pull-back of a form to the vertex $\Eta$:
\be p(\omega) := \omega_{\O(V)}|_{\Eta} \nn\ee

With these definitions, one checks that \cref{OVretract} is a deformation retract of dg vector spaces, as we wanted to show. This completes the proof of \cref{thm: flasque res}.

\printbibliography
\end{document}
